\documentclass[11pt]{article}

\PassOptionsToPackage{sort&compress}{natbib}
\usepackage{mystyle}
\usepackage{mathtools}
\usepackage{tabularx}
\usepackage{booktabs}
\usepackage{placeins}
\usepackage{comment}
\usepackage{authblk}

\usepackage[top=2.54cm,left=3cm,right=3cm,bottom=2.54cm]{geometry}

\title{Finite-sample nonparametric mean tests: \\
  Leave-one-out duality and asymptotic optimality}
\author[1]{Yifan Zhu}
\author[1,2]{John C. Duchi}
\affil{Departments of $^1$Electrical Engineering and $^2$Statistics, Stanford University\\\texttt{\{zhuyifan,jduchi\}@stanford.edu}}
\date{September 3, 2026}

\usepackage{macros}

\begin{document}

\maketitle

\begin{abstract}
  We study finite-sample valid tests of the one-sided mean hypothesis
  \(H_0:\mu\leq 1\) against \(H_1:\mu>1\) for nonnegative random variables.
  To do so, we develop a leave-one-out dual certificate framework,
  where certain
  pointwise inequalities imply \(p\)-value validity under the conditional
  mean null \(\Eb[X_i\mid \bm{X}_{-i}]\leq 1\),
  and which also gives conditions that allow combining dual
  certificates for \(p\)-values to show that their pointwise
  minimum is also a valid \(p\)-value.
  The framework proves finite-sample validity of
  \citeauthor{WangZh03}'s nonparametric
  likelihood-ratio statistic \(\nplr\), yields a new
  \(p\)-value $\pBinPlus$ extending the Clopper-Pearson binomial
  test to general nonnegative random variables,
  and shows that the pointwise minimum
  \(\min\{\nplr,\pBinPlus\}\) is itself a valid and more powerful
  \(p\)-value.
  We establish sharp optimality results for such testing problems in two
  regimes: both \(\nplr\) and \(\pBinPlus\) attain a
  universal detectability for the null $H_0$ without moment or tail
  assumptions, and \(\pBinPlus\) attains a nonparametric power
  lower bound under \(n^{-1/2}\)-local alternatives to $H_0$.
  Efficient algorithms and numerical experiments demonstrate substantial
  finite-sample power gains over existing valid methods.
\end{abstract}

\section{Introduction}

The mean is among the most basic functionals of a probability distribution, but
its distribution-free inference is impossible without some restriction on the
underlying distribution.
Indeed, the classical results of \citet{BahadurSa56} imply that, over the
class of all distributions with a finite mean, any uniformly
level-\(\alpha\) test of a one-sided hypothesis about the mean has power at
most \(\alpha\) at every alternative.
The non-robustness of the mean makes this almost expected: a small
probability mass placed far in the unfavorable direction significantly
changes the mean with negligible effect on any finite-sample statistic.

Imposing one-sided bounds on the support of the distribution allows
circumventing this impossibility.
For lower confidence bounds, this corresponds to a known lower endpoint,
which by shifting we may take to be zero.
Accordingly, let \(X_1,\ldots,X_n\) be i.i.d.\ observations from an unknown
distribution on \(\Rb_+\), with mean \(\mu=\Eb X_i\), and consider testing
\begin{align}
  \label{eq:intro-hypothesis}
  H_0:\mu\le 1
  \qquad\text{versus}\qquad
  H_1:\mu>1.
\end{align}
By inversion, tests of~\eqref{eq:intro-hypothesis} yield lower confidence
bounds for the mean.

Such data arise naturally in applications involving insurance losses, medical
expenditures, service and repair times, transaction values, revenue per
customer, and computational or energy costs.
For example, one may wish to certify that mean revenue exceeds a launch
threshold, or that mean expenditure does not exceed a prescribed budget.
In these settings, a credible upper bound on the observations, their variance,
or a higher moment may be unavailable, particularly when the distribution is
highly skewed or heavy-tailed.

Despite the apparent simplicity of \eqref{eq:intro-hypothesis}, available
tools for generating finite-sample valid \(p\)-values are limited.
Standard nonparametric tests~\citep{LehmannRo05}, such as permutation
and rank tests, do not directly apply, and normal and Studentized
approximations and bootstrap procedures provide
(at most) asymptotic validity.
Concentration-based methods commonly require known bounded support, variance
control, or stronger moment or tail assumptions
\citep{Hoeffding63,Bentkus04,MaurerPo09,Bernstein24,RomanoWo00,AusternMa22},
and consequently do not directly address observations known only to
be nonnegative.

\subsection{A leave-one-out dual framework}

We introduce a systematic method for generating and proving finite-sample
valid \(p\)-values.
Recall that a statistic \(T:\Rb_+^n\to[0,1]\) is a \(p\)-value under a null
class \(\mathcal P_0\) if for all \(\alpha\in(0,1)\),
\begin{equation*}
  \sup_{P\in\mathcal P_0} P\{T(\bm{X})\le\alpha\} \le \alpha.
\end{equation*}
Letting $\mc{P}(\mc{X})$ denote the class of probability measures on $\mc{X}$,
we typically consider the leave-one-out null
\begin{align}\label{eq:loo-null}
  \nullCond^n
  :=
  \left\{
    P \in \mathcal{P}(\Rb_+^n):
  \Eb_P[X_i\mid \bm{X}_{-i}]\le 1\ \text{a.s. for }i = 1, \ldots, n
  \right\}
\end{align}
for the test~\eqref{eq:intro-hypothesis},
where $\bm{X}_{-i}$ indicates $\bm{X}$ with its $i$th entry removed,
and
which includes the i.i.d.\ null as a special case.
Then $T$ is a $p$-value for the null~\eqref{eq:loo-null}
if and only if the value of the optimization problem
\begin{equation}
  \label{eqn:optimization-problem}
  \begin{array}{rl} \maximize & \Pb_{\Xb \sim P}
    \left\{T(\Xb) \le \alpha\right\} \\
    \subjectto & P \in \nullCond^n
  \end{array}
\end{equation}
is at most $\alpha$.

The Lagrangian dual of the optimization
problem~\eqref{eqn:optimization-problem} always upper bounds its optimal
value, so we proceed (asserting no strong duality) to develop this.
Introducing nonnegative Lagrange
multipliers $\lambda_\alpha^{(i)} : \Rb_+^{n-1} \to \Rb_+$,
we have the Lagrangian
\begin{equation*}
  \mathcal{L}(P, \lambda_\alpha^{(1)},\ldots,\lambda_\alpha^{(n)})
  :=
  \int_{\Rb_+^n}
  \left[
    \one\{T(\bm{x})\le \alpha\}
    -
    \sum_{i=1}^n
    \lambda_\alpha^{(i)}(\bm{x}_{-i})(x_i-1)
  \right]
  \,dP(\bm{x}).
\end{equation*}
For feasible \(P \in \nullCond^n\) and nonnegative
$\lambda_{\alpha}^{(i)}$, an immediate weak duality calculation gives
\begin{equation*}
  \Pb_P\{T(\bm{X})\le\alpha\}\le \mathcal{L}(P,\lambda_\alpha^{(1)},\ldots,\lambda_\alpha^{(n)}).
\end{equation*}
If the integrand $\one\{T(\bx) \le \alpha\} -
\sum_{i=1}^n \lambda_{\alpha}^{(i)}(\bx_{-i})(x_i - 1) \le \alpha$
for all $\bx \in \Rb_+^n$,
then evidently $\mc{L}(P, \lambda^{(1)}_\alpha, \ldots,
\lambda^{(n)}_\alpha) \le \alpha$ for all probability measures $P$.
Thus, to show that \( \Pb\{T(\bm{X})\le\alpha\}\le \alpha \) for all
\(P\in\nullCond^n\), it suffices to find nonnegative
\(\lambda_\alpha^{(i)}:\Rb_+^{n-1}\to\Rb_+\) such that
\begin{align}
  \tag{\textsc{cert}}
  \alpha +
  \sum_{i=1}^n \lambda_\alpha^{(i)}(\bm{x}_{-i})(x_i-1)
  \ge
  \one\{T(\bm{x})\le \alpha\}
  \qquad\text{for all }\bm{x}\in\Rb_+^n.
  \label{eq:intro-dual-certificate}
\end{align}
This reduces the optimization problem~\eqref{eqn:optimization-problem},
over all feasible distributions, to finding
functions that satisfy pointwise inequalities.
We give a rigorous derivation in \Cref{lem:dual-certificate}.

This dual viewpoint gives a way to certify that the minimum of
certain pairs of \( p \)-values remains a $p$-value.
In \Cref{thm:general-switch} in the sequel, we show how under certain
conditions, the dual certificates~\eqref{eq:intro-dual-certificate}
of two \(p\)-values $S$ and $T$ can combine
to demonstrate that $\min\{S, T\}$ is a valid $p$-value.

\subsection{Main results}

We apply this framework to prove the finite-sample validity of the one-sided
nonparametric likelihood-ratio statistic that \citet{WangZh03} study for the
mean of a nonnegative distribution.
Let $L(Q; \bx) = \prod_{i = 1}^n Q(\{x_i\})$ be the likelihood
$Q$ assigns to an independent sample $\bx$.
Then building on empirical likelihood~\citep{ThomasGr75,Owen88,Owen90},
\citet[Thm.~4.1]{WangZh03} derive the explicit representation
\begin{equation}
  \begin{split}
    \nplr(\bm{x}) & :=
    \frac{\sup\left\{L(Q;\bm{x}): Q\in
      \mc{P}(\Rb_+),\ \int z dQ(z)\le 1 \right\} }{
      \sup\left\{ L(Q;\bm{x}): Q\in\mc{P}(\Rb_+) \right\} } \\
    & \, = \min_{t \in [0, 1]} \prod_{i = 1}^n
    \left(1 + t (x_i - 1)\right)^{-1}
  \end{split}
  \label{eqn:nplr}
\end{equation}
for the nonparametric likelihood ratio,
and they conjecture that $\nplr$ is a finite-sample valid \(p\)-value
for i.i.d.\ data.
(Here, the numerator optimizes over all
distributions on \(\Rb_+\) with mean at most one, while the denominator
optimizes over all distributions on \(\Rb_+\).)
Throughout the paper, we frequently use the product
\begin{align*}
  \prodPoly_{\bm{x}}(t)
  :=
  \prod_{i=1}^n\{1-t+tx_i\}.
\end{align*}
The notation reflects its e-value interpretation: if \(X_1,\ldots,X_n\) are
independent e-values, then, for each
fixed \(t\in[0,1]\), \(\prodPoly_{\bm{X}}(t)\) is itself an e-value under the
global null.
The maximization over \(t\) in \(\nplr\) is data-adaptive, however,
and \(\max_{t\in[0,1]}\prodPoly_{\bm{X}}(t)\) is not an e-value.
In work concurrent to and independent of the
current paper, \citet{MingShWa26} verify the finite-sample validity of \(\nplr\)
using a different approach from the duality framework we develop.

In Section~\ref{sec:binplus},
we introduce a generalization $\pBinPlus$ of \citeauthor{ClopperPe34}'s
exact binomial $p$-value~[\citeyear{ClopperPe34}],
though we defer the precise definition
to Proposition~\ref{prop:binplus-ratio-monotonicity}.
A leave-one-out dual certificate~\eqref{eq:intro-dual-certificate} once
again proves its finite-sample validity.
The combination principle shows that the pointwise minimum
\(\min\{\nplr,\pBinPlus\}\) is itself a valid \(p\)-value.

Our second set of contributions, in Section~\ref{sec:asymptotics},
revolves around the asymptotic optimality of
these statistics as the sample size \(n\to\infty\).
For this purpose, instead of the leave-one-out null~\eqref{eq:loo-null},
we instead compare against tests valid under the
i.i.d.\ null
\begin{align}
  \label{eq:iid-null}
  \nullIid^n
  :=
  \left\{
  P^n :
  P \in \mc{P}(\Rb_+), ~ \Eb_P[X] \le 1
  \right\}.
\end{align}
This gives a stronger benchmark: because procedures need only be valid over
null~\eqref{eq:iid-null} rather than the larger class~\eqref{eq:loo-null},
there are more admissible competitors.
We establish sharp optimality results in two regimes
for $\Xb \sim Q_n^n$ for an unknown sequence of distribution $Q_n$.
First, without moment or tail assumptions, in
Section~\ref{sec:asymptotics-hellinger} we show how the squared Hellinger
distance $H^2_n$ of $Q_n$ to the null class $\nullIid$ determines
testability of the composite null~\eqref{eq:intro-hypothesis}.
Informally, Theorem~\ref{thm:np-hellinger-scale} shows that if $H_n^2 \ll 1/n$
as $n \to \infty$, then no level $\alpha$ test attains asymptotic power more
than $\alpha$, while if $H_n^2 \gg 1/n$, then there exist tests attaining
asymptotic power $1$.
Corollary~\ref{corollary:statistics-attain-hellinger} shows that the
statistics \(\nplr\), \(\pBinPlus\), and their pointwise minimum
\(\min\{\nplr,\pBinPlus\}\) attain this.
To investigate finer optimality properties, in
Section~\ref{sec:asymptotics-regular} we consider nonparametric local
alternatives approaching $\nullIid$ at a $1/\sqrt{n}$ rate~\citep[see,
  e.g.,][Chapter~25]{VanDerVaart98} rather than the gross Hellinger-distance
bounds.
In this case, more-or-less standard results establish that the optimal
asymptotic (nonparametric) power coincides with that in a Gaussian
sample-mean experiment.
These results highlight the necessity of the new statistic \(\pBinPlus\):
it, and hence \(\min\{\nplr,\pBinPlus\}\), attain this optimal power, while
\(\nplr\) does not.

Finally, we develop efficient algorithms for computing the
proposed statistics and evaluate their finite-sample power.
Because our optimality results use the i.i.d.\ null as the benchmark, the
experiments likewise compare against methods calibrated for i.i.d.\ sampling.
Across a range of distributions, $\pBinPlus$ and
\(\min\{\nplr,\pBinPlus\}\) substantially improve the power over existing
methods, including CDF-band methods
\citep{Anderson69,Massart90}, betting based methods \citep{Waudby-SmithRa24},
and the empirical Bernstein method \citet{Waudby-SmithRa24} discuss.

\subsection{Related work}

Distribution-free lower confidence bounds for nonnegative means have a
long history; see, for example,
\citet{Kaplan87,WangZh03,Gaffke05}.
The work most directly connected to ours is \citet{WangZh03}, which derives
the explicit representation~\eqref{eqn:nplr} of the nonparametric
likelihood-ratio statistic \(\nplr\), and \citet{Gaffke05}, who studied its
finite-sample validity.
The latter establishes special cases and asymptotic results but leaves
validity for arbitrary sample size unresolved.
Our proof resolves this question under the leave-one-out null
\eqref{eq:loo-null}.
Concurrent independent work by \citet*{MingShWa26} proves validity under the
leave-one-out null through a different representation.

The broader finite-sample literature on nonparametric mean inference is much
larger, but many standard tools rely on known bounded support or variance
control
\citep{Hoeffding63,Bentkus04,MaurerPo09,Bernstein24,RomanoWo00,AusternMa22}.
These results therefore do not directly address the
present setting, where we assume only nonnegative observations and allow
unbounded support.

A major exception is the betting and e-value approach
\citep{Waudby-SmithRa24}, which gives valid one-sided tests under
nonnegativity.
This method is valid under the sequential null
\(\Eb[X_i\mid \bm{X}_{1:i-1}]\le1\).
Recalling the classes~\eqref{eq:loo-null} and~\eqref{eq:iid-null},
we can also introduce the independent
and sequential classes
$\mc{P}_{\textup{ind}}^n = \{P_1 \times \cdots \times P_n
: \Eb_{P_i}[X] \le 1, P_i \in \mc{P}(\Rb_+)\}$
and $\mc{P}_{\textup{seq}}^n
= \{P \in \mc{P}(\Rb_+^n) : \Eb_P[X_i \mid \Xb_{1:i-1}] \le 1\}$.
Then the inclusions
\begin{align*}
  \nullIid^n \subset \mc{P}_{\textup{ind}}^n
  \subset \nullCond^n \subset
  \mc{P}_{\textup{seq}}^n
\end{align*}
show that \citeauthor{Waudby-SmithRa24}'s results apply to the larger
sequential model, while our validity results apply to the class
$\nullCond^n$, and the optimality guarantees to $\nullIid^n$.
This distinction is important, because the
multipliers~\eqref{eq:intro-dual-certificate} may use observations on both
sides of \(X_i\), in distinction from predictable betting strategies,
whose multipliers $\lambda^{(i)}$ are
$\Xb_{1:i-1}$-measurable.

Several related works study notions of optimality complementary to ours.
For bounded means, \citet{RomanoWo00} and \citet{AusternMa22} study
asymptotic efficiency through the $\sqrt{n}$-scaled width of confidence
intervals at fixed confidence level, while \citet{ShekharRa23} analyze
betting-based confidence sets using first-order asymptotic width and
nonasymptotic width lower bounds.
\citet{DeepBaJu25} characterize limiting confidence-interval
width in the joint large-sample, high-confidence regime
$N_\alpha\to\infty$ and $\alpha\downarrow0$, indexed by
$N_\alpha/\log(1/\alpha)$.
\citet{AgrawalRa25} study power-one sequential tests, with optimality measured
by expected stopping time in small-error and small-separation regimes.
In contrast, we keep the significance level fixed and study rejection power at
deterministic sample sizes along sequences of alternatives.
In Appendix~\ref{sec:existing-statistics}, we provide somewhat more
discussion of existing statistics.

\subsection{Roadmap}

The paper proceeds as follows.
\Cref{sec:dual-certificates} formalizes the leave-one-out dual certificate
framework.
\Cref{sec:nplr} applies the framework to prove the validity of the
nonparametric likelihood ratio \(\nplr\).
\Cref{sec:binplus} constructs the generalized binomial \(p\)-value
\(\pBinPlus\) and proves its validity.
\Cref{sec:dual-switching} develops a dual-switching principle and applies it
to certify \(\min\{\nplr,\pBinPlus\}\).
In \Cref{sec:asymptotics}, we establish the sharp local asymptotic power results
--- \Cref{sec:asymptotics-hellinger} handles the fully nonparametric regime, and
\Cref{sec:asymptotics-regular} handles the regular local regime.
We present experiment results in \Cref{sec:experiments}.
The appendix contains additional technical details, including proofs and
algorithms.

\section{Leave-one-out dual certificates}
\label{sec:dual-certificates}

We first formalize the leave-one-out duality framework the
introduction introduces, providing an integrability result
to guarantee finiteness.
We will always assume functions are Borel measurable, and let
\([n]:=\{1,\ldots,n\}\).

\begin{lemma}[Leave-one-out dual certificate]
  \label{lem:dual-certificate}
  Let \(\alpha\in(0,1)\) and \(T:\Rb_+^n\to[0,1]\) and
  $\lambda^{(i)}_\alpha : \Rb_+^{n-1} \to \Rb_+$, $i \in [n]$,
  be nonnegative and measurable.
  Assume $T$ and $\lambda^{(i)}_\alpha$
  satisfy~\eqref{eq:intro-dual-certificate}, i.e.,
  \begin{align}
    \dual^{\lambda}_{\alpha}(\bm{x}) := 
    \alpha+
    \sum_{i=1}^n \lambda_\alpha^{(i)}(\bm{x}_{-i})(x_i-1)
    \ge
    \one\{T(\bm{x})\le \alpha\}
    \qquad\text{for }\bm{x}\in\Rb_+^n.
    \label{eq:dual-certificate}
  \end{align}
  Then for all \(P\in\nullCond^n\),
  $\dual^\lambda_\alpha(\Xb)$ is integrable,
  $\Eb_P[\dual^{\lambda}_\alpha(\Xb)] \le \alpha$,
  \begin{align*}
    \Eb_P[\lambda_\alpha^{(i)}(\bm{X}_{-i})]\le \alpha  \text{ for } i\in[n], 
    ~~~ \mbox{and} ~~~
    \Pb_P\{T(\bm{X})\le\alpha\}\le\alpha.
  \end{align*}
\end{lemma}

If such a certificate exists for every \(\alpha\in(0,1)\), then \(T\) is
evidently a valid \(p\)-value for \(\nullCond\).
We call \(\{\lambda^{(i)}_\alpha\}_{i \in [n]}\) a leave-one-out dual
certificate for \(T\) at level \(\alpha\).
When \(T\) is symmetric, it suffices to find a single function
\(\lambda_\alpha\) certifying~\eqref{eq:intro-dual-certificate} holds,
taking \(\lambda_\alpha^{(i)}=\lambda_\alpha\) for every \(i\).
The only subtlety in the result is that, since we impose no integrability
assumption on the functions \(\lambda_\alpha^{(i)}\), we cannot immediately
take expectations in inequality \eqref{eq:dual-certificate}.

\begin{proof}
  We prove the statement by induction on \(n\).
  For \(n=1\), \(\lambda^{(1)}\) is a constant.
  Since \(\dual^{\lambda}_{\alpha}(0)=\alpha-\lambda^{(1)}\ge0\), we have
  \( \Eb_P[\lambda^{(1)}] \le \alpha\).
  Using \(\Eb_P[X_1]\le1\),
  \[
    \Eb_P[\dual^{\lambda}_{\alpha}(X_1)]
    =
    \alpha + \lambda^{(1)} \cdot (\Eb_P[X_1]-1)
    \le \alpha.
  \]

  Now assume the result holds in dimension \(n-1\). Fix \(i\in[n]\).
  For \(\bm{z}=(z_1,\ldots,z_{n-1})\in\Rb_+^{n-1}\), let
  \(\bm{z}^{(\oplus_i a)}\) denote the vector obtained by inserting \(a\) in
  position \(i\):
  \[
  \bm{z}^{(\oplus_i a)}
  :=
  (z_1,\ldots,z_{i-1},a,z_i,\ldots,z_{n-1}).
\]
  For \(\bm{y}=\bm{x}_{-i} \in \Rb^{n-1}\), define
  \[
    D_i(\bm{y})
    :=
    \alpha+\sum_{j\ne i}
    \lambda^{(j)}_\alpha\left((\bm{y}^{(\oplus_i0)})_{-j}\right)\left((\bm{y}^{(\oplus_i0)})_j-1\right).
  \]
  We claim that the induction hypothesis applies to \(D_i\).
  
  Since \(\dual^{\lambda}_{\alpha}(\bm{y}^{(\oplus_i0)})\ge0\), we have
  \[
    0 \le \lambda^{(i)}_{\alpha}(\bm{y}) \le \lambda^{(i)}_{\alpha}(\bm{y}) +
    \dual^{\lambda}_{\alpha}(\bm{y}^{(\oplus_i0)}) = D_i(\bm{y}).
  \]
  The marginal law of \(\bm{X}_{-i}\) belongs to \(\nullCond^{n-1}\) since
  for \(j\ne i\),
  \[
    \Eb[X_j\mid \bm{X}_{-\{i,j\}}]
    =
    \Eb\!\left[\Eb[X_j\mid \bm{X}_{-j}]\mid \bm{X}_{-\{i,j\}}\right]
    \le 1.
  \]
  Hence the induction hypothesis applies to \(D_i\),
  so \(\Eb_P[D_i(\bm{X}_{-i})]\le \alpha \).
  Since \(\lambda^{(i)}_\alpha(\bm{X}_{-i})\le D_i(\bm{X}_{-i})\), it follows that
  \(\Eb_P[\lambda^{(i)}_\alpha(\bm{X}_{-i})]\le \alpha \).
  Thus each \(\lambda^{(i)}_\alpha(\bm{X}_{-i})\) is integrable,
  so
  \[
    \Eb_P[\lambda^{(i)}_\alpha(\bm{X}_{-i}) X_i]
    =
    \Eb_P\!\left[\lambda^{(i)}_\alpha(\bm{X}_{-i})\Eb_P[X_i\mid \bm{X}_{-i}]\right]
    \le
    \Eb_P[\lambda^{(i)}_\alpha(\bm{X}_{-i})]
    <\infty.
  \]
  Hence \(\lambda^{(i)}_\alpha(\bm{X}_{-i})(X_i-1)\) is integrable and
  \[
    \Eb_P[\lambda^{(i)}_\alpha(\bm{X}_{-i})(X_i-1)]
    =
    \Eb_P\!\left[
      \lambda^{(i)}_\alpha(\bm{X}_{-i})\{\Eb_P[X_i\mid \bm{X}_{-i}]-1\}
    \right]
    \le0.
  \]
  Taking expectations for each \(i\) gives
  \[
    \Eb_P[\dual^{\lambda}_{\alpha}(\bm{X})]
    =
    \alpha+\sum_{i=1}^m \Eb_P[\lambda^{(i)}_\alpha(\bm{X}_{-i})(X_i-1)]
    \le \alpha.
  \]
  This proves the induction step and completes the proof.
\end{proof}

\section{The nonparametric likelihood ratio}
\label{sec:nplr}

We use the leave-one-out framework to prove the validity of the
nonparametric likelihood ratio statistic~\eqref{eqn:nplr}.
To develop the idea, recall the product introduced in the
introduction: for \(\bm{z}=(z_1,\ldots,z_m)\in\Rb_+^m\), write
\begin{align}
  \label{eq:def-prodPoly}
  \prodPoly_{\bm{z}}(t)
  :=
  \prod_{j=1}^m\bigl(1+t(z_j-1)\bigr),
\end{align}
with the empty-product convention \(\prodPoly_\varnothing\equiv1\).
Then the representation~\eqref{eqn:nplr} shows
\begin{align}
  \label{eq:def-nplr}
  \nplr(\bm{x})
  =
  \left(\max_{0\le t\le1}\prodPoly_{\bm{x}}(t)\right)^{-1}
  =
  \min_{0\le t\le1}
  \prod_{i=1}^n(1-t+t x_i)^{-1}.
\end{align}

\begin{theorem}[Validity of the nonparametric likelihood ratio]
  \label{thm:nplr-valid}
  The statistic \(\nplr\) is a valid \(p\)-value for \(\nullCond^n\).
\end{theorem}
\begin{proof}
  Fix \(\alpha\in(0,1)\).
  To apply \Cref{lem:dual-certificate}, it suffices to construct a
  nonnegative \(\lamNplr_\alpha:\Rb_+^{n-1}\to\Rb_+\) such that, for every
  \(\bm{x}\in\Rb_+^n\),
  \begin{align*}
    \dual^{\lamNplr}_\alpha(\bm{x})
    :=
    \alpha+
    \sum_{i=1}^n \lamNplr_\alpha(\bm{x}_{-i}) (x_i - 1)
    \ge
    \one\{\nplr(\bm{x})\le\alpha\}
    .
  \end{align*}
  We first construct an upper bound for the indicator using
  the representation~\eqref{eq:def-nplr}.

  For the product~\eqref{eq:def-prodPoly},
  define
  \begin{align}\label{eq:def-tau}
    \tau_\alpha(\bm{z}) := \sup
    \{ t \in [0,1] : \alpha \prodPoly_{\bm{z}}(s) < 1 \text{ for all } s \in [0,t] \}  
    .
  \end{align}
  When \(\nplr(\bm{x})\le\alpha\), there exists \(t\in[0,1]\) such
  that \( \alpha \prodPoly_{\bm{x}}(t)\ge1\).
  Continuity of $\prodPoly_{\bm{x}}(\cdot)$ implies that
  $\tau_\alpha(\bm{x})$ is the infimum of such \(t\) and \(\alpha
  \prodPoly_{\bm{x}}(\tau_\alpha(\bm{x}))=1\).
  When \(\nplr(\bm{x})>\alpha\),
  $\prodPoly_{\bx}(t) \alpha < 1$ for
  all $t \in [0, 1]$,
  so \(\tau_\alpha(\bm{x})=1\) while \(\alpha \prodPoly_{\bm{x}}(1) \ge 0\)
  always.
  Thus we obtain the upper bound
  \begin{align*}
    \one\{\nplr(\bm{x})\le\alpha\}
    \le
    \alpha\prodPoly_{\bm{x}}(\tau_\alpha(\bm{x})).
  \end{align*}

  To mimic the dual~\eqref{eq:intro-dual-certificate},
  we express \(\alpha\prodPoly_{\bm{x}}(\tau_\alpha(\bm{x}))\) in the
  form \( \alpha + \sum_{i=1}^n \lamNplr_\alpha(\bm{x}_{-i}) (x_i - 1)\).
  Noting the product form of \(\prodPoly_{\bm{x}}(t)\) in \cref{eq:def-prodPoly},
  a natural idea is to take derivatives with respect to \(t\) via the product
  rule, as
  \begin{align*}
    \frac{\partial}{\partial t}\prodPoly_{\bm{x}}(t)
    &=
    \sum_{i=1}^n (x_i - 1) \prodPoly_{\bm{x}_{-i}}(t).
  \end{align*}
  Integrating both sides from \( 0 \) to \( \tau_\alpha(\bm{x}) \) gives
  \[
  \prodPoly_{\bm{x}}(\tau_\alpha(\bm{x})) - \prodPoly_{\bm{x}}(0)
  = \sum_{i=1}^n (x_i - 1)
  \int_0^{\tau_{\alpha}(\bm{x})} \prodPoly_{\bm{x}_{-i}}(t) \, dt .
  \]
  Since \(\prodPoly_{\bm{x}}(0)=1\), multiplying by \(\alpha\) yields
  \begin{align*}
    \alpha \prodPoly_{\bm{x}}(\tau_\alpha(\bm{x}))
    &= \alpha + \sum_{i=1}^n \alpha (x_i - 1) \int_0^{\tau_\alpha(\bm{x})} \prodPoly_{\bm{x}_{-i}}(t)\,dt.
  \end{align*}
  This is almost in the desired form,
  but \(\alpha \int_0^{\tau_\alpha(\bm{x})} \prodPoly_{\bm{x}_{-i}}(t)\,dt\) depends on
  \(x_i\) through \(\tau_\alpha(\bm{x})\).
  To remove the dependence on $x_i$, for \(\bm{z}\in\Rb_+^{n-1}\) we define
  \begin{align}
    \lamNplr_\alpha(\bm{z})
    :=
    \alpha\int_0^{\tau_\alpha(\bm{z})}\prodPoly_{\bm{z}}(t)\,dt.
    \label{eq:def-lambda-nplr}
  \end{align}
  Every factor \((1-t) + t z_j\) defining
  $\prodPoly_{\bm{z}}(t)$ is nonnegative for $t \in [0, 1]$,
  whence \(\lamNplr_\alpha\ge0\).

  It suffices to show that for all \(i\),
  \begin{align}\label{eq:nplr-inequality}
    \alpha(x_i - 1)\int_0^{\tau_\alpha(\bm{x}_{-i})}\prodPoly_{\bm{x}_{-i}}(t)\,dt
    \ge
    \alpha(x_i - 1)\int_0^{\tau_\alpha(\bm{x})}\prodPoly_{\bm{x}_{-i}}(t)\,dt.
  \end{align}
  To see the inequality, note
  that
  \(\prodPoly_{\bm{x}}(t) = \bigl(1+t(x_i-1)\bigr)\prodPoly_{\bm{x}_{-i}}(t).\)
  If \(x_i\ge1\), then \(\prodPoly_{\bm{x}}(t)\ge\prodPoly_{\bm{x}_{-i}}(t)\) for
  all \(t\in[0,1]\), and therefore by definition~\eqref{eq:def-tau}, we have
  \[
  \tau_\alpha(\bm{x}_{-i}) \ge \tau_\alpha(\bm{x}),
  \]
  so inequality~\eqref{eq:nplr-inequality} follows.
  If \(x_i<1\), the inequalities reverse:
  $\prodPoly_{\bm{x}}(t) \le \prodPoly_{\bm{x}_{-i}}(t)$ and
  $\tau_\alpha(\bm{x}) \ge \tau_\alpha(\bm{x}_{-i})$.
  The rightmost integral in inequality~\eqref{eq:nplr-inequality} in this case
  is over a larger region, so multiplying by $(x_i - 1) < 0$
  gives the result.
  This finalizes the demonstration that \(\lamNplr_{\alpha}\) is a dual
  certificate~\eqref{eq:intro-dual-certificate} for \(\nplr\),
  and applying \Cref{lem:dual-certificate} completes the proof.
\end{proof}

\section{The generalized binomial p-value}
\label{sec:binplus}

Consider a thought experiment where $X_i \in \{0, b\}$ for some $b>1$.
In this case, the exact one-sided binomial $p$-value gives a \(p\)-value for the
null hypothesis that \(\Eb X_i\le 1\),
which \citet{ClopperPe34} implicitly develop via
one-sided confidence sets for the binomial.
To explore this, let $k$ be the number of observations in a
sample $\Xb$ equal to $b$.
For $q = \Pb(X_i = b)$, the null that $\Eb X = q b \le 1$ has
exact $p$-value \(\Pb\{\operatorname{Bin}(n, q)\ge k\}.\)
By the standard relationship between binomial tails and the incomplete beta
function \citep[Eq.~6.6.4]{AbramowitzSt65}, this tail probability is
\begin{align*}
  \Pb \{\operatorname{Bin}(n, q)\ge k\}
  =
  \Pb \{ \operatorname{Beta}(k, n-k+1) \le q \}
  = 
  \frac{\int_0^q t^{k-1}(1-t)^{n-k}\,dt}{\int_0^1 t^{k-1}(1-t)^{n-k}\,dt}.
\end{align*}
For the ``worst'' null, let $q = 1/b$ 
and consider the change of variables \( t = q + (1 - q) s = 1/b + (1 - 1/b) s\).
Then the integral ratio becomes
\begin{align*}
  \Pb \{\operatorname{Bin}(n,1/b)\ge k\}
  =
  \frac{\int_{-1/(b-1)}^0 (1 + s(b-1))^{k-1} (1 - s)^{n-k}\,ds}{\int_{-1/(b-1)}^1 (1 + s(b-1))^{k-1} (1 - s)^{n-k}\,ds}
  .
\end{align*}

We recognize that for the realization $\bx = (0^{n-k}, b^k)$ and any index
$j$ with $x_j = b$, then $\prodPoly_{\bx_{-j}}(s) = (1 + s(b - 1))^{k-1}(1 -
s)^{n-k}$.
For this $\bx$, we can express the exact \(p\)-value as
\[
  \Pb \{\operatorname{Bin}(n,1/b)\ge k\} = \frac{\int_{-1/(b-1)}^0
    \prodPoly_{\bm{x}_{-j}}(s)\,ds}{\int_{-1/(b-1)}^1
    \prodPoly_{\bm{x}_{-j}}(s)\,ds}.
\]
Proceeding heuristically, if we replace \( \bm{x} \) with an arbitrary
vector in \(\Rb_+^n\), set $j$ to be any index with \(x_j = x_{\max}
:= \max_{i \le n} x_i\), and
replace \(b\) in the integration endpoints with \(x_{\max}\), we obtain a
natural generalization of the binomial \(p\)-value: when \(x_{\max}>1\),
\begin{align}
  \label{eq:def-Pvee}
  \pBinPlus(\bm{x})
  :=
  \frac{\int_{-1/(x_{j}-1)}^0 \prod_{i \ne j} (1 + s(x_i - 1))\,ds}{\int_{-1/(x_j-1)}^1 \prod_{i \ne j} (1 + s(x_i - 1))\,ds}
  = \frac{\int_{-1/(x_{j}-1)}^0 \prodPoly_{\bm{x}_{-j}}(s)\,ds}{\int_{-1/(x_j-1)}^1 \prodPoly_{\bm{x}_{-j}}(s)\,ds}.
\end{align}
and when \(x_{\max}\le 1\), we set \(\pBinPlus(\bm{x})=1\).

The asymmetry of the form~\eqref{eq:def-Pvee} is artificial,
so we rewrite \(\pBinPlus\), which will also contribute to our analysis.
For \(\bm{x}\in\Rb_+^n\) and \(i\in[n]\), define
\begin{align*}
  \polyRightInt_i(\bm{x})
  &:=
  \int_0^1 \prodPoly_{\bm{x}_{-i}}(t) dt
  ~~ \mbox{and} ~~
  \polyLeftInt_i(\bm{x})
  :=
  \begin{cases}
    \displaystyle
    \int_{-(x_{\max}-1)^{-1}}^0
    \prodPoly_{\bm{x}_{-i}}(t) dt
    & \mbox{if}~ x_{\max}>1 \\
    +\infty,
    & \mbox{if}~ x_{\max}\le1,
  \end{cases}
\end{align*}
and let
\begin{align}
  \restatableEqDef{rhoDef}{eq:def-rho-i}{%
    \rho_i(\bm{x})
    :=
    \frac{\polyLeftInt_i(\bm{x})}
    {\polyLeftInt_i(\bm{x})+\polyRightInt_i(\bm{x})}
  }%
  .
\end{align}
The denominator in \(\rho_i(\bm{x})\) is positive, since all factors in the
integrals defining \(\polyLeftInt_i(\bm{x})\) and
\(\polyRightInt_i(\bm{x})\) are nonnegative
and \(\prodPoly_{\bm{x}_{-i}}(0)=1\).
We interpret $+\infty / +\infty$ as $1$, so $\rho_i(\bx) = 1$ if
$x_{\max} \le 1$.

\begin{proposition}[Representation and monotonicity of \(\pBinPlus\)]
  \label{prop:binplus-ratio-monotonicity}
  For \(\bm{x}\in\Rb_+^n\),
  \begin{align}
    \pBinPlus(\bm{x})=\max_i\rho_i(\bm{x}),
    \label{eq:binplus-max-ratio}
  \end{align}
  and \(\pBinPlus\) is elementwise nonincreasing: if \(\bm{x}\le\bm{y}\)
  elementwise, then \(\pBinPlus(\bm{x})\ge\pBinPlus(\bm{y})\).
\end{proposition}

\begin{proof}
  Throughout the proof,
  when \(\bm{x}\) is fixed, we omit the argument $\bx$, writing
  \(\polyLeftInt_i\), \(\polyRightInt_i\), and \(\rho_i\).
  We first prove that the \(\rho_i\) have the same ordering as the \(x_i\).
  This is immediate when \(x_{\max}\le1\), because then \(\rho_i=1\) for every
  \(i\).
  Now suppose that \(x_{\max}>1\).
  When $-\frac{1}{x_{\max} - 1} \le t \le 0$,
  we have $1 + t (x_j - 1) \ge 1 + t(x_{\max} - 1) \ge 0$
  for all $j$,
  so each term defining $\prodPoly_{\bm{x}_{-j}}(t)$ is nonnegative
  for $-\frac{1}{x_{\max} - 1} \le t$.
  Let $x_i\le x_j$.
  When \(t\le0\),
  \[
  1+t(x_j-1)
  \le
  1+t(x_i-1).
  \]
  The remaining (nonnegative)
  factors defining \(\prodPoly_{\bm{x}_{-i}}\) and
  \(\prodPoly_{\bm{x}_{-j}}\) agree, so
  \(\prodPoly_{\bm{x}_{-i}}(t)\le\prodPoly_{\bm{x}_{-j}}(t)\) when
  \(-\frac{1}{x_{\max} - 1} \le t \le 0\).
  Conversely, $\prodPoly_{\bx_{-i}}(t) \ge \prodPoly_{\bx{-j}}(t)$
  on \(0\le t\le1\),
  so
  \(0\le\polyLeftInt_i\le\polyLeftInt_j\) while
  \(0 \le \polyRightInt_j \le \polyRightInt_i\),
  and
  \begin{align}
    \rho_i
    =
    \frac{\polyLeftInt_i}{\polyLeftInt_i+\polyRightInt_i}
    \le
    \frac{\polyLeftInt_j}{\polyLeftInt_j+\polyRightInt_i}
    \le
    \frac{\polyLeftInt_j}{\polyLeftInt_j+\polyRightInt_j}
    =
    \rho_j.
    \label{eq:binplus-ratio-ordering}
  \end{align}
  Choosing \(j\) with \(x_j=x_{\max}\), the definition~\eqref{eq:def-Pvee}
  of $\pBinPlus$ 
  coincides with~\eqref{eq:binplus-max-ratio}.

We next prove coordinatewise monotonicity.
By increasing the coordinates one at a time, it suffices to consider
\(\bm{x}\) and \(\bm{y}\) that differ only in coordinate \(k\), with
\(x_k < y_k\).
If \(x_{\max}\le1\), the conclusion is immediate.
Hence we may assume that \(x_{\max}>1\).

Suppose first that \(y_k\le x_{\max}\).
Then there is an index \(j\ne k\) such that \(x_j=x_{\max}=y_{\max}\).
The integration endpoints defining \(\polyBothInt_j(\bx)\) and
$\polyBothInt_j(\by)$ are therefore identical.
Increasing \(x_k\) to \(y_k\) decreases the integrand for $t \le 0$ and
increases it for $t \ge 0$,
so
\(0\le\polyLeftInt_j(\bm y)\le\polyLeftInt_j(\bm x)\) and
\(\polyRightInt_j(\bm y)\ge\polyRightInt_j(\bm x)\), and
\[
  \pBinPlus(\bm{y}) = \rho_j(\bm{y})
  = \frac{\polyLeftInt_j(\bm{y})}
  {\polyLeftInt_j(\bm{y})+\polyRightInt_j(\bm{y})}
  \le
  \frac{\polyLeftInt_j(\bm{x})}
  {\polyLeftInt_j(\bm{x})+\polyRightInt_j(\bm{x})}
  = \rho_j(\bm{x}) = \pBinPlus(\bm{x}).
\]
It remains to consider the case that \(y_k>x_{\max}\), in which case \(k\)
is the new maximizing index.
Then \(\prodPoly_{\bm{x}_{-k}}=\prodPoly_{\bm{y}_{-k}}\),
so
\[
  \polyRightInt_k(\bm{y})
  = \int_0^1 \prodPoly_{\bm{y}_{-k}}(t)\,dt
  = \int_0^1 \prodPoly_{\bm{x}_{-k}}(t)\,dt
  = \polyRightInt_k(\bm{x})
  .
\]
Since \(y_{\max} = y_k > x_{\max}\), we
have \( - (y_{\max}-1)^{-1} > -(x_{\max}-1)^{-1}\), and
\[
  \polyLeftInt_k(\bm{y})
  = \int_{-(y_{\max}-1)^{-1}}^0 \prodPoly_{\bm{y}_{-k}}(t)\,dt
  \le
  \int_{-(x_{\max}-1)^{-1}}^0 \prodPoly_{\bm{x}_{-k}}(t)\,dt
  = \polyLeftInt_k(\bm{x}) .
\]
Consequently,
\[
  \pBinPlus(\bm{y})
  = \rho_k(\bm{y})
  = \frac{\polyLeftInt_k(\bm{y})}
  {\polyLeftInt_k(\bm{y})+\polyRightInt_k(\bm{y})}
  \le
  \frac{\polyLeftInt_k(\bm{x})}
  {\polyLeftInt_k(\bm{x})+\polyRightInt_k(\bm{x})}
  = \rho_k(\bm{x}) \le
  \pBinPlus(\bm{x}).
  \qedhere
  \]
\end{proof}

The leave-one-out duality framework of the introduction and
Section~\ref{sec:dual-certificates} applies with
the quantities $\polyBothInt_i$ and $\rho_i$, which
we can therefore use to establish the validity of the
generalized binomial statistic $\pBinPlus$ as a $p$-value for
the hypothesis test~\eqref{eq:intro-hypothesis}.

\begin{theorem}
  \label{thm:binplus-valid}
  The generalized binomial statistic \(\pBinPlus\) is a valid \(p\)-value
  for \(\nullCond^n\).
\end{theorem}

\begin{proof}
  Fix \(\alpha\in(0,1)\).
  By \Cref{lem:dual-certificate} and the dual
  certificate~\eqref{eq:intro-dual-certificate}, it suffices to find a
  nonnegative measurable function \(\lamBinPlus_\alpha:\Rb_+^{n-1}\to\Rb_+\)
  such that, for \(\bm{x}\in\Rb_+^n\),
  \begin{align*}
    \dual^{\lamBinPlus}_\alpha(\bm{x})
    :=
    \alpha+
    \sum_{i=1}^n \lamBinPlus_\alpha(\bm{x}_{-i}) (x_i - 1)
    \ge
    \one\{\pBinPlus(\bm{x})\le\alpha\}.
  \end{align*}
  For $\bm{z} \in \Rb_+^{n-1}$, define
  \begin{equation}
    \label{eq:def-lambda-binplus}
    \lamBinPlus_\alpha(\bm{z})
    :=
    \min\{(1-\alpha)\polyLeftInt_1(0, \bm{z}),
    \alpha\polyRightInt_1(0, \bm{z})\}.
  \end{equation}
  We claim that for all $\bx \in \Rb_+^n$,
  \begin{subequations}
    \begin{equation}
      \label{eq:binplus-common-endpoint-majorant}
      D_{\bm{x}}(\alpha)
      :=
      \alpha +
      \sum_{i=1}^n (x_i - 1)
      \min \{(1-\alpha)\polyLeftInt_i(\bm{x}),
      \alpha\polyRightInt_i(\bm{x})\}
      \ge \one\{\pBinPlus(\bm{x})\le\alpha\},
    \end{equation}
    and for \( i \in [n] \), we have the leave-one-out inequality
    \begin{equation}
      \label{eq:binplus-certificate-2}
      (x_i - 1) \lamBinPlus_\alpha(\bm{x}_{-i})
      \ge
      (x_i - 1)
      \min\{(1-\alpha)\polyLeftInt_i(\bm{x}),
      \alpha\polyRightInt_i(\bm{x})\}.
    \end{equation}
  \end{subequations}
  Combining inequalities
  \eqref{eq:binplus-common-endpoint-majorant}
  and \eqref{eq:binplus-certificate-2},
  $\dual^{\lamBinPlus}_\alpha$ provides
  the dual certification~\eqref{eq:intro-dual-certificate}.
  
  We first show inequality~\eqref{eq:binplus-certificate-2}.
  Note that
  \(\polyRightInt_1(0, \bm{x}_{-i})=\polyRightInt_i(\bm{x})\)
  by definition.
  The case \(x_{\max}\le1\) is trivial, since both
  \(\polyLeftInt_1(0, \bm{x}_{-i})\)
  and \(\polyLeftInt_i(\bm{x})\) are infinite, and
  \cref{eq:binplus-certificate-2} holds
  with equality.
  Hence it suffices to show that when \( x_{\max} > 1 \),
  \[
    (x_i - 1)\polyLeftInt_1(0, \bm{x}_{-i})
    \ge
    (x_i - 1)\polyLeftInt_i(\bm{x}),
  \]
  or equivalently,
  \begin{align*}
    (x_i-1) \int_{- ((\bm{x}_{-i})_{\max}-1)^{-1}}^0 \prodPoly_{\bm{x}_{-i}}(t)\,dt
    \ge
    (x_i-1) \int_{-(x_{\max}-1)^{-1}}^0 \prodPoly_{\bm{x}_{-i}}(t)\,dt.
  \end{align*}
  If \(x_i\le1\), then \( x_{\max} = (\bm{x}_{-i})_{\max} \), so equality holds.
  If \(x_i>1\), then \((\bm{x}_{-i})_{\max}\le x_{\max}\), so the left-hand side
  integrates over a larger interval than the right,
  implying inequality~\eqref{eq:binplus-certificate-2}.
  
  Now we show inequality~\eqref{eq:binplus-common-endpoint-majorant},
  considering the two cases that $\pBinPlus(\bx) \le \alpha$
  and $\pBinPlus(\bx) > \alpha$ separately.
  When \( \pBinPlus (\bm{x}) \le \alpha\), the
  definition~\eqref{eq:binplus-max-ratio} implies that \(\rho_i\le\alpha\)
  for all \(i\).
  Hence
  \((1-\alpha)\polyLeftInt_i(\bm{x})
  \le\alpha\polyRightInt_i(\bm{x})\), and
  \begin{equation*}
    D_{\bm{x}}(\alpha) = 
    \alpha +
    \sum_{i=1}^n
    (x_i - 1) (1 - \alpha) \polyLeftInt_i(\bm{x})
    .
  \end{equation*} 
  The product rule \( \prodPoly_{\bm{x}}'(t) = \sum_{i=1}^n (x_i - 1)
  \prodPoly_{\bm{x}_{-i}}(t) \) gives
  \begin{equation}
    \label{eq:binplus-left-normalization}
    \begin{split}
      \sum_{i=1}^n (x_i - 1) \polyLeftInt_i(\bm{x})
      & =
      \int_{-(x_{\max}-1)^{-1}}^0 \sum_{i=1}^n (x_i - 1) \prodPoly_{\bm{x}_{-i}}(t)\,dt
      \\
      & = \prodPoly_{\bm{x}}(0) -
      \prodPoly_{\bm{x}}\left(-\frac{1}{x_{\max}-1}\right) = 1 - 0 = 1,
    \end{split}
  \end{equation}
  so
  \[
    D_{\bm{x}}(\alpha) = \alpha + (1 - \alpha) \cdot 1 = 1 \ge \one\{\pBinPlus(\bm{x})\le\alpha\}.
  \]

  We consider the complementary case that \(\pBinPlus(\bm{x})>\alpha\).
  When \(x_{\max}\le1\), \(\polyLeftInt_i(\bm{x})=\infty\), so the left-hand side of
  inequality~\eqref{eq:binplus-common-endpoint-majorant} becomes
  \[
    \alpha
    +\alpha\sum_{i=1}^n(x_i-1)\polyRightInt_i(\bm{x})
    = \alpha+\alpha\int_0^1\prodPoly_{\bm{x}}'(t)\,dt =
    \alpha\prod_{i=1}^n x_i \ge0.
  \]
  When \(x_{\max}>1\), we consider the function \( D_{\bm{x}}( \cdot ) \)
  on $[0, 1]$.
  It is continuous, piecewise linear, and satisfies \(D_{\bm{x}}(0)=0\).
  We claim that \(D_{\bm{x}}(\cdot)\) is nondecreasing, which implies
  that \(D_{\bm{x}}(\alpha)\ge0\) and hence proves
  \cref{eq:binplus-common-endpoint-majorant}.
  To show this, we take the derivative, which exists at all but finitely many
  points:
  \begin{align*}
    D_{\bm{x}}'(u)
    &= 1 +
      \sum_{i=1}^n(x_i-1)
      \frac{\partial}{\partial u}
      \min\{(1-u)\polyLeftInt_i,u\polyRightInt_i\}
      .
  \end{align*}
  Letting \(I_u:=\{i \in [n] :\rho_i>u\}\), we have
  \begin{align*}
    D_{\bm{x}}'(u)
    =
      1-\sum_{i\notin I_u}(x_i-1)\polyLeftInt_i
      +\sum_{i\in I_u}(x_i-1)\polyRightInt_i
    =
      \sum_{i\in I_u}(x_i-1)
      (\polyLeftInt_i+\polyRightInt_i),
  \end{align*}
  where the last equality uses \cref{eq:binplus-left-normalization}.
  
  By symmetry, we may without loss of generality assume the ordering
  \(x_1\le\cdots\le x_n\).
  Then the ordering~\eqref{eq:binplus-ratio-ordering}
  implies
  \(\rho_1\le\cdots\le\rho_n\), so \(I_u\) is either empty
  or of the form \(\{k,\ldots,n\}\).
  If \(I_u\) is empty, then \(D_{\bm{x}}'(u)=0\).
  If \(I_u\) is nonempty and \(x_k\ge1\), then every term in the
  sum defining $D_{\bx}'(u)$ is nonnegative, so
  \(D_{\bm{x}}'(u) \ge 0\).
  If \(I_u\) is nonempty and \(x_k<1\), then terms with \(i<k\) are nonpositive,
  and
  \begin{multline*}
    D_{\bm{x}}'(u)
    =
    \sum_{i=1}^n(x_i-1)
    (\polyLeftInt_i+\polyRightInt_i)
    -
    \sum_{i<k}(x_i-1)
    (\polyLeftInt_i+\polyRightInt_i)
    \\
    \ge
    \sum_{i=1}^n(x_i-1)
    (\polyLeftInt_i+\polyRightInt_i)
    = \int_{-(x_{\max}-1)^{-1}}^1\prodPoly_{\bm{x}}'(t)\,dt
    = \prod_{i=1}^n x_i \ge 0.
  \end{multline*}
  Hence \(D_{\bm{x}}'(\cdot)\) is nonnegative, so \(D_{\bm{x}}(\cdot)\) is
  nondecreasing.
  This shows that \(D_{\bm{x}}(\alpha)\ge0\) when
  \(\pBinPlus(\bm{x})>\alpha\), and so
  inequality~\eqref{eq:binplus-common-endpoint-majorant} holds.
\end{proof}

\section{Certifying the minimum of two \(p\)-values via dual switching}
\label{sec:dual-switching}

In general, the pointwise minimum $\min\{S, T\}$ of two \(p\)-values $S$ and $T$
need not be a \(p\)-value.
Nonetheless, the dual formulation allows us to combine certificates for two
statistics into a certificate for their minimum under certain conditions.
We first provide some intuition.
In a dual expression
\[
\dual^{\lambda}_{\alpha}(\bx)
= \alpha + \sum_{i=1}^n (x_i-1)\lambda_{\alpha}^{(i)}(\bm{x}_{-i}),
\]
a larger multiplier $\lambda_\alpha^{(i)}$ makes satisfying
$\dual_{\alpha}^{\lambda} \ge \one\{S(\bm{x}) \le \alpha\}$ easier when
\(x_i>1\), whereas a smaller one is helpful when \(x_i<1\).
As \(\lambda^{(i)}(\bm{x}_{-i})\) cannot depend on \(x_i\),
we instead guess the size of $x_i$ from $\bm{x}_{-i}$ and the statistic $S$.

To develop this idea, for \(\bm{x}\in\Rb_+^n\), let \(\bm{x}^{(i\leftarrow
  a)}\) denote the vector obtained by replacing its \(i\)-th coordinate by
\(a\).
We focus on the boundary \(S(\bm{x})=\alpha\) where decisions are
``hardest.''
If $S$ is nonincreasing in each coordinate, $x_i \ge 1$ implies that
$S(\bm{x}^{(i\leftarrow 1)}) \ge S(\bm{x}) > \alpha$, and $x_{i} \le 1$
implies that $S(\bm{x}^{(i\leftarrow 1)}) \le S(\bm{x}) \le \alpha$.
We therefore expect that increasing the multiplier $\lambda_\alpha^{(i)}$
when $S(\bm{x}^{(i\leftarrow 1)}) > \alpha$ and decreasing it when
$S(\bm{x}^{(i\leftarrow 1)}) \le \alpha$ should
maintain the certificate~\eqref{eq:intro-dual-certificate}.

To be able to combine dual certificates,
we require a somewhat technical definition.
\begin{definition}
  \label{def:combinable}
  Let the statistics $S$ and $T$ have
  candidate dual multipliers
  $\lambda^{S,(i)}_{\alpha}$ and
  $\lambda^{T,(i)}_\alpha$, respectively.
  They \emph{admit combinable dual certificates} if
  \begin{enumerate}[label=(A\arabic*),leftmargin=*]
  \item \label{item:monotone}
    \(S(\bm{x})\) is nonincreasing in each coordinate \(x_i\);
  \item \label{item:comparison}
    \(\lambda^{S,(i)}_{\alpha}(\bm{z})
    \le \lambda^{T,(i)}_{\alpha}(\bm{z})\)
    for \(\bm{z}\in\Rb_+^{n-1}\) and \(i\in[n]\);
  \item \label{item:inner}
    \(\dual^S_\alpha(\bm{x})\ge1\) for \( \bm{x} \) such that $S(\bm{x}) \le \alpha$;
  \item \label{item:outer}
    \(\dual^T_\alpha(\bm{x})\ge\one\{T(\bm{x})\le\alpha\}\) for \(\bm{x}\in\Rb_+^n\).
  \end{enumerate}
\end{definition}

When $p$-values admit combinable duals, their pointwise minimum
remains a valid (and more powerful) $p$-value:
\begin{theorem}
  \label{thm:general-switch}
  Let \(\alpha\in(0,1)\) and
  \(S,T:\Rb_+^n\to[0,1]\)
  admit combinable duals $\lambda_\alpha^{S,(i)}$ and
  $\lambda_\alpha^{T,(i)}$ (Definition~\ref{def:combinable}).
  Then the switched multipliers
  \begin{align*}
    \lambda^{\star,(i)}_\alpha(\bm{z})
    :=
    \begin{cases}
      \lambda^{S,(i)}_\alpha(\bm{z})
        & \text{if } S\left(\bm{z}^{(\oplus_i 1)}\right)\le\alpha,\\
      \lambda^{T,(i)}_\alpha(\bm{z})
        & \text{if } S\left(\bm{z}^{(\oplus_i 1)}\right)>\alpha,
    \end{cases}
  \end{align*}
  provide a dual certificate~\eqref{eq:intro-dual-certificate}:
  for all \(\bm{x}\in\Rb_+^n\),
  \begin{align*}
    \dual^\star_\alpha(\bx) :=
    \alpha+ \sum_{i=1}^n
    (x_i-1)\lambda^{\star,(i)}_\alpha(\bm{x}_{-i}) \ge \one\{\min(S(\bm{x}),T(\bm{x}))\le\alpha\} .
  \end{align*}
\end{theorem}
\noindent
If the conditions~\ref{item:monotone}--\ref{item:outer} hold for every
\(\alpha\in(0,1)\), then \Cref{lem:dual-certificate} shows that
\(\min\{S,T\}\) is a valid \(p\)-value for the leave-one-out nulls
\(\nullCond^n\) in definition~\eqref{eq:loo-null}.
\begin{proof}
  First suppose that \(S(\bm{x})\le\alpha\).
  Fix $i \in [n]$.
  If \(x_i\le1\), then \ref{item:monotone} gives
  \[
  S\!\left(\bm{x}^{(i\leftarrow1)}\right)
  \le S(\bm{x})\le\alpha,
  \]
  so the switched multiplier \(\lambda^{\star,(i)}_\alpha(\bx_{-i}) =
  \lambda^{S,(i)}_\alpha(\bm{x}_{-i})\).
  If \(x_i>1\),
  condition~\ref{item:comparison} and that \(x_i-1>0\) give
  regardless of the value $\lambda^{\star,(i)}_\alpha$ that
  \begin{align*}
    (x_i-1)\lambda^{\star,(i)}_\alpha(\bm{x}_{-i})
    \ge
    (x_i-1)\lambda^{S,(i)}_\alpha(\bm{x}_{-i}).
  \end{align*}
  So
  \(\dual^\star_\alpha(\bm{x})\ge \dual^S_\alpha(\bm{x})\)
  if $S(\bx) \le \alpha$, and
  condition \ref{item:inner}
  gives $\dual^\star_\alpha(\bx) \ge 1$ on this event.

  Now suppose that \(S(\bm{x})>\alpha\).
  Fix $i \in [n]$.
  If \(x_i>1\), then \ref{item:monotone} gives
  \[
  \alpha<S(\bm{x})
  \le S\!\left(\bm{x}^{(i\leftarrow1)}\right),
  \]
  so \(\lambda^{\star,(i)}_\alpha(\bx_{-i}) = \lambda^{T,(i)}_\alpha(\bm{x}_{-i})\).
  If \(x_i\le 1\),
  condition \ref{item:comparison} and that \(x_i-1\le 0\) give
  once again that no matter the value of $\lambda^{\star,(i)}_\alpha$,
  \begin{align*}
    (x_i-1)\lambda^{\star,(i)}_\alpha(\bm{x}_{-i})
    \ge
    (x_i-1)\lambda^{T,(i)}_\alpha(\bm{x}_{-i}).
  \end{align*}
  Therefore
  \(
  \dual^\star_\alpha(\bm{x})
  \ge \dual^T_\alpha(\bm{x})
  \ge \one\{T(\bm{x})\le\alpha\}
  \)
  when $S(\bx) > \alpha$,
  by condition~\ref{item:outer}.
\end{proof}

\subsection{\(\min \{\nplr,\pBinPlus\}\) is a valid \(p\)-value}

We may apply \Cref{thm:general-switch} to show
that \(\min\{\nplr,\pBinPlus\}\) is a valid \(p\)-value for \(\nullCond^n\).
This pointwise minimum yields a more powerful $p$-value,
as neither statistic is pointwise smaller than the other.
The \(p\)-value \(\nplr\) is often more powerful when the evidence is spread
across several large observations.  For example,
\[
  \nplr(2, 2, 2, 2, 3, 3, 3, 3) \approx 0.00077 < 0.00091 \approx \pBinPlus(2, 2, 2, 2, 3, 3, 3, 3).
\]
On the other hand, \(\pBinPlus\) can be more powerful when the evidence is
concentrated in a few large observations:
\[
  \nplr(0, 0, 0, 0, 0, 0, 5, 5) \approx 0.94 > 0.50 \approx \pBinPlus(0, 0, 0, 0, 0, 0, 5,
  5).
\]
Thus the two statistics capture different
types of evidence against the null.

\begin{theorem}
  \label{thm:min-nplr-binplus-valid}
  The statistic \(\min\{\nplr,\pBinPlus\}\) is a valid \(p\)-value for
  \(\nullCond^n\).
\end{theorem}
\begin{proof}
  Fix \(\alpha\in(0,1)\).
  Recall the definitions~\eqref{eq:def-lambda-nplr}
  and~\eqref{eq:def-lambda-binplus} of \(\lamBinPlus_\alpha\) and
  \(\lamNplr_\alpha\), respectively.
  We apply \Cref{thm:general-switch} with
  \[
  S=\pBinPlus,
  \qquad
  T=\nplr,
  \qquad
  \lambda^S_\alpha=\lamBinPlus_\alpha,
  \qquad
  \lambda^T_\alpha=\lamNplr_\alpha.
  \]
  \Cref{prop:binplus-ratio-monotonicity} gives the
  monotonicity~\ref{item:monotone} of \(\pBinPlus\), while the
  multipliers~\eqref{eq:def-lambda-nplr} and~\eqref{eq:def-lambda-binplus}
  provide the dual certificate~\eqref{eq:intro-dual-certificate},
  guaranteeing conditions~\ref{item:inner} and \ref{item:outer},
  respectively.
  It remains to verify condition~\ref{item:comparison} of
  Definition~\ref{def:combinable}, that 
  \(\lamBinPlus_\alpha(\bm{z}) \le \lamNplr_\alpha(\bm{z}) \)
  for \( \bm{z} \in \Rb_+^m \).
  Let \(\bm{z}\in\Rb_+^m\) and recall the definition~\eqref{eq:def-tau} of
  \(\tau_\alpha(\bm{z})\).
  Since \(\alpha\prodPoly_{\bm{z}}(0)=\alpha<1\) and \(\prodPoly_{\bm{z}}\)
  is continuous, we must have \(\tau_\alpha(\bm{z})>0\).

  If \(\tau_\alpha(\bm{z})=1\), then
  definitions~\eqref{eq:def-lambda-binplus}
  and~\eqref{eq:def-lambda-nplr} of the dual multipliers immediately give
  \[
  \lamBinPlus_\alpha(\bm{z})
  \le
  \alpha\int_0^1\prodPoly_{\bm{z}}(t)\,dt
  =
  \lamNplr_\alpha(\bm{z}).
  \]
  Now consider the case where \(0<\tau_\alpha(\bm{z})<1\).
  In this case \(\prodPoly_{\bm{z}}\bigl(\tau_\alpha(\bm{z})\bigr)=\frac1\alpha\)
  and \(z_{\max} > 1\).
  On the interval \(\left(-(z_{\max}-1)^{-1},1\right)\), \(\prodPoly_{\bm{z}}\) is a
  product of positive linear factors, so \(\log\prodPoly_{\bm{z}}(t)\) is concave.
  Define \( h(t) = t \cdot \frac{\log (1/\alpha)}{\tau_{\alpha}(\bm{z})} \),
  which satisfies \( h(0) = 0 = \log \prodPoly_{\bm{z}}(0) \) and \(
  h(\tau_{\alpha}(\bm{z})) = \log \frac{1}{\alpha} = \log
  \prodPoly_{\bm{z}}(\tau_{\alpha}(\bm{z})) \), so
  that $h(t)$ is the secant line
  of $\log \prodPoly_{\bm{z}}(t)$ from $t=0$ to $t=\tau_{\alpha}(\bm{z})$.
  Concavity of \( \log\prodPoly_{\bm{z}}(t) \) implies
  that \(\log\prodPoly_{\bm{z}}(t) \ge h(t)\) for
  \(t \in [0,\tau_\alpha(\bm{z})]\), and \(\log\prodPoly_{\bm{z}}(t) \le h(t)\)
  outside this region.
  Exponentiating and integrating these two bounds yields
  \begin{align*}
    \int_{-(z_{\max}-1)^{-1}}^0\prodPoly_{\bm{z}}(t)\,dt
    & \le
    \int_{-(z_{\max}-1)^{-1}}^0 \exp\left(t \cdot
    \frac{\log(1/\alpha)}{\tau_{\alpha}(\bm{z})}\right)\, dt
    \\
    & =
    \frac{\tau_\alpha(\bm{z})}{\log(1/\alpha)} \left( 1 - \exp\left(-\frac{\log(1/\alpha)}{\tau_{\alpha}(\bm{z})(z_{\max}-1)}\right) \right)
    \le
    \frac{\tau_\alpha(\bm{z})}{\log(1/\alpha)}
  \end{align*}
  and
  \begin{align*}
    \int_0^{\tau_\alpha(\bm{z})}\prodPoly_{\bm{z}}(t)\,dt
    & \ge
    \int_0^{\tau_\alpha(\bm{z})} \exp\left(t \cdot
    \frac{\log(1/\alpha)}{\tau_{\alpha}(\bm{z})}\right)\,dt
    \\
    & =
    \frac{\tau_\alpha(\bm{z})}{\log(1/\alpha)}
    \left(\exp\left(\log(1/\alpha)\right)-1\right)
    =
    \frac{\tau_\alpha(\bm{z})(1/\alpha-1)}{\log(1/\alpha)}.
  \end{align*}
  Hence
  \begin{align*}
    \lamBinPlus_\alpha(\bm{z})
    & \le
    (1-\alpha)
    \int_{-1/(z_{\max}-1)}^0\prodPoly_{\bm{z}}(t)\,dt \\
    & \le
    \frac{(1-\alpha)\tau_\alpha(\bm{z})}{\log(1/\alpha)}                 
    =
    \frac{\alpha\tau_\alpha(\bm{z})(1/\alpha-1)}
         {\log(1/\alpha)}                                           
         \le
         \alpha\int_0^{\tau_\alpha(\bm{z})}\prodPoly_{\bm{z}}(t)\,dt
         =\lamNplr_\alpha(\bm{z}).
  \end{align*}
  Thus \(\lamBinPlus_\alpha(\bm{z})\le\lamNplr_\alpha(\bm{z})\) for
  \(\bm{z}\in\Rb_+^m\), verifying
  condition~\ref{item:comparison}.
\end{proof}

\section{Asymptotic optimality}
\label{sec:asymptotics}

This section studies the asymptotic power achievable when testing the
one-sided composite null
$\Eb X \le 1$
against sequences of point alternatives.
Once we develop fundamental limits, we show when the
$p$-values from the preceding sections achieve, and fail to achieve, 
optimal behavior.
For each \(n\), we observe i.i.d.\ data with common law \(P\) and consider
\begin{equation}
  \label{eqn:composite-iid-null}
  H_0 : \Eb_P X \le 1
  \qquad\text{versus}\qquad
  H_{1,n}:P=Q_n,\quad \Eb_{Q_n}X>1.
\end{equation}
Under \(H_0\), the sample law belongs to the i.i.d.\ null~\eqref{eq:iid-null},
\begin{equation*}
  \nullIid^n := \left\{P^n : P \in \mc{P}(\Rb_+),
  ~ \Eb_P[X] \le 1\right\} 
\end{equation*}
This difference from the leave-one-out null \( \nullCond^n \) used for the
finite-sample validity results serves two purposes.
First, it makes the optimality results stronger as comparisons: since competing
procedures need to maintain validity only over a smaller null class, more
statistics are admissible competitors.
Second, the i.i.d.\ formulation yields a cleaner asymptotic theory.

Uniform level-$\alpha$ tests over \(\nullIid^n\) are functions
\(\varphi_n:\Rb_+^n\to[0,1]\) such that
\begin{equation*}
  \sup_{P \in\nullIid}\Eb_{P^n}[\varphi_n] \le \alpha
\end{equation*}
for all $n$.
We define the asymptotic power
against point alternatives $Q_n$ by
\(\liminf_{n \to \infty}\Eb_{Q_n^n}[\varphi_n]\),
which is often a proper limit, and our impossibility results
will demonstrate upper bounds on the
best possible power $\limsup_n \Eb_{Q_n^n}[\varphi_n]$.

We study optimality in two regimes.
\Cref{sec:asymptotics-hellinger} treats fully nonparametric alternatives, with
no moment or tail assumptions, and shows that the squared Hellinger distance to
the null \(H^2(Q_n, \nullIid)\) determines the zero–one detectability boundary.
If $n H^2(Q_n,\nullIid)\to0$, no uniformly valid test has asymptotic power above
its level; if $nH^2(Q_n, \nullIid)\to\infty$, there exist uniformly valid tests with
power tending to one.
The impossibility result is standard.
The achievability result, however, exploits the structure of \(\nullIid\) to
compute the Hellinger projection of \(Q_n\) onto \(\nullIid\) explicitly and
construct from it a uniformly valid oracle test attaining asymptotic
power one whenever $n H^2(Q_n, \nullIid) \to \infty$.
More meaningfully, we also show that \(\nplr\), \(\pBinPlus\), and their
pointwise minimum \(\min\{\nplr,\pBinPlus\}\) attain asymptotic power one in
this case without knowledge of the point alternatives $Q_n$.

By contrast, \Cref{sec:asymptotics-regular} adopts classical local
nonparametrics~\citep[cf.][Ch.~25]{VanDerVaart98} to
provide a finer-grained asymptotic optimality theory.
By considering regular alternatives $Q_n$ approaching $\nullIid$ at rate
$1/\sqrt{n}$, we may compare the limiting power of \(\nplr\), \(\pBinPlus\),
and \(\min\{\nplr,\pBinPlus\}\) to the optimal power in the associated
limiting Gaussian experiment.
We show that \(\pBinPlus\) and \(\min\{\nplr,\pBinPlus\}\) attain this
power, while $\nplr$ does not.

\paragraph{Notation}
We adopt standard asymptotic convergence notation.
For positive \(a_\delta\) and \(b_\delta\),
we use
\( a_\delta \asymp b_\delta \) to mean that there exist constants
\(0<c_1\le c_2<\infty\) such that, for sufficiently
small \(\delta > 0\), \( c_1 b_\delta \le a_\delta \le c_2 b_\delta \).
Analogous definitions apply to sequences $(a_n),(b_n)$ as \(n\to\infty\).
We use \(\convAS\), \(\convP\), and \(\convD\) to denote convergence almost
surely, in probability, and in distribution, respectively.
For random variables
\(Z_n\) and \(A_n\), \(Z_n=o_p(A_n)\)
means $Z_n/A_n \cp 0$, while \(Z_n=O_p(A_n)\) means
\begin{equation*}
  \lim_{M \to \infty}
  \limsup_{n\to\infty} \Pb\{|Z_n| > Ma_n\} = 0.
\end{equation*}

\subsection{Hellinger distance to the null and zero-one detectability}
\label{sec:asymptotics-hellinger}

Establishing the zero-one boundary requires two complementary results.
Firstly, we characterize the distance of
an alternative \(Q\) to the null class
\begin{align*}
  \nullbasic
  = \{P \in \mc{P}(\Rb_+) : \Eb_P X \le 1\}.
\end{align*}
Secondly, we will show that separation from $\nullbasic$ implies the existence
of a test with non-trivial power that remains valid uniformly over the
entire composite null.
The squared Hellinger distance between the alternative $Q$ and
the nulls determines both of these.
We therefore develop the (Hellinger) projection of a distribution
$Q$ onto $\nullbasic$, which will show allows us to characterize
precisely when optimal tests exist in the composite
null~\eqref{eqn:composite-iid-null}.

\subsubsection{The Hellinger projection}

For probability measures \(Q\) and \(P\) on \(\Rb_+\), define their
squared Hellinger distance
\begin{align*}
  H^2(Q,P)
  &:=
    \frac12
    \int
    \left(\sqrt{dQ} - \sqrt{dP} \right)^2
    =
    1- \int \sqrt{dP dQ}
    .
\end{align*}
For a law \(Q\) and class $\mc{P}$, define the squared distance
\begin{equation}
      \label{eq:hellinger-distance-to-null}
  H^2(Q, \mc{P})
  :=
  \inf_{P\in \mc{P}} H^2(Q,P).
\end{equation}
For $\phi(u) = \frac{1}{2} (\sqrt{u} - 1)^2$, the
optimization~\eqref{eq:hellinger-distance-to-null} corresponds to projecting
$Q$ onto $\mc{P}$, measuring error via the
\(\phi\)-divergence~\citep{AliSi66,Csiszar67}.
By developing the minimizers in the
problem~\eqref{eq:hellinger-distance-to-null},
we will be able to characterize the precise consequences for the existence
of optimal tests in the problem~\eqref{eq:intro-hypothesis}.
While such (infinite-dimensional) projection problems enjoy a substantial
duality theory~\citep{BorweinLe91,CsiszarMa12},
the generality of the set $\nullbasic$ (a one-sided moment constraint
with unbounded support) appears to require some care in characterizing
the infimum~\eqref{eq:hellinger-distance-to-null}.

To develop the Hellinger projection~\eqref{eq:hellinger-distance-to-null}
of $Q$ onto $\nullbasic$, suppose heuristically
that \(P\) and \(Q\) have densities \(p\) and \(q\) with respect to a common
measure \(\nu\).
This yields the problem
\begin{equation*}
  \begin{array}{rl}
    \maximize & \int\sqrt{pq}\,d\nu \\
    \subjectto & \int p\,d\nu=1,\quad
    \int(x-1)p(x)\,d\nu(x) \le 0,
    ~~~ p \ge 0,
  \end{array}
\end{equation*}
which, incorporating the nonnegativity $p \ge 0$ directly,
has Lagrangian
\begin{equation*}
  \mathcal{L}(p,\lambda,\eta)
  = \lambda+\int\left[ \sqrt{p(x)q(x)}-\{\lambda+\eta(x-1)\}p(x) \right]d\nu(x).
\end{equation*}
Treating $[p(x)]_{x \in \Rb_+}$ as a variable and maximizing over
$p(x) \ge 0$,
pointwise stationarity gives \( \tfrac12\sqrt{q(x)/p(x)}=\lambda+\eta(x-1)\).
After reparameterizing, \({p(x)}/{q(x)} \propto{1}/(1+t(x-1))^{2}\), which is also the
interior form obtained by specializing \citet[Corollary~3.8]{BroniatowskiKe06}
to $g(x)=x-1$ and $\phi(u)=\tfrac12(\sqrt{u}-1)^2$.  Consequently,
\begin{equation*}
    1 - H^2(Q,P) = \int\sqrt{pq}\,d\nu \propto \int \frac{q(x)}{1+t(x-1)}\,d\nu(x) .
\end{equation*}

Motivated by this heuristic that the projection
$P$ of $Q$ onto $\nullbasic$ should have form
$dP(x) \propto dQ(x) / (1 + t(x - 1))^2$, we define
the scalar affinity function
\begin{equation*}
  A_Q(t) := \Eb_Q\left[\frac{1}{1 + t(X-1)}\right],
\end{equation*}
where we take $1/0=+\infty$.
For \(x\ge0\), the extended-valued function
\(t\mapsto (1+t(x-1))^{-1}\) is closed-convex on \([0,1]\).
Fatou's lemma therefore implies that \(A_Q\) is closed convex.
Since \(A_Q(0)=1\) and \([0,1]\) is compact, \(A_Q\) attains its minimum
on \([0,1]\).
Letting $\pointmass_x$ indicate a point mass at $x$,
the affinity functional allows us to characterize
the projection of $Q$ onto $\nullbasic$:

\begin{lemma}[Hellinger projection]
  \label{lem:hellinger-projection}
  Let \(Q\) be a probability on \(\Rb_+\),
  \(t_Q\in\operatorname*{argmin}_{0\le t\le1}A_Q(t)\), and
  $A_Q^\star:=A_Q(t_Q)=\min_{0\le t\le1}A_Q(t)$.
  Then
  \begin{align}
    H^2(Q,\nullbasic)
    = \inf_{P\in\nullbasic}H^2(Q,P)
    = 1-\sqrt{A_Q^\star}.
    \label{eq:hellinger-radius-dual}
  \end{align}
  Furthermore, there
  exists \( P_Q^\star\in\nullbasic\) attaining this infimum, and
  \begin{align}
    P_Q^\star(dx)
    =
    \left\{
    \begin{aligned}
      &Q(dx)
      && \text{if } t_Q=0,
      \\
      &\dfrac{1}{A_Q^\star\{1+t_Q(x-1)\}^2}\,Q(dx)
      && \text{if } t_Q \in(0,1),
      \\
      &\dfrac{x^{-2}}{A_Q^\star}\,Q(dx)
      +
      \left(
      1-\dfrac{\Eb_Q[X^{-2}]}{A_Q^\star}
      \right) \pointmass_0(dx)
      && \text{if } t_Q=1.
    \end{aligned}
    \right.
    \label{eq:hellinger-nearest-null}
  \end{align}
  In the case \(t_Q=1\),
  \(\Eb_Q[X^{-2}] \le A_Q^\star =\Eb_Q[X^{-1}] <\infty.\)
\end{lemma}
\noindent
We defer the proof to Appendix~\ref{sec:proof-hellinger-projection}.

\subsubsection{The existence of powerful tests}

The Hellinger projection $P_{Q_n}^\star$ of an alternative $Q_n$ onto the
nulls $\nullbasic$ allows us to demonstrate the existence of
finite-sample but asymptotically powerful tests.
It is standard~\cite[e.g.][]{LeCam86,Tsybakov09} that if
\(nH^2(Q_n,\nullIid) \to 0\), then every uniformly valid level-\(\alpha\)
test has asymptotic power at most \(\alpha\).
On the other hand, although \(nH^2(Q_n,P_{Q_n}^\star) \to \infty\)
guarantees the existence of level-\(\alpha\) tests with power approaching
$1$ in the simple hypothesis problem of $P_{Q_n}^\star$ against $Q_n$, this
need not extend to a composite null.
Indeed, as is familiar, uniform validity over rich nonparametric null
classes often precludes nontrivial power
altogether~\citep{BahadurSa56,LeCam73,Barron89,ShahPe20}.
In spite of these impossibility results, we now show how \(P_{Q_n}^\star\)
induces a finite-sample valid test under the composite null
\(\nullCond^n\), while attaining asymptotic power one whenever
\(nH^2(Q_n,\nullIid)\to \infty.\)

To motivate this test, let \(t_n:=t_{Q_n}\) minimize the affinity
$A_{Q_n}(t)$ (recall Lemma~\ref{lem:hellinger-projection}), suppose that
\(t_n\in(0,1)\), and consider the likelihood ratio
\begin{equation*}
  \frac{dQ_n^n}{d(P^\star_{Q_n})^n}(\bm X)
  =
  (A_{Q_n}^\star)^n
  \left[
    \prod_{i=1}^n\{1+t_n(X_i-1)\}
  \right]^2
\end{equation*}
between \(Q_n^n\) and \((P^\star_{Q_n})^n\).
Then the likelihood ratio test has the form
\begin{equation}
  \varphi^{\textup{lr}}_n(\Xb) := \one \left\{
  \prod_{i=1}^n\{1+t_n(X_i-1)\}\ge\theta_n
  \right\}
  \label{eqn:lrtest-unimplementable}
\end{equation}
for some threshold \(\theta_n \in \Rb\).
Because $\prod_{i=1}^n (1 + t(X_i - 1))$
has expectation at most one under every law in
\(\nullCond^n\), Markov's inequality shows that
\(\theta_n=1/\alpha\) guarantees uniform level $\alpha$ under
$\nullCond^n$; the next theorem demonstrates asymptotic power
for the (unimplementable) test~\eqref{eqn:lrtest-unimplementable}.

\begin{theorem}[Zero-one detectability]
  \label{thm:np-hellinger-scale}
  Let \(\alpha\in(0,1)\) and
  \((Q_n)\) be a sequence of probability laws on \(\Rb_+\) with
  \(\Eb_{Q_n}X>1\).
  \begin{enumerate}[label=(\roman*),leftmargin=*]
  \item \label{item:no-detection}
    Let \(\varphi_n: \Rb_+^n\to[0,1]\) satisfy
    \(\sup_{P \in \nullbasic}\Eb_{P^n}[\varphi_n]\le\alpha.\)
    Then
    \begin{align*}
      \Eb_{Q_n^n} [\varphi_n]
      \le
      \alpha+
      \sqrt{
        1-
        \{1-H^2(Q_n,\nullbasic)\}^{2n}
      },
    \end{align*}
    and \(nH^2(Q_n,\nullbasic) \to 0\) implies that
    \(\limsup_{n\to\infty}\Eb_{Q_n^n}[\varphi_n] \le \alpha\).
  \item \label{item:asymptotic-detection}
    Choose any \(t_n\in\operatorname*{argmin}_{0\le t\le1}A_{Q_n}(t).\)
    Then \(0 < t_n \le 1\), and the likelihood ratio
    test~\eqref{eqn:lrtest-unimplementable} with the choice $\theta_n =
    1/\alpha$ has level \(\sup_{P \in \nullbasic}\Eb_{P^n}[\lrtest_n]
    \le \alpha\) and power
    \begin{align*}
      1-\Eb_{Q_n^n} [\lrtest_n]
      \le \frac{ \{1-H^2(Q_n,\nullbasic)\}^{2n}}{\alpha}.
    \end{align*}
    Consequently, \( nH^2(Q_n,\nullbasic)\to\infty \) implies
    that \(\Eb_{Q_n^n}[\lrtest_n] \to1.\)
  \end{enumerate}
\end{theorem}
\noindent
We defer the proof to Appendix~\ref{sec:proof-np-hellinger-scale}.

The theorem characterizes optimal asymptotic power when $n H^2(Q_n,
\nullbasic) \to \{0, \infty\}$ or, equivalently, when
\begin{align*}
  H^2\left(Q_n^n,\nullIid^n\right) \to 0
  ~~ \mbox{or} ~~
  H^2\left(Q_n^n,\nullIid^n\right) \to 1.
\end{align*}
In the intermediate critical case that \(
nH^2(Q_n,\nullbasic)\to\kappa\in(0,\infty) \), the bounds in
Theorem~\ref{thm:np-hellinger-scale} no longer precisely characterize the
optimal power.
We defer discussion of this intermediate case to
\Cref{sec:asymptotics-regular}, where we characterize the limiting power in
this intermediate regime under local asymptotic regularity assumptions.
There, we also provide some discussion of asymptotically full- and trivial-power
tests under regularity conditions relating to the mean
excess $\Eb_{Q_n}[X] - 1$.

\subsubsection{Statistics attaining zero-one detectability}

The likelihood ratio test~\eqref{eqn:lrtest-unimplementable} depends on the
unknown alternative \(Q_n\), so it serves only as an oracle benchmark.
We now show that \(\nplr\) and \(\pBinPlus\) attain the same Hellinger
detectability threshold without knowledge of \(Q_n\).
We first state a comparison between $\pBinPlus$ and $\nplr$ that transfers
power guarantees from $\nplr$ to $\pBinPlus$, deferring the proof to
Appendix~\ref{sec:binplus-nplr-comparison}.

\begin{proposition}[Comparison between \(\pBinPlus\) and \(\nplr\)]
  \label{proposition:binplus-nplr-comparison}
  For every \(\bm{x}\in\Rb_+^n\),
  \begin{equation*}
    \pBinPlus(\bm{x})\le\sqrt{\nplr(\bm{x})}.
  \end{equation*}
\end{proposition}

From this proposition, we immediately obtain a (uniform) asymptotic
power guarantee for $\nplr$, $\pBinPlus$, and their pointwise minimum
\(\min\{\nplr,\pBinPlus\}\).
\begin{corollary}
  \label{corollary:statistics-attain-hellinger}
  Let the assumptions of \Cref{thm:np-hellinger-scale} hold
  and $\alpha \in (0, 1)$.
  Let $\mc{Q}_n$ be collections of distributions on $\Rb_+$ satisfying
  the separation condition
  \begin{equation*}
    \inf_{Q_n \in \mc{Q}_n} n \cdot H^2(Q_n, \nullbasic) \to \infty
  \end{equation*}
  as $n \to \infty$.
  Then both $\nplr$ and $\pBinPlus$ are asymptotically powerful
  uniformly over $\mc{Q}_n$:
  \begin{equation*}
    \inf_{Q_n \in \mc{Q}_n}
    \Pb_{Q_n^n} \bigl\{ \nplr(\Xb)\le\alpha \bigr\} \to 1
    ~~~ \mbox{and} ~~~
    \inf_{Q_n \in \mc{Q}_n} \Pb_{Q_n^n}
    \bigl\{ \pBinPlus(\Xb)\le\alpha \bigr\} \to 1.
  \end{equation*}
\end{corollary}
\begin{proof}
  Define the product \(K_n(t_n) := \prod_{i=1}^n \bigl\{ 1+t_n(X_i-1)
  \bigr\}\), so that \(\nplr(\bm{X}_{1:n}) \le K_n(t_n)^{-1}\) by the
  definition~\eqref{eq:def-nplr}.
  By independence and Lemma~\ref{lem:hellinger-projection},
  for any $Q_n \in \mc{Q}_n$ we have
  \begin{align*}
    \Eb_{Q_n^n}\left[ K_n(t_n)^{-1}\right]
    =
    \left[
      \Eb_{Q_n}
      \frac1{1+t_n(X-1)}
      \right]^n
    =
    A_{Q_n}(t_n)^n
    =
    \{1-H^2(Q_n,\nullbasic)\}^{2n}.
  \end{align*}
  By the standard inequality
  $1 - u \le e^{-u}$ we have
  $(1 - H^2(Q_n, \nullbasic))^n
  \le e^{-n H^2(Q_n, \nullbasic)}$, so
  \begin{align*}
    \Pb_{Q_n^n}\left\{\nplr(\Xb) > \alpha
    \right\}
    \le \Pb_{Q_n^n} \left\{K_n(t_n)^{-1} > \alpha\right\}
    \le \frac{1}{\alpha} \Eb_{Q_n^n}
    \left[K_n(t_n)^{-1}\right]
    \le \frac{1}{\alpha} e^{-n H^2(Q_n, \nullbasic)}.
  \end{align*}
  Thus $\sup_{Q_n \in \mc{Q}_n}
  \Pb_{Q_n^n}(\nplr(\Xb) > \alpha)
  \to 0$, and
  Proposition~\ref{proposition:binplus-nplr-comparison} gives the
  result for $\pBinPlus$.
\end{proof}

\subsection{Local asymptotic power under regularity}
\label{sec:asymptotics-regular}

The zero-one result leaves unresolved the critical regime
\(H(Q_n,\nullbasic) \asymp n^{-1/2}\).
Under typical finite-variance alternatives, this scaling
corresponds to familiar local $n^{-1/2}$
changes in the mean~\citep[cf.][]{VanDerVaart98}.
Leveraging classical semiparametric theory, we can develop asymptotically
optimal power under local nonparametric alternatives, which we recall below
and compare with certain parametric alternatives.
To determine whether the proposed $p$-values $\nplr$, $\pBinPlus$, and
\(\min\{\nplr,\pBinPlus\}\) attain this optimal power, we use a central limit
theorem together with careful quadratic expansions of their defining
objectives.
While in the super-critical regime $H^2(Q_n, \nullbasic) \gg 1/n$, each
attains limiting power $1$
(Corollary~\ref{corollary:statistics-attain-hellinger}), we show that
\(\pBinPlus\) and \(\min\{\nplr,\pBinPlus\}\) attain optimal local
asymptotic power, while \(\nplr\) does not.

\subsubsection{Fundamental power limits}

Let us recapitulate the local asymptotic theory on which our development
reposes~\cite[Chs.~7 and 25]{VanDerVaart98}.
Let $P_0$ be a distribution on $\Rb_+$ satisfying
\begin{align*}
  \Eb_0 X=1,
  \qquad
  \sigma^2:=\Var_0(X) \in (0, \infty).
\end{align*}
Let $\theta_0>0$ and $(Q_\theta)_{|\theta|<\theta_0}$ be a family of
distributions parameterized by $\theta \in \Rb$
with $Q_0=P_0$, all supported on $\Rb_+$.
We assume the model $(Q_\theta)$ is differentiable in quadratic mean
(DQM)~\citep[Sec.~7.2]{VanDerVaart98} at $\theta=0$, meaning there exists a
score $\score_\theta : \Rb_+ \to \Rb$ with $\Eb_0[\score_0^2] < \infty$ such
that
\begin{align}
  \label{eq:dqm-assumption}
  \int\left(
    \sqrt{dQ_\theta}
    -\sqrt{dP_0}
      \left(1+\frac{\theta}{2} \score_0 \right)
  \right)^2
  =o(\theta^2),
\end{align}
which implies
$\Eb_0[\score_0] = 0$.
We assume the Fisher
\(\finfo_\theta := \Eb_\theta[\score_\theta^2]\)
is positive at $\theta = 0$,
i.e., $\finfo_0 > 0$.

We follow standard practice~\citep[cf.][Ch.~25]{VanDerVaart98}
to assume the mean functional $\mu(Q) = \Eb_Q[X]$
is differentiable along the
submodel $(Q_\theta)$
with positive derivative
\begin{align}
  \dot{\mu}
  &:=
  \lim_{\theta\to0}
  \frac{\Eb_\theta X-\Eb_0X}{\theta}
  =
  \Eb_0[X \score_0(X)]
  > 0.
  \label{eq:mean-derivative}
\end{align}
Since \(\Eb_0X=1\), the limit~\eqref{eq:mean-derivative} is equivalent to
\begin{equation*}
  \Eb_\theta X
  =
  1+\dot\mu\,\theta+o(\theta).
\end{equation*}
The influence function representation~\eqref{eq:mean-derivative} is standard
under DQM~\citep[Exs.~25.16 and 25.24]{VanDerVaart98}.
For completeness (see Appendix~\ref{sec:derivative-of-mean}), we record the
following

\begin{example}
  \label{example:dqm-mean}
  Let the model family $(Q_\theta)_{\theta \in \Rb}$ be differentiable in
  quadratic mean at $\theta = 0$ with score $\score_0$.  and assume the
  uniform integrability condition that for some $\varepsilon > 0$,
  \begin{align*}
    \lim_{M\to\infty}
    \sup_{|\theta|<\varepsilon}
    \Eb_\theta\!\left[
      X^2\one\{X>M\}
      \right]
    = 0.
  \end{align*}
  Then $\Eb_\theta[X]$ is finite for $\theta$ near 0, $\Eb_0[|X \score_0|] <
  \infty$, and $\mu(Q) = \Eb_Q[X]$ has derivative~\eqref{eq:mean-derivative}.
\end{example}

Local asymptotic normality and semiparametric efficiency theory allows
us to derive power limits for tests under local alternatives
$Q_{h/\sqrt{n}}$ to $\nullbasic$ for $h > 0$; see \citet[Chs.~15
  and~25]{VanDerVaart98}.
To state these results, let $z_{1 - \alpha}$ be the $1 - \alpha$ quantile
of a standard Gaussian.

\begin{corollary}[Nonparametric power]
  \label{corollary:nonparametric-envelope}
  Let $\alpha \in (0, 1)$ and
  $\varphi_n:\Rb_+^n\to[0,1]$ be a sequence of asymptotic
  level $\alpha$ tests, i.e.,
  \(\limsup_n \sup_{P\in\nullIid^n}\Eb_P[\varphi_n] \le \alpha\).
  Let the model family $(Q_\theta)$ be DQM~\eqref{eq:dqm-assumption}
  at $\theta = 0$
  with derivative $\dot{\mu}$ as in~\eqref{eq:mean-derivative} and
  $h > 0$.
  Then
  \begin{align*}
    \limsup_{n\to\infty}
    \Eb_{Q_{h/\sqrt{n}}^n}\left[\varphi_n\right]
    \le
    1-\Phi\left(
    z_{1-\alpha}-\frac{h\dot\mu}{\sigma}
    \right).
  \end{align*}
\end{corollary}
\noindent
Appendix~\ref{sec:proof-nonparametric-envelope} includes a proof for
completeness.

Corollary~\ref{corollary:nonparametric-envelope}
shows that the optimal asymptotic power of any test in the
local alternatives $Q_{h/\sqrt{n}}$ coincides with that in a
Gaussian mean-shift experiment, and
the sample mean provides an asymptotically optimal test.
Indeed, letting $\wb{X}_n = \frac{1}{n} \sum_{i = 1}^n X_i$,
Le Cam's third lemma gives
\begin{equation*}
  \sqrt{n} (\wb{X}_n-1)
  \convD_{Q_{h/\sqrt n}}
  \normal\left(h\dot\mu, \sigma^2 \right)
\end{equation*}
along alternatives $Q_{h/\sqrt{n}}$.
Given a consistent estimate $\what{\sigma}^2$ of $\sigma^2$,
the asymptotic level-$\alpha$ test
\({\sqrt n(\wb{X}_n - 1)} / \what{\sigma} > z_{1-\alpha}\)
has limiting power
\(1-\Phi(z_{1-\alpha}-\frac{h\dot\mu}{\sigma})\).
While it lacks finite-sample validity, the test provides a comparison in
the same spirit as those underlying \citet{RomanoWo00}.
The nonparametric class $\nullbasic$ leads to some loss
of power over parametric models, which we record here to provide
intuition:
\begin{corollary}[Parametric oracle]
  \label{corollary:parametric-oracle}
  Let the conditions of Corollary~\ref{corollary:nonparametric-envelope} hold,
  $P_0 \in \nullbasic$, and
  $\finfo_0 = \Eb_0[\score_0^2]$ be the Fisher information.
  The level-$\alpha$ likelihood ratio test $\lrtest_n$ for
  testing $P_0^n$ against $Q_{h/\sqrt n}^n$ has asymptotic
  power
  \begin{align*}
    \lim_n \Eb_{Q_{h/\sqrt{n}}^n}\left[\lrtest_n\right]
    =1-\Phi\left(z_{1-\alpha}-h\sqrt{\finfo_0}\right).
  \end{align*}
\end{corollary}
\begin{proof}
  The experiments under \(P_0^n\) and \(Q_{h/\sqrt n}^n\) converge to the
  Gaussian shift experiment \(Z \sim \normal(0,1)\)
  versus \(Z \sim \normal(h\sqrt{\finfo_0},1)\); see, for example,
  \cite[Chs.~6--7 and~15]{VanDerVaart98}.
  The Neyman-Pearson (likelihood ratio) test has limiting power equal to
  that of the level-\(\alpha\) test rejecting when
  $Z > z_{1 - \alpha}$,
  which has power
  $\Pb(\normal(h\sqrt{\finfo_0} ,1)>z_{1-\alpha})
  =1-\Phi(z_{1-\alpha}-h\sqrt{\finfo_0})$.
\end{proof}

The largest possible
asymptotic relative efficiency (ARE) of a test valid over the
nonparametric mean null $\nullbasic$ relative to the unimplementable
parametric oracle is the squared correlation
\begin{align*}
  \frac{(\dot\mu/\sigma)^2}{\finfo_0}
  =
  \frac{\dot\mu^2}{\sigma^2 \finfo_0}
  =
  \frac{\Eb_0[X \score_0]^2}{\sigma^2 \finfo_0}
  =
  \operatorname{Corr}_0(X, \score_0)^2.
\end{align*}
For example, in the lognormal model $X_\theta=\exp(Z_\theta)$, with
$Z_\theta\sim \normal(-\frac{1}{2} + \theta,1)$ we have $\sigma^2=e-1$,
$\score_0(x)=\log x + 1/2$, $\finfo_0=1$, and $\dot{\mu}=1$,
yielding efficiency
$(e-1)^{-1}\approx 58\%$.

\subsubsection{Local behavior of $\nplr$}

With these fundamental power limits recapitulated, we turn to deriving
the local asymptotic behavior of the tests $\nplr$ and $\pBinPlus$,
beginning with $\nplr$.
Let $X_{n,i} \simiid Q_{h/\sqrt{n}}$ satisfy the DQM
conditions~\eqref{eq:dqm-assumption} and~\eqref{eq:mean-derivative}.
Define the centered variables and empirical variance
\begin{align*}
  \xi_{n,i}:=X_{n,i}-1,
  \qquad
  S_n:=\frac1{\sqrt n}\sum_{i=1}^n\xi_{n,i},
  \qquad
  V_n:=\frac1n\sum_{i=1}^n\xi_{n,i}^2 .
\end{align*}
Then quadratic mean differentiability implies~\citep[Theorem 7.2 and Lemma
  6.4]{VanDerVaart98} the convergence under the local alternatives
$Q_{h/\sqrt{n}}$ that
\begin{align}
  \label{eq:local-quadratic-regularity}
  S_n \convD \normal(h\dot\mu,\sigma^2),
  \qquad
  V_n \convP \sigma^2,
  \qquad
  \frac{\max_{i\le n}|\xi_{n,i}|}{\sqrt n}\convP 0.
\end{align}

\begin{theorem}[Limiting power of the nonparametric likelihood ratio]
  \label{thm:nplr-local-limit}
  Assume the limit~\eqref{eq:local-quadratic-regularity} holds
  and let $\Xb_n = (X_{n,1}, \ldots, X_{n,n})$.
  Then under the alternatives $Q_{h/\sqrt{n}}$,
  \begin{align*}
    -\log \nplr(\Xb_n)
    = \frac{(S_n)_+^2}{2V_n} + o_p(1).
  \end{align*}
\end{theorem}

As an immediate consequence, letting \(Z_h\sim \normal(h\dot\mu/\sigma,1)\)
and $h > 0$,
\begin{equation}
  \lim_{n\to\infty}\Pb_{Q_{h/\sqrt n}^n}\{\nplr\le\alpha\}
  =
  \Pb\left\{
  \exp\left(-\frac{(Z_h)_+^2}{2}\right)\le\alpha
  \right\}
  =
  1-\Phi\left(
  \sqrt{2\log(1/\alpha)}-\frac{h\dot\mu}{\sigma}
  \right)
  \label{eq:raw-nplr-power}
\end{equation}
for each \(\alpha\in(0,1)\).
The statistic is therefore strictly more conservative than an asymptotically
exact one-sided Gaussian \(p\)-value.
Indeed, for every $\alpha\in(0,1)$,
\begin{align*}
  1-\Phi\left(\sqrt{2\log(1/\alpha)}\right)<\alpha,
\end{align*}
because the normal tail is strictly smaller than $\exp(-x^2/2)$ for all $x >
0$.
Even more, as $\alpha \downarrow 0$ we have
\begin{align*}
  \frac{1 - \Phi(z_{1 - \alpha} - h \dot{\mu} / \sigma)}{
    1 - \Phi(\sqrt{2 \log(1/\alpha)} - h \dot{\mu} / \sigma)}
  = (1 + o(1)) \cdot 2 \sqrt{\pi \log \frac{1}{\alpha}}.
\end{align*}
Indeed, the second order Gaussian quantile expansion gives
\begin{align*}
  z_{1-\alpha}
  = \sqrt{2\log(1/\alpha)}
  - \frac{\log(4\pi)+\log\log(1/\alpha)}
  {2\sqrt{2\log(1/\alpha)}}
  + o\left(\frac{1}{\sqrt{\log(1/\alpha)}}\right).
\end{align*}
Hence by Mills' ratio,
\begin{align*}
  \log \frac{ 1 - \Phi(z_{1 - \alpha} - h \dot{\mu} / \sigma)}{
    1 - \Phi(\sqrt{2 \log(1/\alpha)} - h \dot{\mu} / \sigma)}
  =
  \frac{2\log(1/\alpha)-z_{1-\alpha}^2}{2}+o(1)
  = \frac{1}{2}\log\left(4\pi\log(1/\alpha)\right)+o(1).
\end{align*}
Exponentiating proves the claim.
Thus \( \nplr \) fails to attain the optimal limiting power in
Corollary~\ref{corollary:nonparametric-envelope}.

We return to prove \Cref{thm:nplr-local-limit}.
We use a small technical lemma to pass from a local quadratic approximation
of the log-likelihood process to its global maximum.

\begin{lemma}
  \label{lem:concave-local-maximization}
  Let \(r_n>0\) satisfy \(r_n\to0\) and
  \(f_n:[0,1]\to \Rb \cup \{-\infty\}\) be a random closed
  concave function.
  Let \((S_n,V_n)\) be random variables such that
    \(S_n=O_p(1)\) and \(V_n\convP v>0\).
    For \(u\ge0\), define
    \begin{equation*}
      g_n(u):=uS_n-\frac{u^2}{2}V_n.
    \end{equation*}
    Assume that $\sup_{0\le u\le A}
    |f_n(r_nu)-g_n(u)| \convP 0$ for each $A<\infty$ as $n
    \to \infty$.
    Then
    \begin{align*}
      \sup_{0\le t\le1}f_n(t)
      =\sup_{u\ge0}g_n(u)+o_p(1)
      =\frac{(S_n)_+^2}{2V_n}+o_p(1).
    \end{align*}
\end{lemma}
\begin{proof}
  Let $\varepsilon,\delta>0$.
  Since $S_n=O_p(1)$, we may choose $A < \infty$ large enough that
  $A^2 v > 16 \varepsilon$ and
  \begin{equation*}
    \limsup_{n\to\infty}\Pb\left(|S_n|>\frac{Av}{8}\right)<\delta.
  \end{equation*}
  Define the event
  \begin{align*}
    \mc{E}_n \defeq \left\{V_n \ge \frac{v}{2},
    ~~ |S_n| \le \frac{Av}{8},
    ~~
    \sup_{0 \le u \le A} |f_n(r_n u) - g_n(u)|
    \le \varepsilon \right\},
  \end{align*}
  Since $V_n\convP v$ and the local approximation holds for this fixed $A$,
  \begin{equation*}
    \liminf_{n\to\infty}\Pb(\mc{E}_n)\ge 1 - \delta.
  \end{equation*}
  On $\mc{E}_n$, we observe that
  \begin{align*}
    f_n(A r_n) - f_n(0)
    \le g_n(A) - g_n(0)
    + 2 \varepsilon
    \le A S_n - \frac{A^2 V_n}{2} + 2 \varepsilon
    \le -\frac{A^2 v}{8} + 2 \varepsilon < 0,
  \end{align*}
  so by concavity any $t_n^\star \in \argmax_{0\le t\le1} f_n(t)$ lies in
  $[0, A r_n]$.
  Similarly, on $\mc{E}_n$, \((S_n)_+/V_n\) maximizes
  \(g_n\) over $\Rb_+$ and lies in \([0,A]\).
  For $n$ large enough that $A r_n \le 1$, we thus obtain
  \begin{equation*}
    \left|\sup_{0\le t\le1}f_n(t)
    -
    \sup_{0\le u\le A} g_n(u)\right|
    \le \varepsilon
    ~~ \mbox{and} ~~
    \sup_{0 \le u \le A}g_n(u)
    = \frac{(S_n)_+^2}{2 V_n}
  \end{equation*}
  on $\mc{E}_n$.
  Consequently,
  \begin{equation*}
    \limsup_{n\to\infty}
    \Pb\left(
      \left|\sup_{0\le t\le1}f_n(t)
      -\frac{(S_n)_+^2}{2V_n}\right|>\varepsilon
    \right)
    \le\delta.
  \end{equation*}
  Since $\delta>0$ is arbitrary, this proves the result.
\end{proof}

\begin{proof}[Proof of \Cref{thm:nplr-local-limit}]
  Let
  \begin{equation*}
    f_n(t):=\sum_{i=1}^n\log(1+t\xi_{n,i}),
    \qquad
    g_n(u):=uS_n-\frac{u^2}{2}V_n,
    \qquad
    r_n:=n^{-1/2}.
  \end{equation*}
  Then $-\log\nplr=\sup_{0\le t\le1}f_n(t)$.  We verify the hypotheses of
  \Cref{lem:concave-local-maximization}.  For fixed $A<\infty$, Taylor's theorem
  on the event $A\max_i|\xi_{n,i}|/\sqrt n\le1/2$ gives, uniformly for
  $0\le u\le A$,
  \begin{equation*}
    f_n(r_nu)
    =g_n(u)
    +O_A\left(
    \frac{\max_i|\xi_{n,i}|}{\sqrt n}V_n
    \right)
    =g_n(u)+o_p(1),
  \end{equation*}
  where the second equality follows by
  assumption~\eqref{eq:local-quadratic-regularity}.
  \Cref{lem:concave-local-maximization} therefore gives
  \begin{equation*}
    -\log\nplr(\Xb_n)
    =
    \sup_{u\ge0}g_n(u)+o_p(1)
    =
    \frac{(S_n)_+^2}{2V_n}+o_p(1).
  \end{equation*}
  Because $V_n \cp \sigma^2$ and
  $S_n \cd \normal(h \dot{\mu}, \sigma^2)$
  under $Q_{h/\sqrt{n}}$
  by assumption~\eqref{eq:local-quadratic-regularity},
  the continuous mapping theorem and Slutsky's lemmas give
  \begin{equation*}
    \nplr(\Xb_n)
    \convD
    \exp\left(-\frac{(Z_h)_+^2}{2}\right).
    \qedhere
  \end{equation*}
\end{proof}

\subsubsection{Sharpness for $\pBinPlus$ and $\min\{\nplr,\pBinPlus\}$}

We can analyze the asymptotic behavior of the
generalized binomial test $\pBinPlus$ from Section~\ref{sec:binplus},
Proposition~\ref{prop:binplus-ratio-monotonicity},
using similar techniques.
In contrast to the limit in Theorem~\ref{thm:nplr-local-limit} for the
nonparametric likelihood ratio, the test does achieve the optimal asymptotic
power that Corollary~\ref{corollary:nonparametric-envelope} specifies.

\begin{theorem}[Limiting power of the generalized binomial statistic]
  \label{thm:Y-local-limit}
  Assume the limit~\eqref{eq:local-quadratic-regularity} holds
  and let $\Xb_n = (X_{n,1}, \ldots, X_{n,n})$.
  Then under the alternatives $Q_{h/\sqrt{n}}$,
  \begin{align*}
    \pBinPlus(\Xb_n)
    =
    \Phi\left(-\frac{S_n}{\sqrt{V_n}}\right) + o_p(1).
  \end{align*}
\end{theorem}

As an immediate consequence of this result,
we obtain optimal limiting asymptotic power.
Indeed, if \(Z_h\sim \normal(h\dot\mu/\sigma,1),\) then for every
$\alpha\in(0,1)$,
\begin{align}
  \lim_{n\to\infty}
  \Pb_{Q_{h/\sqrt n}^n}\{\pBinPlus\le\alpha\}
  = \Pb\{ \Phi(-Z_h)\le\alpha\}
  =1-\Phi\left(z_{1-\alpha}-h\dot\mu/\sigma\right).
  \label{eq:Y-local-power}
\end{align}
Since \(\min\{\nplr,\pBinPlus\}\) is a valid nonparametric \(p\)-value by
Theorem~\ref{thm:min-nplr-binplus-valid} and is pointwise no larger than
\(\pBinPlus\), the upper bound of
Corollary~\ref{corollary:nonparametric-envelope} and the
limit~\eqref{eq:Y-local-power} imply
\begin{align*}
  \lim_{n\to\infty}\Pb_{Q_{h/\sqrt n}^n}
  \left\{\min\{\nplr(\Xb_n),\pBinPlus(\Xb_n)\}\le\alpha\right\}
  =
  1-\Phi\left(z_{1-\alpha}-h\dot\mu/\sigma\right)
\end{align*}
as well.
Thus $\pBinPlus$ and \(\min\{\nplr,\pBinPlus\}\) attain the best possible
nonparametric power at $1/\sqrt{n}$ scale local alternatives, while the
statistic $\nplr$ does not.

To prove \Cref{thm:Y-local-limit}, we follow a similar technique to that we
use to prove Theorem~\ref{thm:nplr-local-limit}.
We begin with a technical lemma to pass from a local quadratic approximation
of the log-likelihood to a Gaussian approximation of the ratio of integrals
defining \(\pBinPlus\), essentially performing a Laplace approximation.

\begin{lemma}[Gaussian ratio for concave integrals]
  \label{lem:gaussian-ratio-concave-integrals}
  Let \(a_n\convP-\infty\) and \(b_n\convP\infty\), and let \(f_n\) be a random
  function that is finite and concave on \((a_n,b_n)\).
  Let \((s_n,v_n)\) be random variables such that \(s_n=O_p(1)\)
  and \(v_n\convP\sigma^2>0\).
  Define the quadratics
  \begin{equation*}
    g_n(u):=us_n - \frac{u^2}{2}v_n
  \end{equation*}
  and
  assume the approximation that for every $A>0$,
  $\sup_{|u|\le A}|f_n(u)-g_n(u)| \cp 0$.
  Then
  \begin{align*}
    \frac{\int_{a_n}^0 e^{f_n(u)}\,du}{\int_{a_n}^{b_n} e^{f_n(u)}\,du}
    =
    \frac{\int_{a_n}^0 e^{g_n(u)} \, du}{
      \int_{a_n}^{b_n} e^{g_n(u)} \, du} + o_p(1)
    =
    \Phi\left(-\frac{s_n}{\sqrt{v_n}}\right) + o_p(1).
  \end{align*}
\end{lemma}
\begin{proof}
  Let \(\eta,\delta>0\).
  Since $s_n=O_p(1)$, choose $B\ge\sigma^2/8$ sufficiently large that
  \begin{equation*}
    \limsup_{n\to\infty}\Pb(|s_n|>B)<\delta,
    \qquad \frac{64}{\sigma^2}e^{-B/2}<\frac{\eta}{2}.
  \end{equation*}
  Set $A:=8B/\sigma^2$, and choose $\varepsilon>0$ small enough that
  \begin{equation*}
    \varepsilon<\frac{B}{4}
    \qquad\text{and}\qquad
    (e^\varepsilon-1)\frac{2\sqrt\pi}{\sigma}
    e^{B^2/\sigma^2}<\frac{\eta}{2}.
  \end{equation*}
  Define the event
  \begin{align*}
    \mc{E}_{n,B} \defeq \Big\{|s_n| \le B,
    ~~ v_n \ge \half \sigma^2,
    ~~ a_n\le-A, ~~ b_n\ge A,
    ~~ \sup_{-A \le u \le A} |f_n(u) - g_n(u)|
    \le \varepsilon \Big\}.
  \end{align*}
  The assumptions on $a_n,b_n,s_n,v_n$, and the local approximation imply
  \begin{equation*}
    \liminf_{n\to\infty}\Pb(\mc{E}_{n,B})\ge1-\delta.
  \end{equation*}
  We first claim that
  on the event $\mc{E}_{n,B}$,
  \begin{align}
    \label{eqn:linear-decreasing-fn}
    f_n(u) \le -\frac{B}{2A} |u| ~~ \mbox{for} ~~ u \in [a_n, -A]
    \cup [A, b_n].
  \end{align}
  Indeed, we have $\varepsilon \ge |f_n(0) - g_n(0)| = |f_n(0)|$,
  while
  $g_n(A) = A s_n - \frac{A^2}{2} v_n
  \le AB - \frac{A^2 \sigma^2}{4} \le -B$, and
  $g_n(-A) \le -B$ as well.
  Concavity and the criterion of increasing
  slopes~\cite[Ch.~I]{HiriartUrrutyLe93}, together with $\varepsilon<B/4$,
  thus imply the linear decrease~\eqref{eqn:linear-decreasing-fn}.

  Now we observe that
  \begin{align*}
    \int_{a_n}^0 e^{f_n(u)} du
    & = \int_{a_n}^{-A} e^{f_n(u)} du
    + \int_{-A}^0 e^{f_n(u)} du \\
    & = \int_{a_n}^{-A} e^{f_n(u)} du
    - \int_{a_n}^{-A} e^{g_n(u)} du
    + \int_{-A}^0 (e^{g_n(u)} - e^{f_n(u)}) du
    + \int_{a_n}^0 e^{g_n(u)} du.
  \end{align*}
  The first integral satisfies
  $\int_{a_n}^{-A} e^{f_n(u)} du
  \le \int_{A}^\infty e^{-B u / (2A)} du
  = \frac{2A}{B} e^{-B/2}$, while
  \begin{align*}
    \left|\int_{-A}^0 (e^{g_n(u)} - e^{f_n(u)}) du\right|
    \le (e^\varepsilon-1) \int_{-A}^0 e^{g_n(u)} du 
    \le (e^\varepsilon-1)\frac{2\sqrt\pi}{\sigma}
      e^{B^2/\sigma^2}.
  \end{align*}
  Observing that $g_n(u) \le |u| B - \frac{u^2 \sigma^2}{4}
  \le -\frac{B}{2A} |u|$ for $u \le -A$ as well,
  we see that
  $\int_{a_n}^{-A} e^{g_n(u)} du
  \le \frac{2A}{B} e^{-B/2}$.
  Combining these bounds yields
  \begin{align*}
    \bigg|\int_{a_n}^0 e^{f_n(u)} du
    - \int_{a_n}^0 e^{g_n(u)} du \bigg|
    \le \frac{4A}{B}e^{-B/2}
    +(e^\varepsilon-1)\frac{2\sqrt\pi}{\sigma}e^{B^2/\sigma^2}
    <\eta
  \end{align*}
  on the event $\mc{E}_{n,B}$.
  Because $\eta,\delta>0$ were arbitrary, we obtain
  $\int_{a_n}^0 e^{f_n(u)} du = \int_{a_n}^0 e^{g_n(u)} du + o_p(1)$.
  A completely similar calculation
  gives $\int_{a_n}^{b_n} e^{f_n(u)} = \int_{a_n}^{b_n} e^{g_n(u)} du + o_p(1)$.

  Because \(v_n\convP\sigma^2>0\) by assumption, the quantities
  \(\sqrt{2\pi/v_n}\exp\{s_n^2/(2v_n)\}\) arising
  in the integrals $\int e^{g_n(u)}du$ are bounded away from zero in
  probability.
  The continuous mapping theorem gives the lemma.
\end{proof}

\begin{proof}[Proof of \Cref{thm:Y-local-limit}]
  Let $k_n$ index a maximal observation $X_{n,i}$, so for $\xi_{n,i} =
  X_{n,i} - 1$ we have \(\xi_{n,k_n}=\max_{i\le n}\xi_{n,i}\).
  On the event $\xi_{n,k_n}\le0$, all $\xi_{n,i}\in[-1,0]$, so
  $\xi_{n,i}\le-\xi_{n,i}^2$ and therefore
  the sum $S_n = \frac{1}{\sqrt{n}}
  \sum_{i = 1}^n \xi_{n,i} \le-\sqrt{n} V_n$.
  This is incompatible with the limits~\eqref{eq:local-quadratic-regularity}
  that $S_n=O_p(1)$ and $V_n\convP\sigma^2>0$, so
  \begin{equation*}
    \Pb_{Q_{h/\sqrt n}^n} \{\xi_{n,k_n}>0\} \to 1.
  \end{equation*}
  On the event $\xi_{n,k_n} > 0$, the binomial statistic~\eqref{eq:def-Pvee}
  satisfies
  \begin{equation*}
    \pBinPlus(\Xb_n)
    =
    \frac{\int_{a_n}^{0}e^{f_n(u)}\,du}{\int_{a_n}^{\sqrt n}e^{f_n(u)}\,du},
    \qquad
    a_n:=-\frac{\sqrt n}{\xi_{n,k_n}},
  \end{equation*}
  where
  \begin{equation*}
    f_n(u):=\sum_{i\ne k_n}
    \log\left(1+\frac{u\xi_{n,i}}{\sqrt n}\right),
    \qquad a_n\le u\le\sqrt n .
  \end{equation*}
  The factors $1 + u \xi_{n,i} / \sqrt{n}$ are nonnegative on $a_n \le u \le
  \sqrt{n}$ because $k_n$ is maximal.
  Since $0<\xi_{n,k_n}\le\max_i|\xi_{n,i}|$ and $\max_i|\xi_{n,i}|/\sqrt
  n\convP0$ by assumption~\eqref{eq:local-quadratic-regularity}, we have
  $a_n\convP-\infty$.

  To apply Lemma~\ref{lem:gaussian-ratio-concave-integrals}, set
  \begin{equation*}
    s_n:=\frac1{\sqrt n}\sum_{i\ne k_n}\xi_{n,i},
    \qquad
    v_n:=\frac1n\sum_{i\ne k_n}\xi_{n,i}^2,
    \qquad
    g_n(u) \defeq u s_n - \frac{u^2}{2} v_n.
  \end{equation*}
  Deleting observation $k_n$ changes $S_n$ and $V_n$ by at most
  $\max_i|\xi_{n,i}|/\sqrt n$ and $\max_i \xi_{n,i}^2 / n$, respectively, so
  the convergence~\eqref{eq:local-quadratic-regularity} implies
  \begin{equation*}
    s_n = S_n + o_p(1) = O_p(1)
    ~~~ \mbox{and} ~~~
    v_n = V_n+o_p(1)
    \cp \sigma^2 > 0.
  \end{equation*}
  Because $|\log(1 + t) - t + \frac{t^2}{2}|
  \le |t|^3$ for $|t| \le \half$,
  for $A < \infty$ we have
  \begin{equation*}
    \sup_{-A \le u \le A} |f_n(u) - g_n(u)|
    \le \sum_{i = 1}^n \frac{A^3 |\xi_{n,i}|^3}{n^{3/2}}
    \le A^3 \cdot \frac{\max_{i \le n} |\xi_{n,i}|}{\sqrt{n}}
    \frac{1}{n} \sum_{i = 1}^n \xi_{n,i}^2 \cp 0
  \end{equation*}
  when $A \max_{i \le n} |\xi_{n,i}| / \sqrt{n} \le \half$.
  Thus \Cref{lem:gaussian-ratio-concave-integrals} applies with $b_n=\sqrt
  n$ and yields
  \begin{equation*}
    \pBinPlus(\Xb_n)
    =
    \Phi\left(-\frac{s_n}{\sqrt{v_n}}\right)+o_p(1)
    =
    \Phi\left(-\frac{S_n}{\sqrt{V_n}}\right)+o_p(1),
  \end{equation*}
  by the continuous mapping theorem.
\end{proof}

\subsection{Discussion and context}

We now place the preceding asymptotic results in two broader contexts.
First, although $H^2(Q_n,\nullbasic)$, the squared Hellinger distance from $Q_n$
to the null class, governs detectability in
Theorem~\ref{thm:np-hellinger-scale}, the mean excess $\Eb_{Q_n}X-1$ is a
simpler and more familiar measure of signal strength.
Therefore we relate the two quantities in \Cref{sec:mean-excess}, both for
regular families with finite variance and for Pareto-type families with infinite
variance.

We next relate our results to anytime-valid testing, where data arrive
sequentially and sampling stops in response to accumulating evidence while
preserving type-I error control~\citep{Waudby-SmithRa24,RamdasWa25}.
Let $\mc{F}_n\defeq\sigma(X_1,\ldots,X_n)$ and let $\tau$ denote the time at which the test
first rejects.
Then $\tau$ is a $(\mc{F}_n)$-stopping time.
For a null class $\mc{P}$, the test is \emph{anytime-valid at level $\alpha$} if
\begin{equation*}
  \sup_{P\in\mc{P}}\Pb_P(\tau<\infty)
  \le\alpha.
\end{equation*}
We show in \Cref{sec:anytime-valid} that this anytime-valid guarantee imposes a
sharp cost in local power: under $n^{-1/2}$-local alternatives from a DQM
family, any such test has trivial asymptotic power.

\subsubsection{Hellinger distances, normality, and mean excess}\label{sec:mean-excess}

While Theorem~\ref{thm:np-hellinger-scale} delineates the cases
$H^2(Q_n, \nullbasic) \gg 1/n$ and
$H^2(Q_n, \nullbasic) \ll 1/n$,
the more familiar measure of signal strength against the mean
null $\Eb X \le 1$ is the excess \(\delta_n \defeq \Eb_{Q_n}X-1\).
For quadratic-mean-differentiable alternatives with limiting variance
\(\sigma^2\) and local uniform integrability as in
Example~\ref{example:dqm-mean}, we can expand the Hellinger distance in
terms of the derivative~\eqref{eq:mean-derivative}.
For quadratic-mean-differentiable families, the
definition~\eqref{eq:dqm-assumption} yields the familiar Fisher-information
scaling $H^2(Q_\theta, Q_0) = \finfo_0 \theta^2 / 8 + o(\theta^2)$ of the
Hellinger distance~\citep[Sec.~7.2]{VanDerVaart98} to the point null \(Q_0\).
In our case, Lemma~\ref{lem:hellinger-projection}
extends this to the full null class $\nullbasic$:
Proposition~\ref{proposition:regularly-varying-np-local} in
Appendix~\ref{sec:regularly-varying-np-local} shows that
\begin{align*}
  H^2(Q_\theta, \nullbasic) = \frac{\dot{\mu}^2 \theta^2}{8 \sigma^2}
  + o(\theta^2)
\end{align*}
as $\theta \downarrow 0$.
Thus \(H^2(Q_{h/\sqrt{n}},\nullbasic) = \dot{\mu}^2 h^2 / (8 n \sigma^2)
+o(h^2/n)\) uniformly in $h \ll \sqrt{n}$,
and so we see that
testing
\begin{equation*}
  H_0 : \Eb[X] \le 1
  ~~~ \mbox{versus} ~~~
  H_1 : Q_{h_n/\sqrt{n}}
\end{equation*}
becomes asymptotically impossible or exact precisely as $h_n \ll 1$ or
$h_n \gg 1$.

By contrast, different tail behavior of the alternative distributions
$Q_\theta$ than the uniform second moment in Example~\ref{example:dqm-mean}
can yield different local alternatives.
As one example,
the family $(Q_\theta)$ of alternatives has tails
scaling as a polynomial $x^{-\rho}$ for some $1 < \rho < 2$,
where we assume (for simplicity) that
\begin{align*}
  \Eb_{Q_\theta} X = 1 + \theta
\end{align*}
and for some values $C_\theta$ with $0 < c_- \le C_\theta \le c_+ < \infty$
that
\begin{align}
  \label{eqn:pareto-tail}
  \limsup_{x \to \infty}
  \sup_{0 < \theta < \theta_0}
  \left|\frac{Q_\theta(X > x)}{C_\theta \cdot x^{-\rho}}
  - 1 \right| = 0.
\end{align}
Then Proposition~\ref{proposition:heavy-tail-np-local} in
Appendix~\ref{sec:heavy-tail-np-local} shows that
as $\theta \downarrow 0$,
\begin{equation}
  \label{eq:regular-varying-hellinger}
  H^2(Q_\theta, \nullbasic)
  =
  \frac{\rho-1}{2\rho}
  \left(
    -\frac{\rho^2\pi C_\theta}{\sin(\pi\rho)}
  \right)^{-\frac{1}{\rho-1}}
  \cdot \theta^{\frac{\rho}{\rho-1}}
  + o\left(\theta^{\frac{\rho}{\rho - 1}}\right).
\end{equation}
That is, the critical radius for mean detection
scales as $n^{\frac{1 - \rho}{\rho}} \gg n^{-1/2}$.

\subsubsection{Anytime-valid tests}\label{sec:anytime-valid}

Consider any test $(T_n)$ that provides any-time validity under the null $H_0$,
so that $\Pb_{0,n}(T_n = 1) \le \alpha$ for all $n$.
Assume that the rejection regions for the test are monotonically
non-decreasing, so that $A_n \defeq \{T_n = 1\}$ satisfies $A_{n-1} \subset
A_n$.
From these, we may construct the stopping times
$\tau = \inf\{n \mid T_n = 1\}$.
For example, any test based on an e-process $E_n$ or sub-martingale that
rejects if $\max_{i \le n} E_i > t$ for a threshold $t$ satisfies this
(this is the Ville-style test).
It turns out that any such test has trivial power under local
alternatives.

\newcommand{\contiguousmutual}{\mathop{\triangleleft\triangleright}}
\newcommand{\contiguous}{\triangleleft}

To see this, we develop a few standard consequences of
quadratic mean differentiability and contiguity,
where we recall~\cite[Ch.~6]{VanDerVaart98}
that for two sequences $P_n$ and $Q_n$ of distributions,
$Q_n$ is contiguous with respect to $P_n$ if
$P_n(A_n) \to 0$ implies $Q_n(A_n) \to 0$, written
$Q_n \contiguous P_n$.
When sequences are mutually contiguous, we write
$Q_n \contiguousmutual P_n$.
We make the following observation:
\begin{lemma}
  \label{lemma:qmd-are-tv-close}
  Let $m \in \mathbb{N}$,
  $\{P_h\}$ be a q.m.d.\ family at $P_0$,
  and $h_n \to h \in \Rb$.
  Then
  \begin{align*}
    \lim_{n \to \infty} \tvnorm{P_{h_n/\sqrt{n}}^m - P_0^m} = 0
    ~~~ \mbox{and} ~~~
    P_{h_n/\sqrt{n}}^n \contiguousmutual P_0^n.
  \end{align*}
\end{lemma}
\begin{proof}
  Recall~\cite[Ch.~7]{VanDerVaart98} that
  for the Fisher information $\finfo = \Eb_0[\score_0^2]$,
  $H^2(P_\delta, P_0) = \frac{\delta^2}{8} \finfo + o(\delta^2)$ as
  $\delta \to 0$.
  Then for any $m \in \mathbb{N}$,
  \begin{align*}
    H^2\left(P_{h_n/\sqrt{n}}^m, P_0^m\right)
    & = 1 - (1 - H^2(P_{h_n/\sqrt{n}}, P_0))^m
    = 1 - \left(1 - \frac{h_n^2}{8 n} \finfo + o(1/n)\right)^m \to 0
  \end{align*}
  as $n \to \infty$.
  The mutual contiguity is similarly standard for quadratic mean
  differentiable families via
  local asymptotic normality~\cite[Ch.~7.3]{VanDerVaart98}.
  Using the standard inequality
  $\tvnorm{P - Q} \le \sqrt{2} H(P, Q)$ gives the result.
\end{proof}

The conditions Lemma~\ref{lemma:qmd-are-tv-close} provides are
enough to guarantee that any such stopping-time-based procedure
has asymptotically no power against local alternatives.
We can state this as a (very slightly) more general result:
\begin{lemma}
  Let $P_0$ and $(Q_n)_{n \ge 1}$ be probability measures such that
  for every $m \in \mathbb{N}$,
  \begin{equation}
    \label{eq:tv-condition}
    \lim_{n \to \infty} \tvnorm{Q_n^m - P_0^m} = 0,
  \end{equation}
  and for which $Q_n^n \contiguous P_0^n$.
  Let $\tau$ be stopping time and
  $\Pb_0$ be the distribution of $\{X_t\}_{t \in \mathbb{N}}$ under $P_0$.
  Then
  \begin{align*}
    \lim_{n \to \infty} Q_n^n(\tau \le n)
    = \Pb_0(\tau < \infty).
  \end{align*}
\end{lemma}
\noindent
So any level-$\alpha$ anytime valid test has trivial power under local
alternatives: Lemma~\ref{lemma:qmd-are-tv-close} shows that the local
alternatives $Q_{h/\sqrt{n}}$ for a DQM family satisfy the
condition~\eqref{eq:tv-condition}.
\begin{proof}
  Define the
  sets $A_m = \{\tau \le m\}$ that the stopping time is at most
  $m$, so that $A_\infty \defeq \cup_{m \in \mathbb{N}} A_m
  = \{\tau < \infty\}$ and $A_m \in \mc{F}_m = \sigma(X_1, \ldots, X_m)$.
  For $m \le n$, we have decomposition
  \begin{equation*}
    Q_n^n (A_n) = Q_n^n(A_m) + Q_n^n(A_n \setminus A_m)
    = Q_n^m(A_m) + Q_n^n(A_n \setminus A_m)
  \end{equation*}
  because $A_m \in \mc{F}_m$.
  The total variation
  limit~\eqref{eq:tv-condition} shows
  that \(Q_n^{m}(A_m) \to P_0^m(A_m)\)
  as \(n \to \infty\).
  For the second term \( Q_n^n(A_n \setminus A_m) \), let
  $\epsilon > 0$ be otherwise arbitrary.
  The contiguity
  condition $Q_n^n \contiguous P_0^n$ guarantees that
  there exists $\delta > 0$ and $N \in \mathbb{N}$ such that
  for $n \ge N$ and $B_n \in \mc{F}_n$, $P_0^n(B_n) < \delta$
  implies that $Q_n^n(B_n) < \epsilon$.
  Observe that monotonically as $n \to \infty$,
  \begin{align*}
    P_0^n(A_n \setminus A_m)
    \uparrow \Pb_0(A_\infty \setminus A_m)
    = \Pb_0(\tau < \infty, \tau > m).
  \end{align*}
  By continuity of measure,
  $\lim_{m \uparrow \infty} \Pb_0(A_\infty \setminus A_m) = 0$,
  so
  we may choose $m$ large enough that
  \begin{align*}
    P_0^n(A_n \setminus A_m) \le \Pb_0(A_\infty \setminus A_m) < \delta.
  \end{align*}
  Then $n \ge \max\{N, m\}$ implies
  \(Q_n^n(A_n \setminus A_m) < \epsilon\),
  so
  \begin{equation*}
    \lim_{n \to \infty} Q_n^n(A_n \setminus A_m) \le \epsilon.
  \end{equation*}
  Because $\epsilon > 0$ was arbitrary,
  we obtain that $\lim_{m \to \infty} \lim_{n \to \infty}
  Q_n^n(A_n \setminus A_m) = 0$, and so
  \begin{align*}
    \lim_{n \to \infty} Q_n^n(A_n) =
    \lim_{m \to \infty} \lim_{n \to \infty} Q_n^m(A_m)
    = \lim_{m \to \infty} P_0^m(A_m) = \Pb_0(A_\infty)
  \end{align*}
  as desired.
\end{proof}

\section{Experiments}\label{sec:experiments}

Our asymptotic theory makes concrete predictions about the behavior of the
proposed tests under regular local alternatives
with \(2nH^2(Q_{n,\kappa},\nullIid)\to\kappa\).
We design our experiments to test how quickly these asymptotic predictions kick
in for different types of underlying distributions.
We also compare the proposed statistics against existing statistics valid over
the i.i.d.\ null~\eqref{eq:iid-null}.
Appendix~\ref{sec:computation} details the (efficient) algorithms for
computing each of the statistics in this paper.

\subsection{Distributions and local alternatives}

\begin{figure}[htbp]
  \centering
  \begin{subfigure}{.5\linewidth}
    \includegraphics[width=\columnwidth]{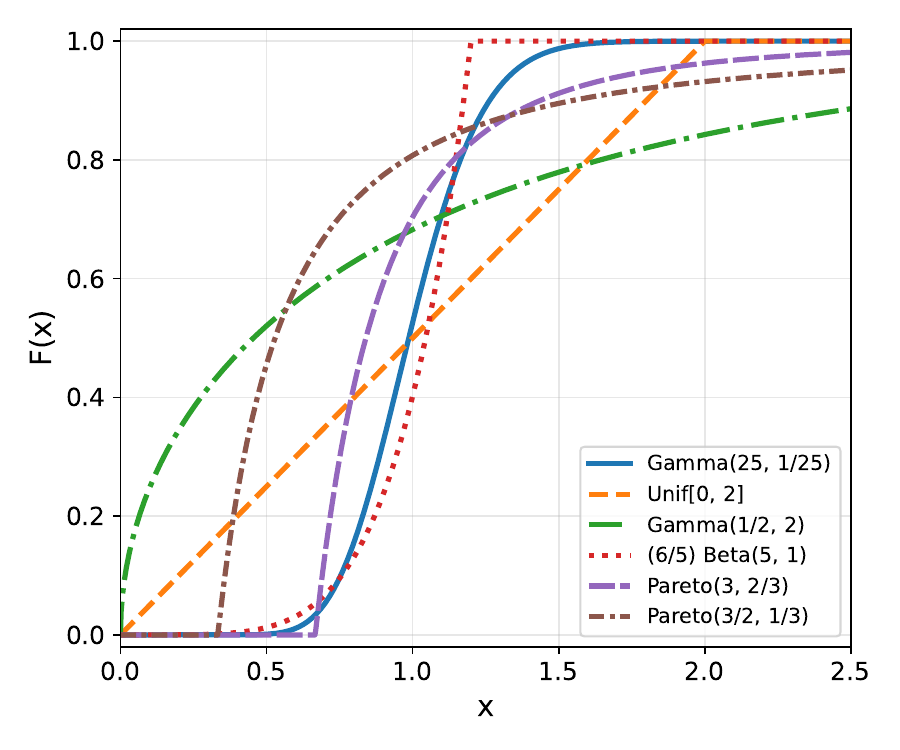}
    \caption{Cumulative distribution functions.}
  \end{subfigure}%
  \begin{subfigure}{.5\linewidth}
    \includegraphics[width=\columnwidth]{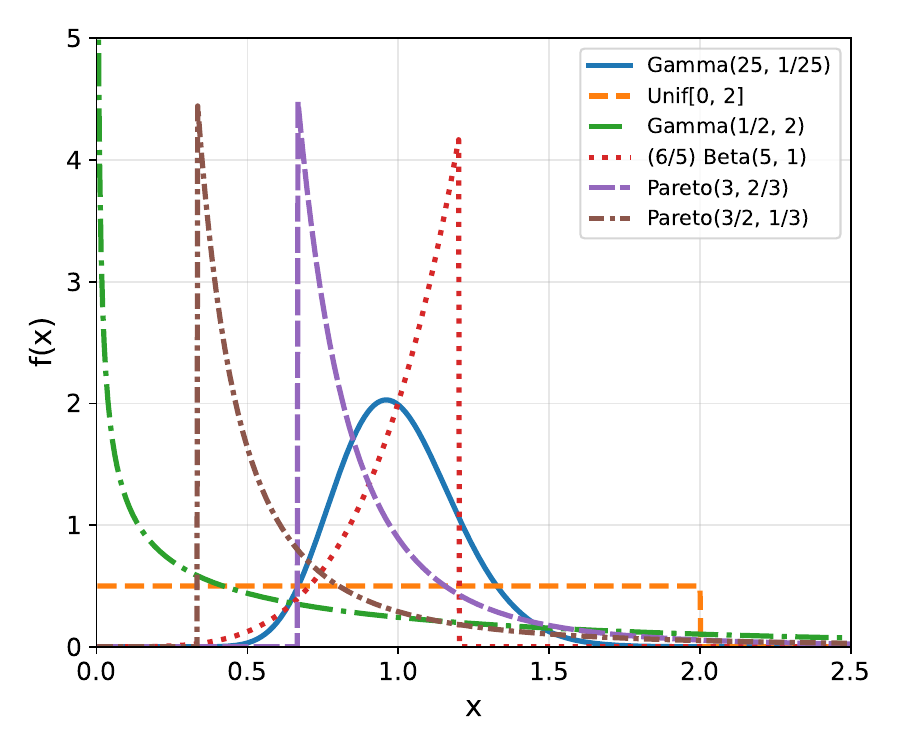}
    \caption{Probability density functions.}
  \end{subfigure}
  \caption{Distributions used in experiments.}
  \label{fig:distributions}
\end{figure}

We evaluate the competing \(p\)-values on a small collection of distributional
families that represent qualitatively different shapes of nonnegative data.
We normalize all null distributions to have mean one.  Specifically, we use the
following six baseline distributions:
\begin{align*}
\text{concentrated continuous:}\quad
& X_0 \sim \mathrm{Gamma}(25,1/25), \\
\text{evenly spread continuous:}\quad
& X_0 \sim \mathrm{Unif}[0,2], \\
\text{right-skewed finite-variance:}\quad
& X_0 \sim \mathrm{Gamma}(1/2,2), \\
\text{left-skewed:}\quad
& X_0 = \frac65 Y,\qquad Y\sim \mathrm{Beta}(5,1), \\
\text{heavy-tailed finite-variance:}\quad
& X_0 \sim \mathrm{Pareto}(3,2/3), \\
\text{infinite-variance finite-mean:}\quad
& X_0 \sim \mathrm{Pareto}(3/2,1/3).
\end{align*}
Here \(\mathrm{Gamma}(k,\theta)\) denotes the gamma distribution with shape
\(k\) and scale \(\theta\), and \(\mathrm{Pareto}(a,x_m)\) denotes the Pareto
distribution with tail index \(a\) and scale \(x_m\), so that
\(\Eb X_0 = a x_m/(a-1)\) when \(a>1\).  Thus each of the six choices
above satisfies \(\Eb X_0=1\).
In particular, we highlight the inclusion of the last distribution, which has finite mean but infinite variance.
\Cref{fig:distributions} shows the cdf and pdf of these distributions.

For each baseline distribution \(X_0\), we generate local alternatives by a multiplicative mean shift:
\begin{align*}
  X_{n,\kappa}
  :=
  \left(1+\frac{h}{n^{\gamma}}\right) X_0,
\end{align*}
and write $Q_{n,\kappa}$ to be the law of $X_{n,\kappa}$.
For $\kappa>0$, we choose the constants $h>0$ and $\gamma>0$ so that
\(2n H^2(Q_{n,\kappa},\nullIid) \longrightarrow \kappa.\)
For $\kappa=0$, we set $h=0$, recovering the null distribution.
For each of the five finite-variance baselines,
\Cref{proposition:regularly-varying-np-local} gives $\gamma=1/2$ and \( h=2\sigma\sqrt{\kappa}. \)
For $X_0\sim\mathrm{Pareto}(3/2,1/3)$,
\Cref{proposition:heavy-tail-np-local} gives $\gamma=1/3$ and
\(  h=\left(\frac{9\pi^2\kappa}{16}\right)^{1/3}.
\)

\subsection{Results}

\foreach \alphav/\alphat in {0-05/0.05} {
  \begin{figure}[htb]
      \centering
    \foreach \index in {0, 1} {
      \begin{subfigure}{\linewidth}
      \includegraphics[width=\columnwidth]{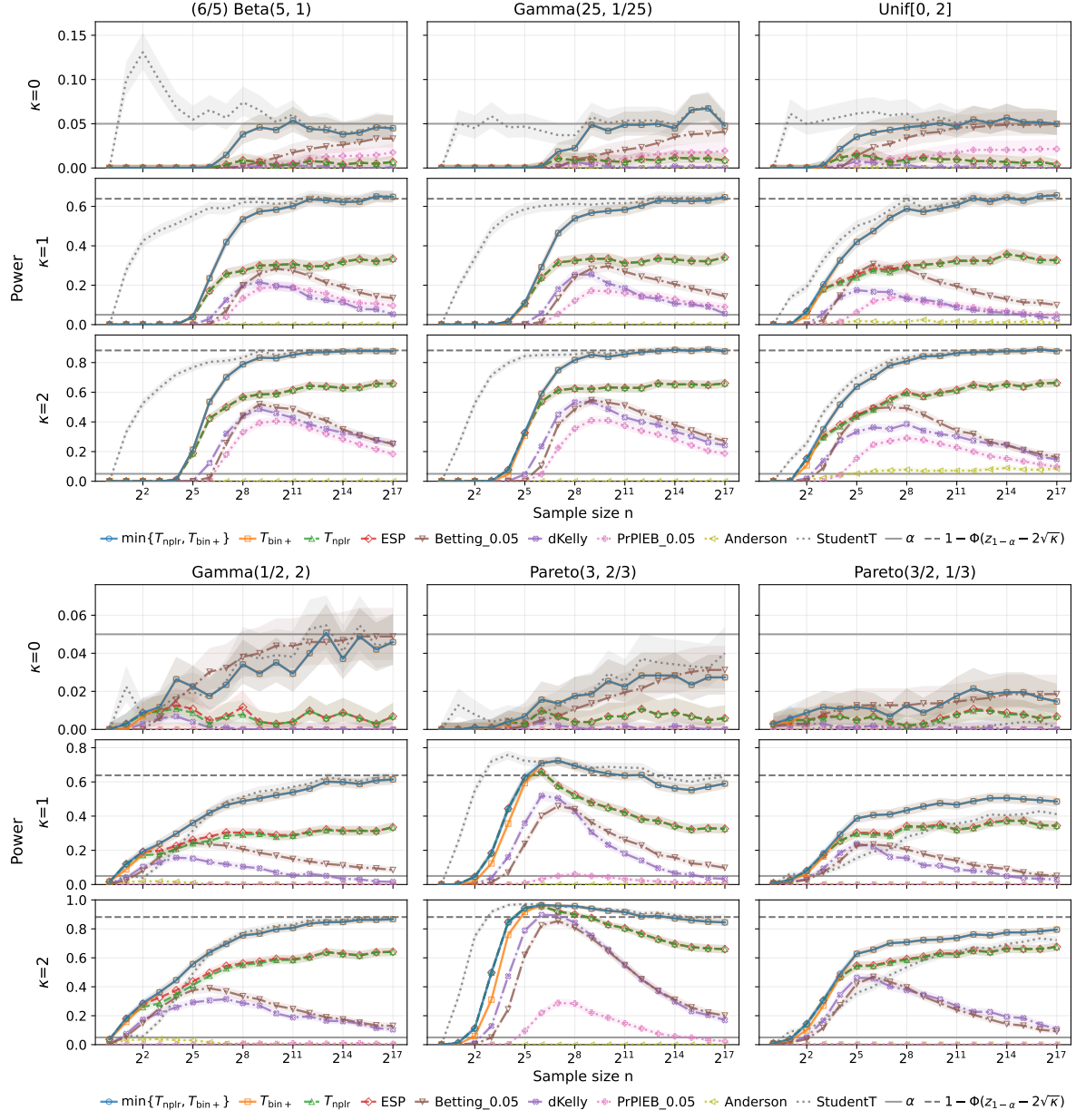}
      \end{subfigure}
    }
    \caption{
      Power plots of various statistics at $\alpha=\alphat$.
      Rows correspond to different $\kappa$, and 
      columns to different distributions.
    }
    \label{fig:level-\alphav-concise}
  \end{figure}
}

We show concise results in \Cref{fig:level-0-05-concise}.
Appendix~\ref{sec:additional-experiments} reports extensive results for more
values of $\kappa$ and $\alpha$.

We plot the power against sample size of various statistics at 
significance level \(\alpha=0.05\) and different strengths of the local alternatives \(\kappa\).
Different rows correspond to different values of \(\kappa\), and different columns
correspond to different distributions.
The shaded areas correspond to the 95\% confidence intervals of the power
estimates over 1024 runs.
We also plot reference lines corresponding to the nominal level \(\alpha\) and the
asymptotic power, \(1 - \Phi(z_{1-\alpha} - 2\sqrt{\kappa})\).
(This asymptotic power does not
apply to \(\mathrm{Pareto}(3/2,1/3)\), which has infinite variance.)

Aside from $\min\{\nplr,\pBinPlus\}$, $\pBinPlus$, and $\nplr$, we give brief
explanations of the other reported statistics in \Cref{tab:stats}.
Note that the studentized mean test is not valid for finite samples---we
include it only for comparison purposes.
Appendix~\ref{sec:existing-statistics} provides a more detailed discussion.

\begin{table}[htbp]
  \centering
  \caption{Baseline statistics included in experiments.}
  \begin{tabularx}{\textwidth}{l|
    >{\hsize=1.0\hsize\linewidth=\hsize\raggedright\arraybackslash}X|
    >{\hsize=1.0\hsize\linewidth=\hsize\raggedright\arraybackslash}X
    }
    \toprule
    Statistic & Description & Reference \\
    \midrule
    ESP & Elementary symmetric polynomial statistic & \citet[Theorem 1]{MingShWa26} \\
    Betting\_$\alpha$ & recommended betting statistic at level $\alpha$ & \citet[Thm. 3 and eq. (26)]{Waudby-SmithRa24} \\
    dKelly & diversified Kelly betting & \citet{Waudby-SmithRa24} \\
    PrPlEB\_$\alpha$ & plug-in empirical Bernstein at level $\alpha$ &
                                                                       \citet[equations (13) to (15)]{Waudby-SmithRa24} \\
    Anderson & CDF-band statistic based on DKW inequality & \citet{Anderson69,Massart90} \\
    \emph{StudentT}* & Studentized mean test, (\textbf{invalid}) &  e.g., see \citet{LehmannRo05} \\
    \bottomrule
  \end{tabularx}
  \label{tab:stats}
\end{table}

The case where $\kappa=0$ corresponds to the null, and we include it not for
comparing power but as a sanity check to ensure that all valid statistics have
power below the nominal level $\alpha=0.05$.
Not surprisingly, the studentized mean test is invalid.

For $\kappa>0$, across all distributions, $\pBinPlus$ and
$\min\{\nplr,\pBinPlus\}$ are substantially more powerful than the other
valid statistics, with \(\nplr\) and ESP often the closest competitors.
For the five finite-variance distributions, the power of $\pBinPlus$ and
$\min\{\nplr,\pBinPlus\}$ approaches the predicted asymptotic envelope as
\(n\) grows and is close to that of the invalid studentized mean test by
\(n=1000\).  In contrast, the power of \(\nplr\) remains below the envelope at
large sample sizes, consistent with its strictly smaller asymptotic power.
The Gaussian power benchmark does not apply to the infinite-variance Pareto
distribution, but the qualitative advantage of $\pBinPlus$ and
$\min\{\nplr,\pBinPlus\}$ persists in this setting.
In cases where $\nplr$ is better at small sample sizes and $\pBinPlus$ is better
at larger sample sizes, $\min\{\nplr,\pBinPlus\}$ dominates both, as expected.
Over the displayed range, the other valid statistics, including the betting
statistics, lose power as \(n\) grows and have little or no power at the largest
sample sizes considered.

Appendix~\ref{sec:additional-experiments} includes similar results for other
significance levels, \(\alpha=0.1\) and \(\alpha=0.01\), and other values of
\(\kappa\).

\section{Discussion and Conclusion}

We approach exact nonparametric mean testing under nonnegativity through
deterministic certificates rather than through an explicit least-favorable
null distribution.
The leave-one-out duality framework reduces finite-sample validity to pointwise
inequalities whose multiplier structure is compatible with the conditional mean
null.
This perspective proves the validity of the nonparametric likelihood-ratio
statistic of \citet{WangZh03}, yields the new generalized binomial
statistic \(\pBinPlus\), and gives a switching principle for combining
valid \(p\)-values without paying the usual cost of a union bound.
The resulting statistics are both theoretically and empirically powerful.

\paragraph{Extensions beyond mean testing.}
Is the dual-multiplier idea specific to nonparametric mean testing, or is it
more broadly applicable?
For monotone global-null tests valid under the product-uniform law, the same
framework constructs explicit dual certificates that establish validity under
every conditionally superuniform law, and thus yields an anytime-valid
extension, reproducing the results of 
\citep{KoningVanMeer26}.

Fix \(\alpha\in(0,1)\), and consider a global-null test
\(\varphi:[0,1]^n\to[0,1]\) satisfying
\begin{align*}
  \int_{[0,1]^n}\varphi(\bm{u})\,d\bm{u}\le\alpha.
\end{align*}
The test is based on sequentially revealed \(p\)-values \(P_1,\ldots,P_n\), and
the global null means that all component nulls hold.
Assume that \(\varphi\) is coordinatewise nonincreasing --- decreasing a
\(p\)-value cannot make rejection less likely --- and the \(p\)-values are
conditionally superuniform, meaning that, for all \(k\in[n]\)
and \(t\in[0,1]\), almost surely \(\Pb(P_k\le t\mid P_{1:k-1})\le t\), or
equivalently, \(\Eb\bigl[\one\{P_k\le t\}-t\mid P_{1:k-1}\bigr]\le0\).
Thus the family of centered thresholds \(\one\{P_k\le t\}-t\) is the global-null
analogue of the mean increment \(X_i-1\): predictable nonnegative mixtures of
these provide the dual certificate.
In other words, we seek predictable nonnegative multipliers \(\Lambda_k\)
such that
\begin{align*}
  \varphi(\bm{p})
  \le\alpha+\sum_{k=1}^n\int_0^1
  \bigl(\one\{p_k\le t\}-t\bigr)
  \Lambda_k(dt;\bm{p}_{1:k-1}).
\end{align*}
Immediately, we see that under every conditionally superuniform law \(Q\),
taking expectations in this inequality gives \(\Eb_Q[\varphi(\bm{P})]\le\alpha\).
We can also directly construct the multipliers by
\begin{align*}
  \varphi_k(\bm{p}_{1:k})
  &:=\int_{[0,1]^{n-k}}
  \varphi(\bm{p}_{1:k},\bm{u}_{k+1:n})\,d\bm{u}_{k+1:n},
  && k=0,\ldots,n,\\
  \Lambda_k([t,1];\bm{p}_{1:k-1})
  &:=\varphi_k(\bm{p}_{1:k-1},t^-)
  -\varphi_k(\bm{p}_{1:k-1},1),
  && k\in[n],\quad t\in[0,1],
\end{align*}
where \(t^-\) is the left limit.
Each section \(u\mapsto\varphi_k(\bm{p}_{1:k-1},u)\) is nonincreasing, so
the second line defines a nonnegative measure \(\Lambda_k\) that depends only
on the history \(\bm{p}_{1:k-1}\), and is therefore predictable.
Integrating against these measures and summing over \(k\) telescopes from
\(\varphi_0\le\alpha\) to \(\varphi_n=\varphi\), yielding the desired
certificate.
Under the product-uniform law,
\(\bigl(\varphi_k(\bm{P}_{1:k})\bigr)_{k=0}^n\) is the Doob martingale of
\(\varphi(\bm{P})\), recovering the induced-test construction of
\citet{KoningVanMeer26}.
Under every conditionally superuniform law \(Q\), monotonicity makes the same
process a supermartingale, so the extension implies an anytime-valid test; at
the terminal time, this also shows that the product-uniform law is least
favorable independent of the form of \(\varphi\).
Mean testing has no analogous universal least-favorable law, hence validity over
the iid null \(\nullIid^n\) does not imply validity over the larger
sequential class \(\mc{P}_{\textup{seq}}^n\).

\paragraph{Using duality to discover statistics.}
Is this duality framework just a clever method to prove the validity of
statistics found in other ways, or can it also guide the construction of new
statistics?
Although \Cref{sec:binplus} motivates the final form of \(\pBinPlus\) through the
exact binomial \(p\)-value, our discovery proceeded in the opposite direction:
we first determined the form of the dual multipliers and these in turn motivated
the statistic.

Recall that
\[
  \dual^{\lamBinPlus}_\alpha(\bm{x}) = \alpha+ \sum_{i=1}^n(x_i-1)\lamBinPlus_\alpha(\bm{x}_{-i})
  \ge \one\{\pBinPlus(\bm{x})\le\alpha\} .
\]
Since the statistic \(\pBinPlus\) is symmetric in its arguments, averaging the
multipliers over all permutations preserves the required inequalities and yields
a single multiplier symmetric in its arguments.
We may therefore assume without loss of generality that \(\lamBinPlus_\alpha\) is
symmetric.
Proceeding heuristically, suppose that \(\dual^{\lamBinPlus}_\alpha(\bm{x})=0\)
whenever \(\prod_i x_i=0\).
The resulting system of linear equations for \(\lamBinPlus_\alpha(\bm{z})\) can be solved recursively
in the number of nonzero coordinates of \(\bm{z}\),
yielding \(\lamBinPlus_\alpha(\bm{z})=\alpha \polyRightInt_1(0, \bm{z})\).
Similarly, if we instead suppose heuristically
that \(\dual^{\lamBinPlus}_\alpha(\bm{x})=1\) whenever the top two coordinates
of \(\bm{x}\) are equal and greater than $1$, we
obtain \( \lamBinPlus_\alpha(\bm{z})=(1-\alpha)\polyLeftInt_1(0, \bm{z})\).

Combining these two heuristics motivates the final form
of
\(\lamBinPlus_\alpha(\bm{z}) := \min \{\alpha \polyRightInt_1(0, \bm{z}), (1-\alpha)\polyLeftInt_1(0, \bm{z})\}\).
We then read the statistic from these multipliers by
taking \(\pBinPlus(\bm{x})\) to be the smallest level \(\alpha\) at
which \( (1-\alpha)\polyLeftInt_i(\bm{x}) \le \alpha \polyRightInt_i(\bm{x})\) for
every \(i\).
By \eqref{eq:def-rho-i}, this level is \(\max_i\rho_i(\bm{x})\), which
equals \(\pBinPlus(\bm{x})\) by \eqref{eq:binplus-max-ratio}.
In this way,  duality served as a design principle, not only as a proof
technique.

\section*{LLM Usage}

In addition to language editing, we used ChatGPT, Codex, and GitHub Copilot
throughout this project to check derivations, explore generalizations and proof
strategies, assist with code and numerical experiments, and search for related
literature.
The authors originated the generalized binomial statistic~$\pBinPlus$, the
leave-one-out dual formulation, the combination principle, the proof
for $\nplr$, and the main computational ideas without LLMs.
AI generated much of an initial proof of the validity of~$\pBinPlus$ and initial
proofs of the combination principle and
Proposition~\ref{proposition:binplus-nplr-comparison}; the authors checked and
substantially reworked these arguments.
The authors developed the (relatively standard) results in
\Cref{sec:asymptotics}, except for the initial proof of
\Cref{proposition:binplus-nplr-comparison}.
The authors independently verified every mathematical claim, numerical result,
and citation, and take responsibility for the article.

\FloatBarrier

\bibliography{fullbib}

\appendix

\section{Existing statistics}\label{sec:existing-statistics}

Existing finite-sample procedures for mean inference separate into two
qualitatively different groups.
The first obtains finite-sample guarantees by imposing bounded support, variance
control, or tail conditions.
These methods are important context, but they are not direct competitors for our
model without truncation or additional assumptions.
The second group uses only nonnegativity, and therefore forms the direct comparison class for this paper.
Within this second group, we further distinguish between i.i.d. fixed-sample
procedures, leave-one-out or simultaneous-valid procedures, and sequential
betting or e-value procedures.

Note that fully distribution-free inference for a mean without any structural
restriction is subject to classical impossibility phenomena of
\citet{BahadurSa56}.
In our setting, nonnegativity is the structural restriction that makes
finite-sample inference possible.

\subsection{Methods requiring bounded support, variance control, or tail control}

The following procedures rely on assumptions beyond nonnegativity, such as a known finite upper bound, a variance parameter, or a tail condition.
Our model imposes only $X_i\ge 0$ and a mean constraint, and therefore allows unbounded support and infinite variance.
Consequently, these procedures are not direct finite-sample competitors for our model class, although they are useful context for what is possible under stronger assumptions.

\paragraph{Classical concentration inequalities.}
Classical one-sided confidence bounds based on \citet{Hoeffding63},
\citet{Bernstein24}, \citet{Bentkus04}, and empirical
Bernstein bounds such as \citet{MaurerPo09} require
assumptions beyond nonnegativity.
The relevant assumptions may be a known bounded range, a known variance proxy,
or a moment/tail condition strong enough to control the upper tail.

\paragraph{\citet{RomanoWo00}}
\citet{RomanoWo00} construct finite-sample conservative and asymptotically efficient confidence intervals for the mean of a distribution supported on $[0,1]$.
They define their exact interval $I_{n,1}$ through sampling quantiles optimized
over a Kolmogorov--Smirnov confidence band around the empirical distribution, and
their computable conservative enlargements $I_{n,2}$ and $I_{n,3}$ use
Berry--Esseen bounds and worst-case variance bounds over that band, which
essentially relies on compactness of support.

\paragraph{\citet{AusternMa22}}
\citet{AusternMa22} construct finite-sample confidence intervals for bounded
i.i.d. means by adding explicit corrections to Gaussian tail bounds and then
inverting them.
The corrections vanish asymptotically, so the intervals attain the Gaussian
optimal width.
When the variance is unknown, they optimize the known-variance bound over a
high-probability confidence interval for the variance.
Both constructions rely on bounded support.

\subsection{Methods for nonnegative, potentially unbounded observations}

The methods in this subsection only use nonnegativity.
They are directly comparable to the methods we develop in this paper, and we include several of them in our experiments.
The main differences between these methods is the null class under which they are valid: some are valid under i.i.d. fixed-sample assumptions, some under the leave-one-out conditional-mean null, and some under the sequential conditional-mean null.

\subsubsection{I.i.d. fixed-sample methods}

\paragraph{\citet{Kaplan87}}
\citet{Kaplan87} constructs a finite-sample one-sided test for the mean of an i.i.d. nonnegative population by stopping multiplicative martingales for every permutation of the observations and applying Markov's inequality.
A simple symmetric version yields the p-value
\[
  T_{\mathrm{Kap}}(x)
  = \min\left\{1,\, \min_{1\le j\le n}
      \binom{n}{j}\prod_{k=1}^j x_{(n-k+1)}\right\},
\]
where $x_{(1)}\le \cdots\le x_{(n)}$ are the order statistics.
The ESP statistic of \citet{MingShWa26} strictly dominates this statistic, so we
do not include it in the experiments.

\paragraph{\citet{Anderson69}}
\citet{Anderson69} constructs fixed-sample confidence bounds for the mean of an
i.i.d. distribution by integrating an extremal cdf within a uniform
Dvoretzky--Kiefer--Wolfowitz band; we use the sharp Massart~\citep{Massart90}
constant for this band.  In our setting, the lower confidence bound at
level $\alpha\le 1/2$ takes the form
\[
  \widehat\mu_{A,\alpha}(x)
  := \sum_{i=1}^n
  \left(\frac{n-i+1}{n}-\varepsilon_{n,\alpha}\right)_+
  \bigl(x_{(i)}-x_{(i-1)}\bigr),
  \qquad
  \varepsilon_{n,\alpha}:=\sqrt{\frac{\log(1/\alpha)}{2n}},
\]
with $x_{(0)}:=0$.
We find the corresponding p-value by bisecting over $\alpha$.

\paragraph{\citet{PhanThLe21}}
\citet{PhanThLe21} construct finite-sample confidence bounds for the mean based
on general ordering of the observations, which they take to be the
centered $\ell_2$ ordering in their experiments.
While they formulate their statistic for bounded distributions,
it naturally extends to the nonnegative setting.
However, in our experiments this test had near-zero power for several choices of
the center of the $\ell_2$ ordering.
In addition, it was very computationally expensive, so we do not include it in
the plots.

\subsubsection{Likelihood-ratio and symmetric-polynomial methods}

\paragraph{The nonparametric likelihood-ratio.}
Building on earlier nonparametric likelihood-ratio and empirical-likelihood
work \citep{ThomasGr75,Owen88,Owen90}, \citet{WangZh03} study $\nplr$ for the
present problem, derive its closed-form one-dimensional representation, and
prove its validity for $n=2$.
\citet{Gaffke05} compare $\nplr$ to related statistics and prove additional
partial validity results, including validity for two-point distributions.
The present paper proves finite-sample validity of $\nplr$ under the leave-one-out conditional-mean null $\Eb[X_i\mid \bm{X}_{-i}]\le 1$.
This makes $\nplr$ a central existing comparator, and we include it in all experiments.

\paragraph{\citet{MingShWa26}.}
In concurrent independent work, \citet{MingShWa26} prove that $\pESP:=\min_k(\binom{n}{k}/e_k(x))$
is a valid $p$-value under the same leave-one-out conditional-mean null
considered here, where \(e_k(x)\) is the \(k\)th elementary symmetric polynomial of \(x\).
In particular, this implies the validity of the \(\nplr\) statistic.
While \citet{MingShWa26} note that evaluating the ESP statistic requires $O(n^2)$
time due to the need to consider all orders $k$, we show in
Appendix~\ref{sec:esp} that it suffices to
check only two values of $k$, and using tilted Fourier transforms we can compute
the statistic in $O(n\log^2 n)$ time.
We include \(\pESP\) in our experiments, and find that its
performance is almost the same as that of \(\nplr\).

\subsubsection{Sequential betting and e-value methods}

Sequential betting methods are valid under the larger null class $\Eb[X_i\mid \bm{X}_{1:i-1}]\le 1$.
This makes them broadly valid and natural competitors, but also means that they
do not exploit the additional leave-one-out structure available
to $\nplr$, $\pBinPlus$, and $\min\{\nplr,\pBinPlus\}$.
For background on time-uniform confidence sequences and e-values, see
\citet{RamdasWa25}.

\citet{Waudby-SmithRa24} construct any-time valid confidence sequences of the
mean of bounded distributions by constructing test supermartingales.
In particular, if $M_t$ is a nonnegative supermartingale with $M_0=1$ under the
null, then by Ville's maximal inequality, \( \Pb\{\max_{t\ge0}M_t\ge 1/\alpha\}\le \alpha\).
Hence \(\{\max_{s \le t}M_s\ge 1/\alpha\}\) is a valid rejection region for the null at
level \(\alpha\) and \(\min\{1,1/\max_{s \le t}M_s\}\) is a valid \(p\)-value.

Many of the test supermartingales constructed in \citet{Waudby-SmithRa24} are
valid under our nonnegative setting, and we include three of them in our
experiments: the betting capital process (Betting), diversified Kelly betting
(dKelly), and the predictable plug-in empirical Bernstein (PrPlEB).

\paragraph{Betting capital process.}
In the present one-sided nonnegative setting, the betting capital process is
\[
  K_t:=\prod_{i=1}^t\{1+\lambda_i(X_i-1)\},
\]
where $\lambda_i\in[0,1]$ is predictable.
We include the recommended level-dependent betting statistic, denoted Betting$_\alpha$, using the tuning
\[
  \lambda_t=\widetilde\lambda_t\wedge \frac12,
  \qquad
  \widetilde\lambda_t=
  \sqrt{\frac{2\log(1/\alpha)}{\widehat\sigma_{t-1}^2\,t\log(t+1)}},
\]
with
\[
  \widehat\sigma_t^2
  =\frac{1+\sum_{i=1}^t(X_i-\widehat\mu_i)^2}{t+1},
  \qquad
  \widehat\mu_t=\frac{1+\sum_{i=1}^tX_i}{t+1}.
\]

\paragraph{Diversified Kelly betting}
For diversified Kelly betting, rather than averaging over finitely many
prespecified strategies, we use the continuous mixture
\[
  K_t:=\int_0^1\prod_{i=1}^t\{1+\lambda(X_i-1)\}\,d\lambda,
\]
where we efficiently evaluate the integral using the log-concave approximation
methods described in \Cref{sec:computing-binplus-integrals}.
Using a fixed grid of $\lambda$ values is undesirable for large $n$, because in regular local regimes the optimal $\lambda$ is typically of order $n^{-1/2}$.
For computational efficiency, the plotted dKelly statistic uses the final value
at $t=n$.

\paragraph{Predictable plug-in empirical Bernstein}
\citet[Theorem~2]{Waudby-SmithRa24} present their predictable plug-in empirical
Bernstein confidence sequence for $X_t\in[0,1]$, but its
lower-bound component extends to nonnegative, potentially unbounded
observations.
Under $H_0:\Eb[X_t\mid \bm{X}_{1:t-1}]\le 1$, let $q_{t-1}\in[0,1]$ and $\lambda_t\in[0,1)$ be predictable and set
\[
  \psi_E(\lambda):=-\log(1-\lambda)-\lambda,
\]
\[
  M_t^{\rm PrPl\text{-}EB}
  =\prod_{i=1}^t
  \exp\left\{\lambda_i(X_i-1)-4(X_i-q_{i-1})^2\psi_E(\lambda_i)\right\}.
\]
Then $(M_t^{\rm PrPl\text{-}EB})_{t\ge 0}$ is a test supermartingale under the null.
We use the predictable choices
\[
  q_t=\min\left\{1,\frac{1+\sum_{j=1}^tX_j}{t+1}\right\},
  \qquad
  \widehat\sigma_t^2=\frac{1+\sum_{j=1}^t(X_j-q_j)^2}{t+1},
\]
with $q_0=1$, $\widehat\sigma_0^2=1$, and
\[
  \lambda_t=\min\left\{\frac12,
  \sqrt{\frac{2\log(1/\alpha)}{\widehat\sigma_{t-1}^2\,t\log(1+t)}}\right\}.
\]
Clipping $q_t$ at one is necessary in the unbounded setting to ensure $X_t-q_{t-1}\ge -1$.
The construction follows from the lower-increment empirical-Bernstein inequality
used in \citet[Lemma~1]{Waudby-SmithWuRaKaMi24} and \citet[proof of
Lemma~4.1]{FanGrLi15}.

\section{Additional Experiments}\label{sec:additional-experiments}

We show detailed results in \Cref{fig:level-0-1,fig:level-0-05,fig:level-0-01}.
These additional experiments support the claims in the main text.

\foreach \alphav/\alphat in {0-1/0.1, 0-05/0.05, 0-01/0.01} {
  \begin{figure}[htb]
    \centering
            \includegraphics[width=\columnwidth]{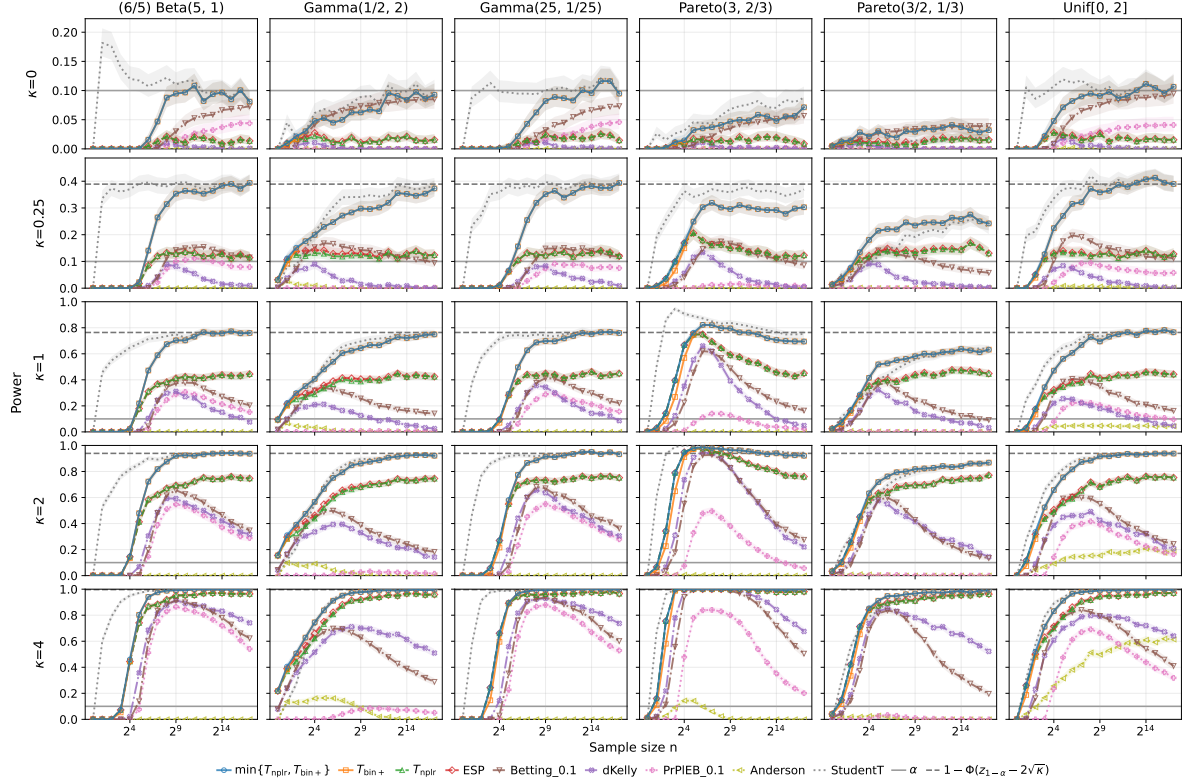}
    \caption{
      Extended power plots of various statistics at $\alpha=\alphat$.
      Different rows correspond to different $\kappa$, and different
      columns correspond to different distributions.
    }
    \label{fig:level-\alphav}
  \end{figure}
}

\section{Computation}
\label{sec:computation}

In this section we describe efficient algorithms for computing the various
statistics.

\subsection{Computing $\nplr$}

The nonparametric likelihood ratio $\nplr$ in
definition~\eqref{eqn:nplr} requires maximizing the product
$\prodPoly_{\bx}(t) = \prod_{i=1}^n (1 + t(x_i - 1))$ over $t \in [0, 1]$,
which is trivial by bisection, as $\log \prodPoly_{\bx}(t)$ is concave in $t$.

\subsection{Computing \(\pBinPlus\)}
\label{sec:computing-binplus-integrals}

We first recall the definition~\eqref{eq:def-Pvee} of \( \pBinPlus \).
For \(\bm{x}\in\Rb_+^n\) let \(k\in\arg\max_i x_i\), and define
\(\pBinPlus(\bm{x})=1\) when \(x_k\le 1\) and
\[
  \pBinPlus(\bm{x})
  =
  \frac{\int_{-1/(x_k-1)}^{0} \prodPoly_{\bm{x}_{-k}}(t)\,dt}
       {\int_{-1/(x_k-1)}^{1} \prodPoly_{\bm{x}_{-k}}(t)\,dt}
\]
otherwise.
This requires evaluating two polynomial integrals.

We mention four approaches in passing, which we do not use, before
describing the procedure.
Naively expanding the integrand $\prodPoly_{\bx}$ in the monomial basis and
integrating term by term is numerically unstable even for \(n\) in the
hundreds: while the integrand is nonnegative on its domain, the monomial
coefficients can have large alternating signs.
Gauss-Legendre quadrature with $q = \ceil{n/2}$ nodes is exact as the
integrands $\prodPoly_{\bx_{-k}}(t)$ have degree at most $n - 1$, is
numerically stable, and requires time $O(n^2)$; recursive Bernstein basis
expansions also yield a stable $O(n^2)$ procedure~\citep{Lorentz86}.
An  FFT-based  polynomial  multipication  scheme~\citep{CormenLeRiSt01}  can
reduce complexity to $\widetilde{O}(n)$, but in our experiments this becomes
numerically unstable around $n = 100$ because of severe cancellations.

Instead, we use an adaptive-rejection-sampling style envelope construction
\citep{GilksWi92} to approximate the integral.
The integrand $\prodPoly_{\bm{x}_{-k}}(t)$ is log-concave on the interval
$[-1/(x_k - 1), 1]$, so that $h(t) = \log \prodPoly_{\bx_{-k}}(t)$ is
concave; its secants thus provide lower bounds and its first-order
approximations provide upper bounds.
As both bounds are piecewise affine, we can compute integrals of their
exponentials in closed form.
To compute either integral, we partition the domain into intervals,
compute the secant-based lower integral and tangent-based upper integral
(with secants and tangents computed in each interval), yielding
global lower and upper integral bounds.
We bisect the interval contributing the largest gap until the global upper
and lower bounds are within a specified relative tolerance; we maintain a
max-heap~\citep{CormenLeRiSt01} to make each interval selection logarithmic
in the number of active cells.

We provide no formal complexity guarantees for this procedure, noting simply
that with \(N_{\mathrm{bis}}\) bisections, evaluating the bounds on the
intervals takes \(O(nN_{\mathrm{bis}})\), because each split requires only a
constant number of evaluations of \(h\) and \(h'\), and each requires a
single pass through the \(n-1\) factors.
The maximum heap operations add \(O(N_{\mathrm{bis}}\log N_{\mathrm{bis}})\),
giving total
runtime
\(O\left(nN_{\mathrm{bis}}+N_{\mathrm{bis}}\log N_{\mathrm{bis}}\right).\)
In our experiments, we observe that \(N_{\mathrm{bis}}<1000\) at relative
tolerance \(10^{-5}\), even for \(n\) up to \(10^5\).
Thus, for fixed accuracy, the adaptive routine can be substantially faster than
the exact \(O(n^2)\) methods for large \(n\).

\subsection{Computing \(\pESP\)}\label{sec:esp}

We discuss efficiently and stably computing \citeauthor{MingShWa26}'s
normalized elementary symmetric polynomial $\pESP$ statistic
[\citeyear{MingShWa26}].
Let \(e_k\) and \(A_k\) be the elementary symmetric polynomials and their
normalized versions: for \( \bm{x} \in \Rb^n \),
\[
e_k(\bm{x})=\sum_{\substack{S\subseteq[n]\\ |S|=k}}\prod_{i\in S}x_i, \qquad
A_k(\bm{x})=\frac{e_k(\bm{x})}{\binom nk}, \qquad A_0(\bm{x})=1.
\]
Then, for \( \bm{x} \in \Rb_{+}^n \),
\begin{equation}
  \label{eqn:esp-stat}
  \pESP (\bm{x}) \defeq \left(\max_{0\le k\le n} A_k(\bm{x})\right)^{-1}.
\end{equation}
\citet[Sec.~6]{MingShWa26} use an FFT-based product tree to compute each
\(e_k\) in \(O(n\log^2 n)\) arithmetic operations.
In our experiments, however, the floating-point implementation of this
procedure became numerically unreliable for \(n\) in the hundreds when the
coefficients spanned many orders of magnitude.
We therefore develop an alternative method that considers only two indices
to compute the maximum in~\eqref{eqn:esp-stat}, and then use FFT-based convolution
after exponential tilting to compute the corresponding \(A_k\) values.

\subsubsection{Finding the maximizing index of $A_k$}

Expanding \( \prodPoly_{\bm{x}}(t)\) in the
Bernstein basis and using the definition of $A_k$ yields
\[
\prodPoly_{\bm{x}}(t)
= \prod_{i=1}^n (1 + t(x_i - 1))
= \sum_{k=0}^n \binom nk  t^k(1-t)^{n-k} A_k(\bm{x}),
\]
so that for $K \sim \binomial(n,t)$, we obtain \(\prodPoly_{\bm{x}}(t) =
\Eb[A_K(\bm{x})].\)
That $K$ is typically near its mean $nt$ gives the heuristic approximation
\[
  \prodPoly_{\bm{x}}(t) \approx A_{ nt }(\bm{x})
\]
(where we assume $nt$ is integer), so if \(t^\star\) maximizes
\(\prodPoly_{\bm{x}}(t)\), we might hope that either \(k=\lfloor
nt^\star\rfloor\) or \(k=\lceil nt^\star\rceil\) maximizes \(A_k(\bm{x})\).
This indeed holds:
\begin{proposition}
  \label{proposition:AK}
  Let $\bm{x} \in \Rb_+^n$ and
  $t^\star\in\argmax_{0\le t\le1}\prodPoly_{\bm{x}}(t)$.
  Then
  \begin{equation*}
    \max_{0\le k\le n}A_k(\bm{x})
    =
    \max\left\{
    A_{\lfloor nt^\star\rfloor}(\bm{x}),
    A_{\lceil nt^\star\rceil}(\bm{x})
    \right\}.
  \end{equation*}
\end{proposition}
\begin{proof}
  We show the monotonicity results that
  \begin{subequations}
    \label{eqn:index-implications}
    \begin{align}
      \label{eqn:index-implication-a}
      k \ge n t^\star & ~~ \mbox{implies}~~
      A_{k+1}(\bx) \le A_k(\bx) \\
      k + 1 \le n t^\star
      & ~~\mbox{implies~~} A_{k}(\bx) \le A_{k + 1}(\bx)
      \label{eqn:index-implication-b}
    \end{align}
  \end{subequations}
  for $0 \le k \le n$ (where we ignore the out of bounds indices); this
  monotonicity clearly implies the result.
  To develop the inequality, we relate indices to $t\opt$ by
  taking derivatives of $\log \prodPoly_{\bx}(t)$,
  which for $q_i(t) \defeq \frac{t x_i}{1 + t (x_i - 1)} \in [0, 1]$ satisfy
  \begin{align}
    \frac{\partial}{\partial t} \log \prodPoly_{\bx}(t)
    & = \sum_{i = 1}^n \frac{x_i - 1}{1 + t(x_i - 1)}
    = \sum_{i = 1}^n \left(\frac{q_i(t)}{t}
    - \frac{1 - q_i(t)}{1 - t}\right)
    = \frac{1}{t(1 - t)}
    \sum_{i = 1}^n \left(q_i(t) - t \right).
    \label{eqn:derivative-log-E-mean}
  \end{align}
  Immediately, we see that if $\sum_{i = 1}^n q_i(t) > nt$, then
  concavity implies $t\opt > t$, while if
  $\sum_{i = 1}^n q_i(t) < nt$, then $t\opt < t$.

  To relate this to the polynomials, consider
  the random variable $K_t$ defined to be the number of successes
  in $n$ independent Bernoulli trials with success probabilities
  $q_i(t)$ (a Poisson-Binomial random variable), which has probability mass function
  \begin{align*}
    f_t(k) & = \sum_{\substack{S \subset [n] \\ |S| = k}}
    \prod_{i \in S} q_i(t) \prod_{i \not \in S} (1 - q_i(t))
    = \sum_{\substack{S \subset [n] \\ |S| = k}}
    \prod_{i \in S} (t x_i) \prod_{i \not \in S} (1 - t)
    \cdot \prod_{i = 1}^n \frac{1}{1 + t (x_i - 1)} \\
    & = \frac{1}{\prodPoly_{\bx}(t)} t^k (1 - t)^{n - k}
    \sum_{\substack{S \subset [n] \\ |S| = k}} \prod_{i \in S} x_i
    = \frac{1}{\prodPoly_{\bx}(t)} t^k (1 - t)^{n-k} e_k(\bx),
  \end{align*}
  where we recall the $k$th elementary symmetric polynomial.
  Normalizing, we equivalently write
  \begin{align*}
    f_t(k) = \frac{\binom{n}{k}}{\prodPoly_{\bx}(t)} t^k(1 - t)^{n-k}
    A_k(\bx)
    ~~ \mbox{or} ~~
    \frac{A_k(\bx)}{\prodPoly_{\bx}(t)}
    = \frac{f_t(k)}{\binom{n}{k} t^k (1 - t)^{n-k}}.
  \end{align*}

  To represent the difference \( A_{k + 1}(\bx) - A_k(\bx)\) in terms of these
  probabilities, for $t \in (0, 1)$, we have
  \begin{align*}
    \frac{A_{k+1}(\bx) - A_k(\bx)}{\prodPoly_{\bx}(t)}
    & = \frac{1 - t}{t} \frac{f_t(k + 1)}{\binom{n}{k + 1}}
    t^{-k} (1 - t)^{k - n}
    - \frac{f_t(k)}{\binom{n}{k}} t^{-k} (1 - t)^{k - n} \\
    & = \frac{1}{\binom{n}{k} t^k(1 - t)^{n-k}}
    \cdot \left[\frac{1 - t}{t} 
      \frac{k + 1}{n - k} f_t(k + 1) - f_t(k) \right]
  \end{align*}
  because $\binom{n}{k+1} = \binom{n}{k} \frac{n-k}{k + 1}$.
  To cancel factors, take $\hat{t} = \frac{k}{n}$, whence
  $(1 - \hat{t}) / \hat{t} = \frac{n - k}{k}$, so
  \begin{align}
    \label{eq:esp-adjacent-ft}
    \frac{A_{k+1}(\bx) - A_k(\bx)}{\prodPoly_{\bx}(\hat{t})}
    = \frac{1}{\binom{n}{k} \hat{t}^k (1 - \hat{t})^{n-k}}
    \cdot \left[\frac{k + 1}{k} f_{\hat{t}}(k + 1) - f_{\hat{t}}(k)\right].
  \end{align}
  
  Now we leverage a result on the monotonicity properties of the probability
  mass functions of Poisson-Binomial random variables that
  \citet{DumbgenWe20} develop:
  \begin{lemma}
    \label{lem:dw-score}
    Let $K$ be a Poisson-binomial random variable with probability mass
    function $f$ and mean $\mu$.
    If $j \ge \mu$, then $(j+1)f(j+1)\le\mu f(j) \le j f(j)$.
  \end{lemma}

  \begin{proof}
    The second inequality is obvious, and we prove the first.
    When every Bernoulli success probability is strictly less than one, this
    follows from \citet[Equation~(7)]{DumbgenWe20}.
    If a success probability of $K$ equals one, the result follows by
    approximation.
  \end{proof}

  The expansion~\eqref{eq:esp-adjacent-ft} shows that $A_{k+1}(\bm{x})
  > A_k(\bm{x})$ if and only if $(k+1) f_{\hat{t}} (k+1) > k
  f_{\hat{t}}(k)$.
  Hence, by the contrapositive of \Cref{lem:dw-score},
  if $A_{k+1}(\bx) > A_k(\bx)$ then 
  \[
    k < \mu_{\hat{t}} := \Eb K_{\hat{t}} = \sum_{i = 1}^n q_i(\hat{t})
    .
  \]
  By the discussion following equality~\eqref{eqn:derivative-log-E-mean}, we see
  that $\mu_{\hat{t}} = \sum_{i = 1}^n q_i(\hat{t}) > k = n \hat{t}$
  implies $\hat{t} < t\opt$.
  To sum up, we have shown that if $A_{k+1}(\bx) > A_k(\bx)$,
  then \(\hat{t} < t\opt\), or equivalently, $k < n t\opt$.
  Taking contrapositives, we see that
  \begin{align*}
    k \ge n t\opt ~~ \mbox{implies} ~~
    A_k(\bx) \ge A_{k+1}(\bx)
    ~~ \mbox{for~} 1 \le k \le n-1.
  \end{align*}

  To argue the converse direction, apply the preceding argument to the
  reflected variable $\wt{K}_t = n - K_t$, which is a Poisson-binomial
  random variable with probabilities $\wt{q}_i(t) = 1 - q_i(t)$ and mean
  $\Eb \wt{K}_t = n - \mu_t$.
  Then repeating the same argument, \emph{mutatis mutandis}
  (replace $k$ by $n - k - 1$ and $\hat{t} = \frac{k + 1}{n}$),
  we obtain $A_{k + 1}(\bx) < A_k(\bx)$ implies $k + 1 > n t\opt$,
  whose contrapositive is that
  $k + 1 \le n t\opt$ implies $A_{k+1}(\bx) \ge A_k(\bx)$.
  
  There remain two cases to prove the
  inequalities~\eqref{eqn:index-implications}: $k=0$ in the
  implication~\eqref{eqn:index-implication-a} and $k=n-1$
  in~\eqref{eqn:index-implication-b}.
  These can occur only when $t^\star=0$ and $t^\star=1$,
  respectively.
  By concavity,
  if \(t^\star=0\), then
  \[
  0 \ge \prodPoly_{\bm{x}}'(0) = n(A_1(\bm{x})-A_0(\bm{x})),
  \]
  while if \(t^\star=1\), then
  \[
  0 \le \prodPoly_{\bm{x}}'(1) = n(A_n (\bm{x})-A_{n-1}(\bm{x})).
  \]
  The inequalities~\eqref{eqn:index-implications} follow.
\end{proof}

\subsubsection{Computation by exponentially tilted Fourier transforms}

Proposition~\ref{proposition:AK} reduces the discrete optimization over
$n+1$ values to at most two coefficient evaluations.
We compute the relevant coefficients using the exponentially shifted
FFT-convolution method of \citet{PeresLeKe21}, i.e., to obtain $A_k(\bm{x})$, we
use FFT convolution to compute the coefficients of $z^k$ in the polynomial
\begin{align*}
  \prod_{i=1}^n(1 + r x_i z)
  =\prod_{i=1}^n(1 + r x_i)
  \cdot \prod_{i=1}^n (1 - q_i + q_i z)
  = \sum_{j=0}^n \binom{n}{j} A_j(\bm{x}) r^j z^j,
\end{align*}
where
\[
  q_i=\frac{r x_i}{1+r x_i},
  \qquad i=1,\ldots,n.
\]
For numerical stability, \citet{PeresLeKe21} choose \(r\)
satisfying \(\sum_i q_i=k\) when computing the coefficient of $z^k$.
When $t^{\star} \in (0, 1)$, we take \(r=t^\star/(1-t^\star)\) for simplicity, for
which \(\sum_i q_i=nt^\star\) is approximately equal to the index of the maximizing
coefficient.
For \(t^\star\in\{0,1\}\), we compute the corresponding coefficient directly.
Since $\prodPoly_{\bm{x}}$ is log-concave, we can find $t^\star$ by bisection, and the overall
run time is $O(n\log^2 n)$.
This procedure was numerically stable in our experiments.

\section{Proofs and additional results for Section~\ref{sec:asymptotics}}

\subsection{Proof of Lemma~\ref{lem:hellinger-projection}}
\label{sec:proof-hellinger-projection}

We begin with an auxiliary lemma on the convexity properties of the affinity
functional $A_Q(t) = \Eb_Q[\frac{1}{1 + t(X - 1)}]$.

\begin{lemma}
  \label{lemma:properties-affinity}
  Let \(Q\) be a probability law on \(\Rb_+\). Then
  \begin{enumerate}[label=(\roman*),leftmargin=*]
  \item \label{item:inverse-moment-endpoint-limit}
    \(A_Q(0)=1\), \(0<A_Q(t)<\infty\) for \(0\le t<1\), and
    \begin{align*}
      \lim_{t\uparrow1}A_Q(t)=A_Q(1)
    \end{align*}
    in the extended real line.
    Consequently, \(A_Q\) attains its minimum on \([0,1]\)
    and \(A_Q^\star:=\min_{0\le t\le1}A_Q(t)\) satisfies \(0<A_Q^\star\le1\).
  \item \label{item:inverse-moment-derivative}
    The function \(A_Q\) is differentiable on \((0,1)\) and has an
    extended right derivative at zero.
    Using \(A_Q'(t)\) to denote the right derivative,
    for \(0 \le t < 1\)
    \begin{align*}
      A_Q'(t)
      =
      -\Eb_Q\!\left[
        \frac{X-1}{\{1+t(X-1)\}^2}
        \right].
    \end{align*}
  \end{enumerate}
\end{lemma}

\begin{proof}
  For \(x\ge0\), define the closed convex functions
  $\phi_x(t) = \frac{1}{1 + t(x - 1)}$,
  $t \in [0, 1]$.
  Then $A_Q(t)=\Eb_Q\phi_X(t)$
  is a proper lower-semicontinuous convex function on \([0,1]\); see
  \citet[Lemma~2.1]{Bertsekas73} for preservation of convexity under
  expectation, while lower semicontinuity follows from Fatou's lemma.
  Moreover,
  \begin{equation*}
  A_Q(0)=1,
  \qquad
  0<A_Q(t)\le\frac{1}{1-t}<\infty
  \quad (0\le t<1).
  \end{equation*}

  Continuity of closed convex
  functions on $\Rb$~\citep[Ch.~I.3.2]{HiriartUrrutyLe93} gives
  \begin{equation*}
  \lim_{t\uparrow1}A_Q(t)=A_Q(1)
  \end{equation*}
  in the extended real line.
  Since \(A_Q\) is lower semicontinuous on \([0,1]\)
  it attains its infimum, and \(A_Q(0)=1\) together with \(\phi_X(t)>0\) for
  every \(t\) implies \(0<A_Q^\star\le1\).

  For $0 < t < 1$,
  \citet[Proposition~2.1]{Bertsekas73} implies that
  \begin{equation*}
  A_Q'(t)
  =
  \Eb_Q\!\left[\phi_X'(t)\right]
  =
  -\Eb_Q\!\left[
    \frac{X-1}{\{1+t(X-1)\}^2}
    \right].
  \end{equation*}
  To compute the right derivative $A_Q'(0)$,
  observe that
  $A_Q'(0) = \lim_{t \downarrow 0}
  \frac{1}{t} (A_Q(t) - A_Q(0))$~\citep[Prop.~I.4.1.3]{HiriartUrrutyLe93}.
  The criterion of
  increasing slopes~\citep[Prop.~I.1.1.4]{HiriartUrrutyLe93}
  implies $t \mapsto \frac{1}{t}(\phi_x(t) - \phi_x(0))$ is
  monotonically non-decreasing in $t$, so
  \begin{align*}
    \frac{1}{t} (\phi_x(t) - \phi_x(0))
    = \frac{1 - x}{1 + t(x - 1)}
    \downarrow 1 - x
    ~~ \mbox{as} ~~ t \downarrow 0.
  \end{align*}
  Because
  \begin{align*}
    \frac{1}{t} (\phi_x(t) - \phi_x(0)) - 2
    = \frac{1 - x - 2(1 + t(x - 1))}{1 + t(x - 1)}
    = \frac{2 t - 1 - (2 t + 1) x}{1 + t(x - 1)} \le 0
  \end{align*}
  for $0 \le t \le \frac{1}{2}$ and $x \ge 0$, monotone convergence
  implies $A_Q'(0) = 1 - \Eb_Q[X]$.
\end{proof}

We now return to prove Lemma~\ref{lem:hellinger-projection} proper
using the preceding lemma.
Let \(P\in\nullbasic\) and \(0 \le t < 1\), so
\(1+t(x-1)>0\) for \(x\ge0\).
Cauchy-Schwarz gives
\begin{align*}
  1 - H^2(Q, P)
  & =
  \int \sqrt{dQ dP} =
  \int\sqrt{
    \left(\frac{dQ}{1+t(x-1)}\right)
    \bigl((1+t(x-1))\,dP\bigr)
  } \\
  &\le
  \left(\int \frac{dQ}{1+t(x-1)}\right)^{1/2}
  \left(\int \{1+t(x-1)\}\,dP\right)^{1/2}
  \\
  &=
  \sqrt{A_Q(t)} \sqrt{\Eb_P[1+t(X-1)]}
  \le \sqrt{A_Q(t)}
\end{align*}
because \(P \in\nullbasic\) satisfies $\Eb_P[X] \le 1$,
so $H^2(Q, P) \ge 1 - \sqrt{A_Q(t)}$ for $0 \le t < 1$.
By part~\ref{item:inverse-moment-endpoint-limit} of
Lemma~\ref{lemma:properties-affinity}, \(\inf_{0\le t<1}A_Q(t)=A_Q^\star\),
so we obtain
\begin{equation*}
  H^2(Q,\nullbasic) \ge 1 - \inf_{0 \le t < 1} \sqrt{A_Q(t)}
  = 1 - \sqrt{A_Q^\star}.
\end{equation*}
The equality conditions for Cauchy--Schwarz suggest that to attain this
bound, we should choose \(P\) such that
\begin{equation*}
  \frac{dQ}{1+t(X-1)} \propto \{1+t(X-1)\}\,dP,
  ~~ \mbox{i.e.} ~~
  dP \propto \frac{dQ}{\{1+t(X-1)\}^2}.
\end{equation*}
We make this precise.

By Lemma~\ref{lemma:properties-affinity},
$t_Q = \argmin_{t \in [0, 1]} A_Q(t)$ exists.
First suppose that \(t:=t_Q<1\), so that
$A_Q'(t) \ge 0$ by optimality.
If $t = 0$, then $A_Q'(0) = \Eb_Q[1 - X] \ge 0$, that is,
$\Eb_Q[X] \le 1$ and $Q \in \nullbasic$, proving
Lemma~\ref{lem:hellinger-projection} in this case as $A_Q(0) = 1$.
When $0 < t < 1$, \(A_Q'(t)=0\),
and we observe that
\begin{align*}
  \Eb_Q\frac{1}{\{1+t(X-1)\}^2}
  = 
  \Eb_Q\frac{1 + t(X-1)}{\{1+t(X-1)\}^2} - \Eb_Q\frac{t (X-1)}{\{1+t(X-1)\}^2}
  = A_Q(t) + tA_Q'(t) = A_Q^\star,
\end{align*}
where we use the form of $A_Q'(t)$ from
Lemma~\ref{lemma:properties-affinity}.\ref{item:inverse-moment-derivative}.
Therefore
\begin{equation*}
  dP_Q^\star(x)
  :=
  \frac{1}{A_Q^\star} \cdot \frac{dQ(x)}{\{1+t(x-1)\}^2}
\end{equation*}
defines a probability measure on $\Rb_+$, with mean satisfying
\begin{equation*}
  \Eb_{P_Q^\star}[X-1]
  = 
  \frac{1}{A_Q^{\star}}  \Eb_Q\frac{X-1}{\{1+t(X-1)\}^2}
  =
  -\frac{A_Q'(t)}{A_Q^\star}
  \le 0,
\end{equation*}
where we again use
Lemma~\ref{lemma:properties-affinity}.\ref{item:inverse-moment-derivative}
to evaluate $A_Q'(t)$.
Thus \(P_Q^\star\in \nullbasic\) satisfies the equality
conditions in the Cauchy-Schwarz inequality, so
\begin{equation*}
  H^2(Q,P_Q^\star)
  = 1 - \sqrt{A_Q^\star}.
\end{equation*}
This proves both identities~\eqref{eq:hellinger-radius-dual}
and~\eqref{eq:hellinger-nearest-null} in
Lemma~\ref{lem:hellinger-projection} when \(0 \le t_Q < 1\).

We address the remaining case that \(t_Q=1\).
We would like to take $dP^\star \propto dQ^\star / X^2$, but it is
unclear that this normalizes appropriately or belongs to $\nullbasic$.
Since \(1\) minimizes \(A_Q\) and \(A_Q(0)=1\),
\begin{equation*}
  A_Q^\star = A_Q(1)=\Eb_QX^{-1}\le A_Q(0)=1,
\end{equation*}
so \(Q(X = 0) = 0\).
Let $\epsilon > 0$.
Because \(1\) minimizes \(A_Q\),
\begin{align*}
  0
  & \le A_Q\left(\frac{1}{1 + \epsilon}\right)-A_Q(1)
  =
  \Eb_Q\!\left[
    \frac{1 + \epsilon}{X + \epsilon}-\frac1X
  \right]
  =
  \Eb_Q\!\left[
    \frac{\epsilon(X-1)}{X (X + \epsilon)}
  \right].
\end{align*}
Dividing by $\epsilon > 0$ and rearranging gives
\begin{equation*}
  \Eb_Q\frac{1}{X(X+\epsilon)}
  \le
  \Eb_Q\frac{1}{X+\epsilon}
  \le
  \Eb_QX^{-1}
  =
  A_Q^\star,
\end{equation*}
so taking $\epsilon \downarrow 0$ and applying monotone convergence yields
\begin{equation*}
  \Eb_QX^{-2} \le A_Q^\star = \Eb_QX^{-1} < \infty.
\end{equation*}
Therefore
\begin{equation*}
  P_Q^\star(dx)
  =
  \frac{x^{-2}}{A_Q^\star}\,Q(dx)
  +
  \left(
    1-\frac{\Eb_QX^{-2}}{A_Q^\star}
  \right)\pointmass_0(dx)
\end{equation*}
defines a probability measure with mean
\begin{equation*}
  \Eb_{P_Q^\star}X
  =
  \frac1{A_Q^\star}\Eb_QX^{-1}
  =
  1.
\end{equation*}
Finally, because \(Q(X = 0) = 0\), 
\begin{align*}
  H^2(Q,P_Q^\star)
  =
  1-\int_0^{\infty}
  \sqrt{
    dQ(x)\,
    \frac{x^{-2}}{A_Q^\star}\,dQ(x)
  }
  =
  1-\frac1{\sqrt{A_Q^\star}}\Eb_QX^{-1}
  =
  1-\sqrt{A_Q^\star}.
\end{align*}
This gives the final case of the
projection~\eqref{eq:hellinger-nearest-null}, and
\cref{eq:hellinger-radius-dual} follows.

\subsection{Proof of Theorem~\ref{thm:np-hellinger-scale}}
\label{sec:proof-np-hellinger-scale}

The first assertion~\ref{item:no-detection} is essentially
standard.
Let $P_n = \argmin_{P \in \nullbasic} H^2(Q_n, P)$ be the
projection from Lemma~\ref{lem:hellinger-projection},
so that $H^2(Q_n, \nullbasic) = H^2(Q_n, P_n)$.
Then by tensorization,
\begin{align*}
  H^2\left(Q_n^n, \nullIid^n\right)
  = H^2\bigl(Q_n^n,P_n^n\bigr)
  =
  1-\{1-H^2(Q_n,\nullbasic)\}^n.
\end{align*}
Recall the standard relationship between
total variation and Hellinger distance that
\begin{equation*}
  \tvnorm{P - Q}
  \defeq \sup_A|P(A) - Q(A)|
  \le
  \sqrt{1-\{1-H^2(Q,P)\}^2}
\end{equation*}
and the inequality $\Eb_Q[f] \le \Eb_P[f] + \tvnorm{P - Q}$ for
any $0 \le f \le 1$
\citep[see, e.g.,][]{LeCamYa00}.
Then because \(0\le\varphi_n\le1\),
\begin{align*}
  \Eb_{Q_n^n}[\varphi_n]
  \le
  \Eb_{P_n^n}[\varphi_n]
  +
  \tvnorm{Q_n^n - P_n^n}
  \le
  \alpha+
  \sqrt{
    1-\{1-H^2(Q_n,\nullbasic)\}^{2n}
  }.
\end{align*}
Since \(1-(1-u)^{2n}\le2nu,\) as \( n H^2(Q_n, \nullbasic) \to 0\), \(
\limsup_{n \to \infty} \Eb_{Q_n^n}[\varphi_n] \le \alpha\).

To demonstrate part~\ref{item:asymptotic-detection},
observe for any $t$ independent of the sample
$\Xb$ that
\begin{align*}
  \sup_{P \in \nullbasic}
  \Eb_{P^n}\bigg[\prod_{i=1}^n\{1+t(X_i-1)\}\bigg]
  =
  \sup_{P \in \nullbasic} \left(\Eb_P\left[1 + t (X-1)\right]\right)^n
  \le 1.
\end{align*}
Markov's inequality therefore implies that \(\Eb_{P^n}[\lrtest_n]
\le\alpha\) for any $P \in \nullbasic$, guaranteeing
that $\lrtest_n$ has uniform level $\alpha$ on the null $H_0$.
To obtain the power guarantee, recall $t_n = \argmin_{t \in [0,1]}
A_{Q_n}(t)$ and let \(K_n:=\prod_{i=1}^n\{1+t_n(X_i-1)\}.\)
We first observe that if $X_i \simiid Q_n$, then \(K_n>0\) almost surely.
When \(t_n<1\), this is immediate, and if \(t_n=1\),
\begin{equation*}
  \Eb_{Q_n}X^{-1}
  =
  A_{Q_n}(1)
  \le
  A_{Q_n}(0)
  =
  1
\end{equation*}
by Lemma~\ref{lem:hellinger-projection}, so \(Q_n(X = 0) = 0\).
Applying Markov's inequality to \(K_n^{-1}\),
\begin{align*}
  1-\Eb_{Q_n^n}[\lrtest_n]
  =
  Q_n^n\left\{K_n<\frac1\alpha\right\}
  =
  Q_n^n\{K_n^{-1}>\alpha\}
  \le
  \frac{1}{\alpha}\Eb_{Q_n^n}\left[ K_n^{-1}\right],
\end{align*}
while the definition $A_Q(t) = \Eb_Q[(1 + t(X - 1))^{-1}]$ and
Lemma~\ref{lem:hellinger-projection} yield
\begin{align*}
  \frac{1}{\alpha}\Eb_{Q_n^n}\left[ K_n^{-1}\right]
  =
  \frac1\alpha
  \left\{
    \Eb_{Q_n}\left[\frac{1}{1+t_n(X-1)}\right]
  \right\}^n
  =
  \frac{A_{Q_n}(t_n)^n}{\alpha}
  =
  \frac{\{1-H^2(Q_n,\nullbasic)\}^{2n}}{\alpha}.
\end{align*}
This yields part~\ref{item:asymptotic-detection}, as the final power
claim immediately follows from the inequality
\begin{equation*}
  \{1-H^2(Q_n,\nullbasic)\}^{2n}
  \le
  \exp\{-2nH^2(Q_n,\nullbasic)\}.
\end{equation*}

\subsection{Proof of Proposition~\ref{proposition:binplus-nplr-comparison}}
\label{sec:binplus-nplr-comparison}

We first record the following elementary inequality.

\begin{lemma}
\label{lem:binplus-exponential-ratio}
Let \(c\ge1\) and \(u\in\Rb\).
Then
\begin{equation*}
  \frac{e^u-1}{e^{cu}-1}
  \le
  \frac{1}{\sqrt{c e^{(c-1)u}}},
\end{equation*}
where at $u=0$, we interpret the ratio by its continuous extension $1/c$.
\end{lemma}

\begin{proof}
The LHS can be rewritten as
\begin{equation*}
  \frac{e^u-1}{e^{cu}-1}
  =
  e^{-\frac{1}{2}(c-1)u}
  \frac{\sinh(u/2)}{\sinh(cu/2)}
  .
\end{equation*}
Since \(\sinh\) is convex on \([0,\infty)\) and \(\sinh(0)=0\), we have
\( \sinh(ca)\ge c\sinh(a)\) for \(a\ge0\).
Thus for \( a \ge 0 \),
\begin{equation*}
  \frac{\sinh(a)}{\sinh(ca)} \le \frac{1}{c} .
\end{equation*}
Since \(\sinh\) is odd, the same inequality holds for \(a\le0\).
Hence

\begin{equation*}
  \frac{e^u-1}{e^{cu}-1}
  \le
  \frac{e^{-(c-1)u/2}}{c}
  \le
  \frac{1}{\sqrt{c e^{(c-1)u}}}.
\end{equation*}
The bound extends to \(u=0\) by continuity.
\end{proof}

\begin{proof}[Proof of \Cref{proposition:binplus-nplr-comparison}]
  If \(x_{\max}\le1\), then \(\pBinPlus(\bm{x})=\nplr(\bm{x})=1,\) so there is
  nothing to prove.
  Hence suppose that \(x_{\max}>1\), and let \(t_\star\)
  maximize \(\prodPoly_{\bm{x}}(t)\) over \(t\in[0,1]\).
Since
\(\prodPoly_{\bm{x}}(0)=1\), the maximum is at least one.
If it equals one, then \(\nplr(\bm{x})=1\), and the claim follows
from \(\pBinPlus(\bm{x})\le1\).
We may therefore assume \(t_\star>0\).

Choose \(k\) with \(x_k=x_{\max}\).
On
\(\left(-(x_k-1)^{-1},1\right)\),
\(\log\prodPoly_{\bm{x}_{-k}}(s)\) is concave.
Consider its secant function
from \(s=0\) to \(s=t_\star\):
\begin{equation*}
  L(s)
  :=
  \frac{s}{t_\star}
  \log\prodPoly_{\bm{x}_{-k}}(t_\star).
\end{equation*}
Concavity therefore gives
\(\log\prodPoly_{\bm{x}_{-k}}(s)\le L(s)\) for
\(s\in\left(-(x_k-1)^{-1},0\right]\), and
\(\log\prodPoly_{\bm{x}_{-k}}(s)\ge L(s)\) for
\(s\in[0,t_\star]\).
Hence,
\begin{equation*}
  \int_0^{t_\star}\prodPoly_{\bm{x}_{-k}}(s)\,ds
  \ge
  \int_0^{t_\star}e^{L(s)}\,ds
  \qquad\text{and}\qquad
  \int_{\frac{-1}{x_k-1}}^0\prodPoly_{\bm{x}_{-k}}(s)\,ds
  \le
  \int_{\frac{-1}{x_k-1}}^0e^{L(s)}\,ds.
\end{equation*}
\Cref{eq:def-Pvee} implies that
\begin{multline*}
  \pBinPlus(\bm{x})
  =
  \left\{
    1+
    \frac{\int_0^1\prodPoly_{\bm{x}_{-k}}(s)\,ds}
    {\int_{-1/(x_k-1)}^0\prodPoly_{\bm{x}_{-k}}(s)\,ds}
  \right\}^{-1}
  \le
  \left\{
    1+
    \frac{\int_0^{t_\star}e^{L(s)}\,ds}
    {\int_{-1/(x_k-1)}^0e^{L(s)}\,ds}
  \right\}^{-1}
  \\
  =
  \frac{\int_{-1/(x_k-1)}^0e^{L(s)}\,ds}
  {\int_{-1/(x_k-1)}^{t_\star}e^{L(s)}\,ds}
  =
  \frac{
    \exp \left( \frac{L(t_{\star})}{t_{\star}(x_k-1)} \right) - 1
  }{
    \exp \left( \frac{(1+t_{\star}(x_k-1))}{t_{\star}(x_k-1)} \cdot L(t_{\star}) \right) - 1
  }
  .
\end{multline*}
Moreover
\begin{equation*}
  \nplr(\bm{x})
  =
  \prodPoly_{\bm{x}}(t_\star)^{-1}
  =
  \frac{1}{\{1+t_\star(x_k-1)\}e^{L(t_\star)}}
  .
\end{equation*}
Applying \Cref{lem:binplus-exponential-ratio}
with \(u=\frac{L(t_\star)}{t_\star(x_k-1)}\) and \(c=1+t_\star(x_k-1)\) gives
\begin{equation*}
  \pBinPlus(\bm{x})
  \le
  \sqrt{\nplr(\bm{x})}.
  \qedhere
\end{equation*}
\end{proof}

\section{Regular alternatives and asymptotic power}

\subsection{Derivation of Example~\ref{example:dqm-mean}:
  influence functions for the mean}
\label{sec:derivative-of-mean}

The tail condition implies that \(X\in L^2(Q_\theta)\) for every
\(|\theta|<\varepsilon\). In particular, \(\Eb_\theta[X]\) is finite,
and Cauchy-Schwarz gives \(\Eb_0|X \score_0|<\infty\).
For every bounded measurable $h$, quadratic
mean differentiability~\eqref{eq:dqm-assumption} yields
\begin{align*}
  \frac{\Eb_\theta h-\Eb_0h}{\theta}
  & =
  \int h
  \frac{\sqrt{dQ_\theta}-\sqrt{dQ_0}}{\theta}
  \left(\sqrt{dQ_\theta}+\sqrt{dQ_0}\right) \\
  & = \int h \frac{\score_0}{2} \sqrt{dQ_0}
  \left(2 + \frac{\theta}{2} \score_0\right) \sqrt{dQ_0}
  + o(1)
  =
  \Eb_0[h \score_0] + o(1).
\end{align*}
Applying this expansion to \(X\one\{X\le M\} \in [0, M]\) gives
\begin{align*}
  \lim_{\theta \to 0} \frac{
    \Eb_\theta[X\one\{X\le M\}]
    -
    \Eb_0[X\one\{X\le M\}]
  }{\theta}
  =
  \Eb_0[X \score_0 \one\{X\le M\}].
\end{align*}
For the tail, Cauchy--Schwarz gives
\begin{multline*}
  \left|
    \frac{
      \Eb_\theta[X\one\{X>M\}]
      -
      \Eb_0[X\one\{X>M\}]
    }{\theta}
  \right|
  \\
  \le
  \left\{
    \int
    \left(
      \frac{\sqrt{dQ_\theta}-\sqrt{dQ_0}}{\theta}
    \right)^2
  \right\}^{1/2}
  \left(
    \int X^2\one\{X>M\}
    \left(\sqrt{dQ_\theta}+\sqrt{dQ_0}\right)^2
  \right)^{1/2}.
\end{multline*}
Quadratic mean differentiability~\eqref{eq:dqm-assumption}
  implies that the first factor is \(O(1)\) as $\theta \to 0$,
while
\begin{equation*}
  \left(\sqrt{dQ_\theta}+\sqrt{dQ_0}\right)^2
  \le
  2(dQ_\theta+dQ_0),
\end{equation*}
implies that
\begin{equation*}
  \int X^2\one\{X>M\}
  \left(\sqrt{dQ_\theta}+\sqrt{dQ_0}\right)^2
  \le
  2\Eb_\theta\!\left[X^2\one\{X>M\}\right]
  +
  2\Eb_0\!\left[X^2\one\{X>M\}\right].
\end{equation*}
By the uniform integrability assumption in Example~\ref{example:dqm-mean},
we can choose $M$ large enough that each of these terms is negligible,
so
\begin{align*}
  \lim_{M\to\infty}
  \limsup_{\theta\to0}
  \left|
  \frac{
    \Eb_\theta[X\one\{X>M\}]
    -
    \Eb_0[X\one\{X>M\}]
  }{\theta}
  \right|
  =
  0.
\end{align*}
Finally, that $\sup_{|\theta| < \varepsilon} \Eb_\theta[X^2 \one\{X > M\}] \to 0$
as $M \to \infty$ implies that
\(X\one\{X\le M\}\to X\) in \(L^2(Q_0)\), so Cauchy-Schwarz gives
\begin{align*}
  \lim_{M \to \infty} \Eb_0[X \score_0 \one\{X\le M\}]
  =
  \Eb_0[X \score_0].
\end{align*}
Combining each of the preceding limits demonstrates the example.

\subsection{Proof of Corollary~\ref{corollary:nonparametric-envelope}}
\label{sec:proof-nonparametric-envelope}

Define the mean functional \(\Psi(P):=\Eb_PX-1.\)
Thus the null and alternative hypotheses are \(\Psi(P)\le0\) and
\(\Psi(P)>0\), respectively.
The tangent space of the nonparametric model
(see, e.g., Section 25.3 and especially
Example 25.16 of \citet{VanDerVaart98}) at \(P_0\) is
\begin{equation*}
  \mathcal T
  :=
  \left\{
    g\in L^2(P_0):\Eb_0g=0
  \right\},
\end{equation*}
as each bounded \(g\in\mathcal T\) is the score of the DQM submodel
$dP_{t,g}=(1+tg) dP_0$ for small \(\lvert t\rvert\), and
bounded elements are dense in \(\mathcal T\).

The functional $\Psi$
has pathwise (Gateaux) derivative
\begin{equation*}
  \dot\Psi_{P_0}(g)
  =
  \lim_{t\to0} \frac{\Psi(P_{t,g})-\Psi(P_0)}{t}
  =
  \Eb_0[Xg]
  =
  \Eb_0[(X-1)g]
\end{equation*}
for \(g\in\mathcal T\),
so $\Psi$ has efficient influence function
\(\widetilde\psi(x):= x-1\)  at \(P_0\)
with $L^2$ norm \(\|\widetilde\psi\|_{L^2(P_0)}=\sigma\).
Along given parametric submodel with score $\score_0$,
\begin{equation*}
  \langle\widetilde\psi, \score_0\rangle_{P_0}
  := \Eb_0[\widetilde{\psi} \, \score_0]
  =
  \dot\mu.
\end{equation*}
The one-sided semiparametric local power bound
\citep[Theorem~25.44]{VanDerVaart98} gives
\begin{equation*}
  \limsup_{n\to\infty}
  \Eb_{Q_{h/\sqrt n}^n}[\varphi_n]
  \le
  1-\Phi\left(
    z_{1-\alpha}
    -
    h\frac{\langle\widetilde\psi, \score_0\rangle_{P_0}}
            {\|\widetilde\psi\|_{L^2(P_0)}}
  \right)
  =
  1-\Phi\left(
    z_{1-\alpha}-\frac{h\dot\mu}{\sigma}
  \right).
  \qedhere
\end{equation*}

\subsection{Regular alternatives with finite variance}
\label{sec:regularly-varying-np-local}

\begin{proposition}[Finite-variance local expansion
    of squared Hellinger distance]
  \label{proposition:regularly-varying-np-local}
  Let $(Q_\theta)$ be a DQM~\eqref{eq:dqm-assumption} family
  with derivative~\eqref{eq:mean-derivative},
  and assume the local integrability
  condition in Example~\ref{example:dqm-mean}.
  Then, as $\theta\downarrow0$,
  \begin{equation*}
    H^2(Q_\theta,\nullbasic)
    =
    \frac{\dot\mu^2\theta^2}{8\sigma^2}
    +o(\theta^2).
  \end{equation*}
\end{proposition}
\begin{proof}
  The dual representation~\eqref{eq:hellinger-radius-dual} implies
  \begin{equation*}
    H^2(Q_\theta,\nullbasic) = 1-\sqrt{A_{Q_\theta}^\star}
    = 1 - \sqrt{\inf_{0\le t\le1}A_{Q_\theta}(t)}.
  \end{equation*}
  For shorthand, write \( \delta_{\theta} := \Eb_{\theta}[X-1]\).
  We approximate $A_{Q_\theta}$ in a neighborhood of $t = 0$ by a quadratic,
  which allows us to develop the closed form.
  Adding and subtracting $t \delta_\theta$
  gives
  \begin{equation*}
    A_{Q_\theta}(t) = \Eb_\theta\frac{1}{1+t(X-1)}
    = 1 - \delta_{\theta} t + t^2\Eb_{\theta}\frac{(X-1)^2}{1+t(X-1)}.
  \end{equation*}
  We will apply Lemma~\ref{lem:concave-local-maximization} for arbitrary
  $\theta_n \downarrow 0$ by taking \(r_n:=\delta_{\theta_n}\), \(S_n:=1\),
  \(V_n:=2\sigma^2\), and
  \begin{equation}
    f_n(t):=\frac{1-A_{Q_{\theta_n}}(t)}{\delta_{\theta_n}^2}
    = \frac{t}{r_n}
    - \frac{t^2}{r_n^2}
    \Eb_{\theta_n}\frac{(X - 1)^2}{1 + t (X - 1)}
    ~~~ \mbox{and} ~~~
    g_n(u) \defeq u - u^2 \sigma^2.
    \label{eq:finite-variance-A-expansion}
  \end{equation}

  We verify its conditions.
  By assumption~\eqref{eq:mean-derivative},
  \(\delta_\theta=\dot\mu\,\theta+o(\theta),\) so \(r_n = \dot{\mu} \theta_n
  + o(\theta_n) >0\) for large $n$ and \(r_n\to0\).
  Since \(A_{Q_{\theta_n}}\) is closed convex, \(f_n\) is
  closed concave.
  It remains to verify the uniformity
  condition $\sup_{0 \le u \le C} |f_n(r_n u) - g_n(u)| \to 0$
  for all finite $C$.
  Representation~\eqref{eq:finite-variance-A-expansion} gives
  \begin{equation*}
    f_n(r_nu)
    =
    u-u^2
    \Eb_{\theta_n}
    \frac{(X-1)^2}{1+r_nu(X-1)}.
  \end{equation*}
  Thus, it suffices to show that
  \begin{equation}
    \lim_{n \to \infty} \sup_{0\le u\le C}
    \left|
      \Eb_{\theta_n}
      \frac{(X-1)^2}{1+r_nu(X-1)}
      -\sigma^2
    \right|
    = 0.
    \label{eq:finite-variance-uniform-curvature}
  \end{equation}
  Assume $n$ is large enough
  that $C r_n < \half$.
  Then
  \begin{equation*}
    \left|
      \frac{(x-1)^2}{1+r_nu(x-1)}
      -(x-1)^2
    \right|
    =
    \frac{r_nu|x-1|^3}{1+r_nu(x-1)}.
  \end{equation*}
  On \(\{|x-1|\le M\}\), the denominator is at least \(1-Cr_n\ge1/2\), so the
  right-hand side is at most \(2Cr_nM^3\).
  When \(M > 2\), the event \(\{|x-1| > M\}\) implies $x > 1$ so
  $(x-1)^2 / (1 + r_n u (x - 1)) \le (x - 1)^2$.
  Taking expectations therefore gives
  \begin{equation*}
    \left|
      \Eb_{\theta_n}
      \frac{(X-1)^2}{1+r_nu(X-1)}
      -
      \Eb_{\theta_n}(X-1)^2
    \right|
    \le
    2Cr_nM^3
    +
    \Eb_{\theta_n}\!\left[
      (X-1)^2\one\{|X-1|>M\}
    \right]
  \end{equation*}
  for $0 \le u \le C$ and $M > 2$.
  Take $n \to \infty$ and then $M \to \infty$, using
  the integrability condition in Example~\ref{example:dqm-mean}
  to see that
  \begin{equation*}
    \lim_{M\to\infty}\limsup_{n\to\infty}
    \Eb_{\theta_n}[(X - 1)^2 \one\{X - 1 > M\}] = 0.
  \end{equation*}
  Moreover, DQM implies $Q_{\theta_n}\to Q_0$ in total variation, which together
  with the same uniform integrability gives
  $\Eb_{\theta_n}(X-1)^2\to\Eb_0(X-1)^2=\sigma^2$.
  Hence \eqref{eq:finite-variance-uniform-curvature} follows.

  \Cref{lem:concave-local-maximization} therefore applies when
  \(\theta_n \downarrow 0\).
  Thus
  \begin{equation*}
    1-A_{Q_\theta}^\star
    =
    \delta_\theta^2
    \sup_{u\ge0}\{u-\sigma^2u^2\}
    +o(\delta_\theta^2)
    =
    \frac{\delta_\theta^2}{4\sigma^2}
    +o(\delta_\theta^2).
  \end{equation*}
  Since \(1-\sqrt{1-z}=z/2+o(z)\) as \(z\to0\),
  \begin{equation*}
    H^2(Q_\theta,\nullbasic)
    =
    1-\sqrt{A_{Q_\theta}^\star}
    =
    \frac12\{1-A_{Q_\theta}^\star\}
    +o(\delta_\theta^2)
    =
    \frac{\dot\mu^2\theta^2}{8\sigma^2}
    +o(\theta^2).
    \qedhere
  \end{equation*}
\end{proof}

\subsection{Uniform Pareto-type alternatives}\label{sec:heavy-tail-np-local}

The following lemma provides a uniform analogue of the Abelian theorem for
Mellin convolutions; compare \citet[Theorem 4.1.6, p.~201]{BinghamGoTe87}.
We include a proof
because the standard theorem is pointwise, whereas we require uniformity over
the family \( \{ F_{\delta} \} \).

\begin{lemma}
  \label{lem:uniform-tail-integration}
  Let \(\rho \in (1, 2)\).
  For \(0<\delta\le\delta_0\), let \(C_\delta\) satisfy
  \(0<c_-\le C_\delta\le c_+<\infty\).
  Assume the functions
  \(G_\delta:\Rb_+\to[0,1]\) satisfy the tail condition~\eqref{eqn:pareto-tail},
  i.e.,
  \begin{align*}
    \lim_{x\to\infty}
    \sup_{0<\delta\le\delta_0}
    \left|
    \frac{G_\delta(x)}{C_\delta x^{-\rho}}
    - 1
    \right|
    = 0,
  \end{align*}
  and let \(r : \Rb_+ \to \Rb\) satisfy
  \begin{equation*}
    \int_1^\infty
    |r(u)|u^{-\rho}\,du
    <\infty
    ~~~ \mbox{and} ~~~
    \sup_{0 < u \le 1} \frac{|r(u)|}{u} \le K \label{eq:r-tail}
  \end{equation*}
  for some finite $K$.
  Then
  \begin{align*}
    \lim_{t\downarrow0}
    \sup_{0<\delta\le\delta_0}
    \left|
    \frac{1}{C_\delta t^\rho}
    \int_0^\infty
    r(u)G_\delta(u/t)\,du
    -
    \int_0^\infty r(u)u^{-\rho}\,du
    \right|
    =0.
  \end{align*}
\end{lemma}

\begin{proof}
  Let \(\varepsilon>0\).
  By assumption, there exists \(x_0\ge1\) such that for all \(x\ge x_0\) and
  \(0<\delta\le\delta_0\),
  \begin{align}
    \left|
    \frac{G_\delta(x)}{C_\delta x^{-\rho}}
    -1
    \right|
    \le\varepsilon
    \label{eq:uniform-tail-epsilon-bound}
    .
  \end{align}
  For \(0 < t\le x_0^{-1}\), split the integrals at \(tx_0 \le 1\).
  We claim that as \(t\downarrow0\),
  \begin{subequations}
    \begin{equation}
    \sup_{0<\delta\le\delta_0}
    \left|
    \frac{1}{C_\delta t^\rho}
    \int_0^{tx_0} r(u)G_\delta(u/t)\,du
    -
    \int_0^{tx_0} r(u)u^{-\rho}\,du
    \right|
    =
    O(t^{2-\rho}),
    \label{eq:uniform-tail-small-part-claim}
  \end{equation}
  and
  \begin{equation}
    \sup_{0<\delta\le\delta_0}
    \left|
    \frac{1}{C_\delta t^\rho}
    \int_{tx_0}^\infty r(u)G_\delta(u/t)\,du
    -
    \int_{tx_0}^\infty r(u)u^{-\rho}\,du
    \right|
    \le
    \varepsilon
    \int_0^\infty |r(u)|u^{-\rho}\,du.
    \label{eq:uniform-tail-large-part-claim}
  \end{equation}
  \end{subequations}
  To prove \eqref{eq:uniform-tail-small-part-claim}, use
  that \(G_\delta\le1\), \(C_\delta\ge c_-\), and
  \(|r(u)|\le Ku\) on \(0 < u \le 1\) to obtain
  \begin{align*}
    \sup_{0<\delta\le\delta_0}
    \frac{1}{C_\delta t^\rho}
    \left|
    \int_0^{tx_0}
    r(u)G_\delta(u/t)\,du
    \right|
    &\le
    \frac{K}{c_-t^\rho}
    \int_0^{tx_0}u\,du
    =
    \frac{Kx_0^2}{2c_-}t^{2-\rho}
    .
  \end{align*}
  Similarly, 
  \begin{equation*}
    \int_0^{tx_0}|r(u)|u^{-\rho}\,du
    \le
    K\int_0^{tx_0}u^{1-\rho}\,du
    =
    \frac{K}{2-\rho}(tx_0)^{2-\rho}
    .
  \end{equation*}
  \Cref{eq:uniform-tail-small-part-claim} follows from the triangle inequality.

  To prove \eqref{eq:uniform-tail-large-part-claim}, note that for \(u\ge
  tx_0\), we have \(u/t\ge x_0\), so the tail
  bound~\eqref{eq:uniform-tail-epsilon-bound} gives
  $|\frac{G_\delta(u/t)}{C_\delta t^\rho} - u^{-\rho}| \le \varepsilon
  u^{-\rho}$.
  Multiplying by \(|r(u)|\) and integrating yields
  claim~\eqref{eq:uniform-tail-large-part-claim}.
  The assumptions of the lemma trivially make
  the right-hand integral~\eqref{eq:uniform-tail-large-part-claim} finite.
  Combine \eqref{eq:uniform-tail-small-part-claim} and
  \eqref{eq:uniform-tail-large-part-claim} and take
  \(\varepsilon\downarrow0\).
\end{proof}

\begin{proposition}
  \label{proposition:heavy-tail-np-local}
  Let \((Q_\theta)_{0<\theta\le\theta_0}\) be probability laws on \(\Rb_+\)
  satisfying \(\Eb_{Q_\theta} X = 1 + \theta\)
  and the tail condition~\eqref{eqn:pareto-tail} for some
  $1 < \rho < 2$.
  Then as $\theta \to 0$,
  \begin{align*}
    H^2(Q_\theta,\nullbasic)
    =
    \frac{\rho-1}{2\rho}
    \left(
    -\frac{\rho^2\pi C_\theta}{\sin(\pi\rho)}
    \right)^{-\frac{1}{\rho-1}}
    \cdot \theta^{\frac{\rho}{\rho-1}}
    \cdot (1 + o(1)).
  \end{align*}
\end{proposition}
\begin{proof}
  As in the proof of
  Proposition~\ref{proposition:regularly-varying-np-local}, we
  begin with the representation~\eqref{eq:hellinger-radius-dual}:
  \begin{equation*}
    H^2(Q_\theta,\nullbasic)
    =
    1-\sqrt{A_{Q_\theta}^\star}
    =
    1-\sqrt{\inf_{0\le t\le1}A_{Q_\theta}(t)}.
  \end{equation*}
  We claim that the following uniform approximation holds:
  \begin{align}
    \lim_{t\downarrow0}
    \sup_{0<\theta\le\theta_0}
    \left|
    \frac{
      A_{Q_\theta}(t)-1+\theta t
    }{
      C_\theta t^\rho
    }
    +
    \frac{\rho\pi}{\sin(\pi\rho)}
    \right|
    = 0.
    \label{eq:heavy-tail-inverse-curvature}
  \end{align}
  Deferring the proof of the limit~\eqref{eq:heavy-tail-inverse-curvature},
  we now use it to evaluate $A^\star_{Q_\theta}$, though
  we can no longer leverage the quadratic approximation
  in Lemma~\ref{lem:concave-local-maximization}.
  Rescaling and multiplying by $C_\theta t^\rho$,
  the limit~\eqref{eq:heavy-tail-inverse-curvature}
  implies that for each finite $M$,
  \begin{equation*}
    \lim_{\theta \downarrow 0} \sup_{0\le s\le M}
    \left|
    \frac{
      A_{Q_\theta}
      \left(\theta^{\frac{1}{\rho-1}}s\right)-1
    }{
      \theta^{\frac{\rho}{\rho-1}}
    }
    -
    \phi_\theta(s) \right|
    = 0
    ~~ \mbox{for} ~~
    \phi_\theta (s)
    :=
    -s-\frac{\rho\pi C_\theta}{\sin(\pi\rho)}s^\rho.
  \end{equation*}

  Because \(1 < \rho < 2\) and \(\sin(\pi\rho)<0\), $\phi_\theta$ is convex and coercive
  on $\Rb_+$ as $C_\theta \ge c_- > 0$, and has $\phi_\theta'(0) < 0$, so we may choose
  \(M\) large enough that the global minimum of \(\phi_\theta\) is attained in
  \([0,M]\) for all small $\theta > 0$.
  Consequently, for
  sufficiently small \(\theta\),
  \begin{equation*}
    A_{Q_\theta}
    \left(M\theta^{\frac{1}{\rho-1}}\right)
    >
    A_{Q_\theta}(0)
    =
    1.
  \end{equation*}
  Since \(A_{Q_\theta}\) is convex, it attains its global minimum 
  in \([0,M\theta^{1/(\rho-1)}]\), so  
  \begin{align*}
    \frac{
      A_{Q_\theta}^\star-1
    }{
      \theta^{\frac{\rho}{\rho-1}}
    }
    =
    \min_{0\le s\le M}
    \frac{
      A_{Q_\theta}
      \left(\theta^{\frac{1}{\rho-1}}s\right)-1
    }{
      \theta^{\frac{\rho}{\rho-1}}
    }
    =
    \inf_{s\ge0}\phi_\theta(s) + o_\theta(1).
  \end{align*}
  Differentiating \(\phi_\theta\) shows that
  \begin{align*}
    \argmin_{s \ge 0} \phi_\theta(s)
    &=
    \left(
    -\frac{\rho^2\pi C_\theta}{\sin(\pi\rho)}
    \right)^{-\frac{1}{\rho-1}},
    ~~ \mbox{whence} ~~
    \inf_{s \ge 0} \phi_\theta(s) = -\frac{\rho-1}{\rho}
    \left(
    -\frac{\rho^2\pi C_\theta}{\sin(\pi\rho)}
    \right)^{-\frac{1}{\rho-1}}.
  \end{align*}
  Thus
  \begin{align*}
    1-A_{Q_\theta}^\star
    =
    \frac{\rho-1}{\rho}
    \left(
    -\frac{\rho^2\pi C_\theta}{\sin(\pi\rho)}
    \right)^{-\frac{1}{\rho-1}}
    \theta^{\frac{\rho}{\rho-1}}
    \cdot (1 + o(1))
  \end{align*}
  as $\theta \downarrow 0$.
  In particular, \(A_{Q_\theta}^\star\to1\), and $1 - \sqrt{z} =
  \frac{1-z}{2} + O((1-z)^2)$ as $z\to1$ gives
  the limit~\eqref{eq:regular-varying-hellinger}:
  \begin{equation*}
    H^2(Q_\theta, \nullbasic)
    =
    1-\sqrt{A_{Q_\theta}^\star}
    =
    \half(1 - A_{Q_\theta}^\star) \cdot (1 + o(1)).
  \end{equation*}

  We return to demonstrate the local uniform
  convergence~\eqref{eq:heavy-tail-inverse-curvature}.
  Adding and subtracting $\theta t$ gives
  \begin{equation*}
    A_{Q_\theta}(t)
    =
    \Eb_{Q_\theta}
    \left[
      \frac{1}{1+t(X-1)}
      \right]
    =
    1-\theta t
    +
    t^2
    \Eb_{Q_\theta}
    \left[
      \frac{(X-1)^2}{1+t(X-1)}
      \right].
  \end{equation*}
  The limit~\eqref{eq:heavy-tail-inverse-curvature} is thus equivalent
  to showing that
  \begin{align*}
    \lim_{t\downarrow0}
    \sup_{0<\theta\le\theta_0}
    \left|
    \frac{t^{2-\rho}}{C_\theta}
    \Eb_{Q_\theta}
    \left[
      \frac{(X-1)^2}{1+t(X-1)}
      \right]
    +
    \frac{\rho\pi}{\sin(\pi\rho)}
    \right|
    =0.
  \end{align*}
  We split
  \begin{multline*}
    t^2
    \Eb_{Q_\theta}
    \left[
      \frac{(X-1)^2}{1+t(X-1)}
      \right]
    =
    \Eb_{Q_\theta}
    \left[
      \frac{\{t(X-1)\}^2}{1+t(X-1)}
      \one\{X<1\}
      \right]
    +
    \Eb_{Q_\theta}
    \left[
      \frac{\{t(X-1)\}^2}{1+t(X-1)}
      \one\{X>1\}
      \right].
  \end{multline*}
  and bound the two terms separately.
  First consider the term on \(X < 1\).
  For \(t\le1/2\), uniformly in \(\theta\), we have
  \begin{align*}
    0
    &\le
    \Eb_{Q_\theta}
    \left[
      \frac{\{t(X-1)\}^2}{1+t(X-1)}
      \one\{X<1\}
      \right]
    \le
    \frac{t^2}{1-t}
    =
    O(t^2)
    =
    o(t^\rho).
  \end{align*}

  To bound the term involving \(\{X>1\}\), let
  \(f(u):=u^2/(1+u)\) for \(u\ge0\) and
  $G_\theta(x) = Q_\theta(X > x) = 1 - Q_\theta(X \le x)$ be the complementary
  CDF.
  Then using the change of variables
  $u = t (x - 1)$,
  \begin{align*}
    \Eb_{Q_\theta}
    \left[
      \frac{\{t(X-1)\}^2}{1+t(X-1)}
      \one\{X>1\}
      \right]
    & =
    \Eb_{Q_\theta}
    \left[
      f\{t(X-1)\}\one\{X>1\}
      \right]
    \\
    & =
    -\int_0^\infty
    f(u)\,d\left\{
    G_\theta\left(1+\frac{u}{t}\right)
    \right\}.
  \end{align*}
  For fixed \(t>0\), the tail condition~\eqref{eqn:pareto-tail}
  implies that \(G_\theta\left(1+\frac{u}{t}\right) =
  O(u^{-\rho})\)  as \(u\to\infty\), uniformly in $\theta$ near 0.
  Since \(f(u) \le u\) and \(\rho>1\), \(f(u)G_\theta\left(1+\frac{u}{t}\right)
  = O(u^{1-\rho}) \to 0.\)
  Because \(f(0)=0\),
  \begin{equation*}
    \Eb_{Q_\theta}
    \left[
      \frac{\{t(X-1)\}^2}{1+t(X-1)}
      \one\{X>1\}
      \right]
    =
    \int_0^\infty
    f'(u)G_\theta\left(1+\frac{u}{t}\right)\,du
  \end{equation*}
  by integration by parts.
  We apply \Cref{lem:uniform-tail-integration} to
  \(x \mapsto G_\theta(x+1)\) with \(r=f'\).
  The tail condition~\eqref{eqn:pareto-tail}, together with
  \((x/(x+1))^\rho\to1\), implies that
  \begin{align*}
    \lim_{x\to\infty}
    \sup_{0<\theta\le\theta_0}
    \left|
    \frac{G_\theta(x+1)}
         {C_\theta x^{-\rho}}
         -1
         \right|
         =0.
  \end{align*}
  Moreover,
  \(f'(u)=1-(1+u)^{-2}=O(u)\) as \(u\downarrow0\), and
  \(f'(u)=O(1)\) as \(u\to\infty\).
  So \(\int_1^\infty |f'(u)|u^{-\rho}\,du<\infty\),
  and \Cref{lem:uniform-tail-integration} gives
  \begin{align*}
    \lim_{t\downarrow0}
    \sup_{0<\theta\le\theta_0}
    \left|
    \frac{1}{C_\theta t^\rho}
    \int_0^\infty
    f'(u)G_\theta\left(1+\frac{u}{t}\right)\,du
    -
    \int_0^\infty f'(u)u^{-\rho}\,du
    \right|
    =0.
  \end{align*}

  We thus evaluate the integral
  $\int f'(u) u^{-\rho} d\rho$.
  As \(f(u)u^{-\rho}=u^{2-\rho}/(1 + u)\to0\) as \(u \to \{0, \infty\}\),
  a second integration by parts yields
  \begin{equation*}
    \int_0^\infty f'(u)u^{-\rho}\,du
    =
    \rho
    \int_0^\infty
    f(u)u^{-\rho-1}\,du
    =
    \rho
    \int_0^\infty
    \frac{u^{1-\rho}}{1+u}\,du
    =
    -\frac{\rho\pi}{\sin(\pi\rho)}.
  \end{equation*}
  Combining the bounds for the \(\{X<1\}\) and \(\{X>1\}\) terms proves
  \eqref{eq:heavy-tail-inverse-curvature}.
\end{proof}

\end{document}